\documentclass[reqno,twoside]{amsart}
\usepackage{amsfonts}
\usepackage{amsmath,amssymb,amsthm,amsthm}
\usepackage{mathtools}
\usepackage{etoolbox}
\usepackage{color}
\usepackage[dvipsnames,svgnames,table]{xcolor}
\usepackage{graphicx}
\usepackage{subcaption}
\usepackage{float}
\usepackage{placeins}
\usepackage{comment}
\usepackage[dvips,bottom=1.4in,right=1in,top=1in, left=1in]{geometry}
\usepackage{eurosym}
\usepackage{amsfonts}
\usepackage{amsmath,amssymb,lineno,amsthm,epic,amsthm}
\usepackage{natbib}
\def\R{\mathbb R}

\def\N{\mathbb N}

\def\D{\displaystyle}

\newtheorem{theorem}{Theorem}[section]
\newtheorem{proposition}[theorem]{Proposition}
\newtheorem{remark}[theorem]{Remark}
\newtheorem{lemma}[theorem]{Lemma}
\newtheorem{corollary}[theorem]{Corollary} 
\newtheorem{definition}[theorem]{Definition}

\numberwithin{equation}{section}
\newcommand\emb{\stackrel{\mathclap{\normalfont d}}{\hookrightarrow}}
\newcommand\embc{\stackrel{\mathclap{\normalfont c}}{\hookrightarrow}}
\usepackage{thmbox}
\usepackage{dsfont}
\usepackage{float}
\usepackage{appendix}
\usepackage{enumitem}

\begin{document}
\title[Optimal control via chemotactic sensitivity]{Optimal control via chemotactic sensitivity of a logistic Keller-Segel model}

\author[Yacouba Simporé]{Yacouba Simporé\textsuperscript{1,2}}
\author[Mahamadi Warma]{Mahamadi Warma\textsuperscript{3}}
\author[Luis P. Yapu]{Luis P. Yapu\textsuperscript{1,4}}

\thanks{\textsuperscript{1}Chair for Dynamics, Control, Machine Learning and Numerics, Department of Mathematics, Friedrich-Alexander-Universit¨at
Erlangen–Nürnberg, Cauerstrasse 11, 91058 Erlangen, Germany.}
\thanks{\textsuperscript{2}Université Yembila Abdoulaye TOGUYENI, BP 54 Fada N'Gourma, Burkina Faso. }
\thanks{\textsuperscript{3}Department of Mathematical Sciences and the Center for Mathematics and Artificial Intelligence, George Mason University, Fairfax, VA 22030, USA}
\thanks{\textsuperscript{4}Universidade Federal Fluminense, Instituto de Matemática e Estatistica, Campus Gragoatá, rua Professor Marcos Waldemar de Freitas Reis, s/n - São Domingos, Niterói - RJ - Brazil}


\date{}
\begin{abstract}
We study an optimal control problem for a two-equations Keller-Segel system on a bounded domain with Neumann boundary conditions, where the chemotactic sensitivity is modulated by a control $f(x,t)$ acting on the taxis flux $\nabla\!\cdot(f\,u\,\nabla v)$. 
We prove the well-posedness of the state equation under low regularity assumptions.
The analysis is based on the study of an $\varepsilon$–regularized problem, the application of Schauder’s fixed point theorem, $\varepsilon$–independent,  positivity and a priori estimates, the passage to the limit as $\varepsilon \to 0$, and a uniqueness argument. 
For a cost functional combining trajectory tracking, taxis penalization, control regularization, and terminal objectives, we establish the existence of optimal controls and derive the first-order optimality conditions through a Lagrange multiplier framework, leading to an adjoint system and a variational inequality for the optimal control. 
A numerical scheme based on spline collocation in space and Runge-Kutta time-stepping is set and implemented.
Numerical experiments are performed
to simulate the free and controlled dynamics in conservative and non-conservative cases.
\end{abstract}

\maketitle

\section{Introduction, problem formulation and main results}
 \label{sec:introduction}
 In this section we first describe the model we would like to study by giving the mathematical formulation of the problem and the objectives of our study. Secondly, we state the main results obtained in the paper.
 Next, we give the novelty, a brief motivation and background on the topic of the paper, then we introduce the real life phenomena where the model can be applied. 
 
\subsection{Mathematical formulation of the problem}
In this work, we focus on the control of a Keller-Segel type model for cell population dynamics. The system consists of two coupled nonlinear parabolic equations in a bounded domain \(\Omega \subset \mathbb{R}^2\) with boundary $\partial\Omega$, and with Neumann boundary conditions.  Throughout the rest of the paper, given a time $T>0$, we let $\Omega_T:=\Omega\times (0,T)$ and $\Gamma_T:=\partial\Omega\times (0,T)$.
We consider the controlled Keller-Segel system
\begin{equation}\label{1}
    \begin{cases}
        \D\partial_t u - D_u \Delta u + \nabla \cdot (f u \nabla v) = r u - \mu u^2 & \mbox{ in  }\Omega_T, \\
        \D\partial_t v - D_v \Delta v + \alpha v = \beta u & \mbox{ in } \Omega_T, \\
        \partial_\nu u = \partial_\nu v = 0 & \mbox{ on } \Gamma_T, \\
        \D u(\cdot,0) = u_0, \; v(\cdot,0) = v_0 &  \mbox{ in } \Omega,
    \end{cases}
\end{equation}
where
\begin{itemize}
    \item \( u \) denotes the cell density, \( v \) the chemical concentration;
    \item \( f\) is the control variable modulating the chemotactic sensitivity;
    \item $r,\mu>0$ are real numbers,
    \item the logistic term \( ru - \mu u^2 \) models growth and saturation;
    \item $\partial_\nu u$ and $\partial_\nu v $ denote the normal derivative of the functions $u$ and $v$, respectively;
    \item the homogeneous Neumann boundary conditions ensure the mass conservation if the logistic term is not present.
\end{itemize}
We notice that in \eqref{1}, $D_u$, $D_v$ are universal constants that do not depend on the functions $u$ and $v$.

We analyze an optimal control problem where the control \( f \) acts locally on the taxis flux
\[
    \nabla \cdot (f u \nabla v),
\]
with the aim of steering the cell population distribution towards a desired state while minimizing the control effort.  More precisely, we consider the cost functional
\begin{equation}
    \begin{aligned}\label{cout}
        J(u,v,f) := &\gamma_u \int_0^T \| u(t) - u_d(t) \|^2_{L^2(\Omega)} dt 
        + \gamma_{\mathrm{taxis}}\int_0^T \| f u \nabla v)\|^2_{L^2(\Omega)}dt\\ 
        &+ \gamma_f \int_0^T \| f(t) \|^4_{L^4(\Omega)} dt 
        + \gamma_T \| u(T) - u_T \|^2_{L^2(\Omega)}.
    \end{aligned}
\end{equation}
Here also, $\gamma_u$ and $\gamma_f$ are positive constants that are independent of $u$ and $f$.

The functional \( J \) balances the following objectives. Tracking a desired spatiotemporal profile \( u_d \in L^2(\Omega_T)\) of the cell density over the time horizon;
regulating the intensity of the taxis through an explicit penalization of the control flux magnitude $\int_0^T\|  f u \nabla v  \|_{L^2(\Omega)}dt$; 
minimizing the energy or magnitude of the control \( f \)  and achieving the desired terminal state \( u_T\in L^2(\Omega) \) at time \( T \).

Before we give the formulation of the control problem, we introduce the spaces needed in the formulation.
Let $\xi_1,\xi_2\in L^\infty(\Omega_T)$ satisfy
\[
    \xi_1(x,t)\leq \xi_2(x,t) \quad\text{for a.e. }(x,t)\in\Omega_T.
\]
We define the set of admissible controls by
\begin{equation}
    \label{eq:def_F}
    \mathcal{F}:=\left\{f\in L^4(\Omega_T)\;:\;\xi_1(x,t)\leq f(x,t)\leq \xi_2(x,t)\quad\text{for a.e. }(x,t)\in\Omega_T\right\}.
\end{equation}
Next, let
\begin{equation}\label{def-W4}
    W_4:=\bigg\{u \in L^{\infty}(0, T; L^4(\Omega))\cap L^2(0,T;H^1(\Omega)):\; \partial_t u \in L^{2}(0,T;(H^{1}(\Omega))') \bigg\},
\end{equation}
and
\begin{equation}\label{X4}
    X_4:=\bigg\{v \in C([0, T]; W_{\mathbf{n}}^{\frac{3}{2},4}(\Omega)) \cap L^4(0, T; W^{2,4}(\Omega)):\;  \partial_t v \in L^{4}(\Omega_T)\bigg\}.
\end{equation}
The precise definition of the Sobolev spaces involved will be given in Section \ref{Func-Spa}. Let also denote
\begin{equation}\label{Sad}
    \mathcal S_{\rm ad}:=\left\{(u,v,f)\in W_4\times X_4\times \mathcal F:\; (u,v) \mbox{ is a solutions of } \eqref{1} \mbox{ with control } f\right\}.
\end{equation}
Our notion of solutions to the system \eqref{1} will be given in Section \ref{wesk-sol}.
 
Our control problem can then be formulated as follows:
\begin{equation}\label{cont-prob}
    \min_{(u,v,f)\in \mathcal S_{\rm ad}}J(u,v,f),
\end{equation}
where the functional $J$ is given by \eqref{cout}.

The main contributions of the present work are the following.
\begin{itemize}
    \item We analyze the well-posedness of the controlled Keller-Segel system under weak regularity assumptions, which are particularly relevant when considering singular or spatially localized controls.
    \item We establish the existence of an optimal control that minimizes the cost functional $J$.
    \item We derive the corresponding first-order necessary optimality conditions, involving an adjoint system and a variational inequality for the optimal control.
    \item We highlight the mathematical challenges associated with the control of the nonlinear taxis terms and provide a rigorous framework for their analysis.
    \item We set a numerical algorithm for the resolution of the optimal control problem. The strategy combines collocation methods using B-splines and integration on time using Runge-Kutta formulae.
\end{itemize}

The results obtained contribute to understanding optimal control mechanisms in chemotaxis models and may serve as a theoretical basis for future applications in biological regulation and targeted interventions.

\subsection{Statement of the main results}\label{wesk-sol}
In this section, we state the main results obtained in the paper.
We start by introducing the notion of weak solutions to the system \eqref{1} and show their existence, uniqueness, and regularity. We are interested in solutions which are bounded with respect to the control $f$. 

Throughout the rest of the article,
$\Omega\subset\mathbb R^2$ is a bounded domain with a smooth boundary $\partial\Omega$. 
Recall that the control domain $\mathcal F$ is given by \eqref{eq:def_F}. 

\begin{definition} \label{def:weak-solution}
    Let $f \in \mathcal{F}$, $u_0 \in L^4(\Omega)$, $v_0 \in W_{\textbf{n}}^{\frac{3}{2},4}(\Omega),$
    with $u_0 \geq 0$ and $v_0 \geq 0$ a.e. in $\Omega$. A pair $(u,v)$ is said to be a \emph{weak solution} of the system \eqref{1}, if the following assertions hold:
    \begin{itemize}
        \item $u \in W_4$ and $v \in X_4$,
        \item $u \geq 0$, $v \geq 0$ a.e. in $\Omega_T$,
        \item the equation for $u$ in \eqref{1} and the boundary condition hold in a variational sense, 
        \item the equation for $v$ and the boundary condition hold pointwise, and,
        \item the initial conditions for $u$ and $v$ hold in the $L^4(\Omega)$ and $W_{\bf n}^{3/2,4}(\Omega)$ sense, respectively.
    \end{itemize}
    
    More precisely, for all $\varphi\in L^4(0,T;W^{1,4}(\Omega))$, we have that
    \begin{align}\label{W1}
        \!\int_0^T\!\!\langle \partial_t u,\varphi\rangle_{(H^{1}(\Omega))',H^1(\Omega)}\,dt
        &+ D_u\!\int_{\Omega_T} \nabla u\!\cdot\!\nabla\varphi\;dtdx
        - \int_{\Omega_T} fu\,\nabla v\!\cdot\!\nabla\varphi\;dtdx\notag\\
        =& \int_{\Omega_T} (r-\mu u)\,u\,\varphi\;dtdx,
    \end{align}
    $u(\cdot,0)=u_0$ a.e. in $\Omega$, while $v$ solves the system 
    \begin{equation}\label{W2}
        \begin{cases}
            \partial_t v - D_v\Delta v + \alpha v=\beta u\quad &\text{ in }\Omega_T,\\
            \partial_\nu v=0 &\mbox{ on }\Gamma_T,\\
            v(0)=v_0 &\mbox{ in } \Omega.
        \end{cases}
    \end{equation}
    in the strong sense.
\end{definition}

The following theorem is the first main result of the paper. Its proof will be given in Section \ref{sec:proof_theorem_wellp}.

\begin{theorem}\label{mainresult1}
    Let $T>0$, $u_0 \in L^4(\Omega)$, $v_0 \in W_{\textbf{n}}^{\frac{3}{2},4}(\Omega)$ with $u_0,v_0 \geq 0$ a.e. in $\Omega$, and let  $f \in \mathcal{F}$.
    Then,  the system \eqref{1} admits a unique weak solution $(u,v)\in W_4 \times X_4$ in the sense of Definition \ref{def:weak-solution}.
    Moreover, there is a constant $K>0$ such that
    \[
        \begin{aligned}
            \|u\|_{L^2(0,T;H^{1}(\Omega))} + &\|u\|_{L^{\infty}(0,T;L^4(\Omega))} + \|\partial_t u\|_{L^{2}(0,T;(H^{1}(\Omega))')} \\
            + &\|v\|_{C([0,T];W_{\mathbf{n}}^{\frac{3}{2},4}(\Omega))} + \|v\|_{L^4(0,T;W^{2,4}(\Omega))} + \|\partial_t v\|_{L^4(\Omega_T)} \leq K,
        \end{aligned}
    \]
    where  $K = K(\Omega, T, \|u_0\|_{L^4(\Omega)},\|v_0\|_{W_n^{\frac{3}{2},4}(\Omega)}, \|f\|_{L^{\infty}(\Omega_T)})$ and also depends on the other parameters involved in the system.
\end{theorem}

The following result is the second main result of the article.

\begin{theorem}[\bf Existence of optimal controls] \label{thm:existence}
    Let $u_0 \in L^4(\Omega)$ and  $v_0 \in W_{\textbf{n}}^{\frac{3}{2},4}(\Omega)$ be such that $u_0 \geq 0$ and $v_0 \geq 0$ a.e. in $\Omega$. 
    Then, there exists at least one optimal control $\tilde{f} \in \mathcal{F}$ and corresponding state solution $(\tilde{u}, \tilde{v})$ satisfying Definition \ref{def:weak-solution} such that $(\tilde{u}, \tilde{v}, \tilde{f})$ solves the optimal control problem \eqref{1}-\eqref{cont-prob}.
\end{theorem}

\begin{remark}
    {\em 
        In the system~\eqref{1}, the chemotactic sensitivity \(f\) has been introduced as a scalar function depending on space and time.  
        For a finer control over the chemical gradient, one could instead consider \(f\) as a square matrix-valued control of the form  
        \[
            f(x,t) =
            \begin{pmatrix}
                f_{1}(x,t) & 0 \\
                0 & f_{2}(x,t)
            \end{pmatrix},
        \]
        acting on \(\nabla v\) in the advection term \(\nabla \cdot \big( u\,f\,\nabla v \big)\).  
        From an analytical standpoint, the existence and uniqueness of optimal controls in this matrix-valued setting can be established using the same functional framework as in the scalar case studied in this paper, by treating \(f\) as an element of \((L^4(\Omega_T))^{2\times 2}\) with suitable box constraints on each entry.
        
        Intuitively, this generalization allows direction-dependent modulation of the chemotactic flux, enabling more precise steering of the population density \(u\) in both the magnitude and direction of its response to \(\nabla v\).  
        Such flexibility could be advantageous in applications where anisotropic or cross-directional chemotactic effects are relevant.
    }
\end{remark}

\subsection{Novelty, motivation and background}

Chemotaxis, the directed movement of biological cells in response to chemical gradients, is a central process in many biological and medical phenomena such as embryogenesis, wound healing, immune response, and tumor invasion. The mathematical framework most widely used to describe these processes is the Keller–Segel (KS) system and its numerous extensions. These models capture complex spatiotemporal behaviors such as aggregation, blow-up, pattern formation, and dispersion. 
From a control-theoretic viewpoint, chemotaxis offers a rich context for investigating nonlinear dynamics driven by coupled reaction–diffusion processes. Recent years have seen growing interest in the control and optimization of chemotaxis-type models for biological regulation or targeted therapies.

The standard Keller–Segel system couples the cell density \(u(x,t)\) and the concentration of a chemical substance \(v(x,t)\) and is given by
\begin{equation}\label{ks}
    \begin{cases}
        \partial_t u - D_u \Delta u + \nabla \cdot (u \chi \nabla v) = 0\;\;&\mbox{ in }\Omega_T,\\
        \partial_t v - D_v \Delta v + \alpha v = \beta u &\mbox{ in }\Omega_T,
    \end{cases}
\end{equation}
with homogeneous Neumann boundary conditions complemented with initial conditions. The term \(\chi\) denotes the chemotactic sensitivity. The mathematical analysis of these equations has been widely developed in the last three decades. Without logistic damping, solutions may exhibit finite-time blow-up in dimension two or higher. 
For instance, Winkler \cite{Winkler2010} and Liu–Lorz \cite{LiuLorz2010} established the existence of weak and classical solutions in \(2D\) under subcritical conditions on the initial data, while blow-up can occur for large masses.  

The present work introduces a novel control paradigm. Instead of modifying the external chemical gradients, we act directly on the \emph{chemotactic sensitivity} of the cells through a distributed control \(f\). 

This formulation is consistent with the Keller–Segel framework, where the chemotactic sensitivity coefficient \( \chi \) in \eqref{ks} governs the strength of cell response to chemical cues. 
As highlighted in the comprehensive review by Horstmann~\cite{horstmann2003from}, the functional form of \( \chi \) plays a decisive role in determining aggregation, pattern formation, and stability in the chemotaxis models. 

Our control \( f \) can thus be interpreted as a space–time dependent modulation of this sensitivity, providing a biologically meaningful mechanism to regulate directed cell migration. 
In particular, the controlled chemotactic drift may appear in different biologically relevant forms. 
A typical example is
\(
\nabla \cdot \!\left( f\, u \, \nabla v \right),
\)
representing a direct modulation of the chemotactic response amplitude.
However, in some biologically inspired models, the chemotactic sensitivity saturates at high signal concentrations, leading to a rational-type drift of the form $\displaystyle
\nabla \cdot \!\left( \frac{f}{\gamma + v} \, u \, \nabla v \right)$,
where $\gamma > 0$ represents a saturation threshold (see, e.g., \cite{ARUMUGAM2020103090, LankeitWinkler2017}).

Nonlinear or rational sensitivity laws have been analyzed in the framework of global existence and boundedness of weak solutions for Keller–Segel systems with saturation effects.
These formulations remain consistent with experimental observations of receptor saturation in chemotactic signaling pathways (see, e.g., \cite{Teicher2010, OcalSahin2012, karp2008mesenchymal, Mezzapelle2022}, 
and they preserve the Keller–Segel structure while incorporating biologically realistic response regulation.
In particular, Teicher \& Fricker~\cite{Teicher2010} showed that the CXCL12/CXCR4 signaling axis is a major regulator of cell chemotaxis, survival, and proliferation. 
This pathway modulates the chemotactic sensitivity of cancer cells and plays a crucial role in tumor invasion, metastasis, and persistence, highlighting its importance as a therapeutic target.  

Based on these biological insights, we formulated in \eqref{1} a controlled Keller–Segel-type system~ \cite{ARUMUGAM2020103090,horstmann2003from} in which the control function \(f\) modulates the chemotactic sensitivity either linearly or through a saturating mechanism.
The addition of the logistic source term \(r u - \mu u^2\) in \eqref{1} plays a crucial regularizing role by preventing blow-up and guaranteeing the global-in-time existence of bounded weak solutions. 

Lankeit \cite{Lankeit2016} analyzed global existence and asymptotic stability for the logistic Keller–Segel model in two dimensions. In three dimensions, further extensions were developed by Tello \& Winkler \cite{TelloWinkler2012}.
These results confirm that the logistic term stabilizes the population density and allow the derivation of global attractors.

\subsection{Existing results on controlled Keller–Segel (KS) systems}
Although the analysis of uncontrolled KS models is now mature, the optimal control of such nonlinear chemotaxis systems remains relatively recent. The control can act linearly or bilinearly, and on different components (cell equation or chemical equation), leading to distinct mathematical difficulties.

Linear control problems for KS systems have been investigated in low-dimensional cases. Ryu \& Yagi \cite{RyuYagi1999} considered a linear distributed control acting on a chemical equation in 2D chemo-attraction models, proving the existence of optimal controls and deriving first-order conditions.
López-Ríos \& Villamizar-Roa \cite{LopezRios2021} analyzed a 3D chemotaxis–Navier–Stokes system with distributed linear controls acting on the chemical and velocity equations, showing the existence of weak and strong solutions as well as an optimal control and its first-order necessary optimality conditions.

Bilinear control problems, where the control multiplies a state variable (for example \(f v\) or \(f u\)), are more realistic for biological interpretation but analytically much more delicate due to the nonlinear couplings.
The work of Braz et al. \cite{optimalbilinear_KellerSegel} uses a bilinear control on the chemical concentration $v$. In that work, the system evolves in two spatial dimensions, the reaction term is also given by a logistic equation and they show the existence of optimal controls and compute the optimality system using Lagrange multipliers.
Their previous work \cite{2D_ChemoRepulsion_FGG} studies the bilinear control associated to a 2D chemo-repulsion model when the logistic term is not present.
Let us remark that in \cite{2D_ChemoRepulsion_FGG}, the bilinear control $f$ belongs to $L^4(\Omega_T)$. The 3D case is treated in \cite{3D_ChemoRepulsion_FGG}. 

In contrast, in the present work the control acts directly on the chemotactic 
sensitivity in the cell-density equation. Thus, the control is coupled with 
the chemotactic drift $u\nabla v$, whereas in \cite{optimalbilinear_KellerSegel,3D_ChemoRepulsion_FGG} it is 
coupled with the chemical variable $v$ in the signal equation. Consequently, the 
two models involve different controlled nonlinearities and require distinct analytical 
frameworks.

A related numerical work on optimal control of the KS model using collocation methods with splines is presented in \cite{Weiwei}. Some basic references on collocation methods using splines are \cite{siam_multivariate-22,Lai-07}.

In the chemo-repulsion case, Guillén-González et al. \cite{Guillen2018} considered a 3D bilinear optimal control problem
with linear production. They first established the existence of weak solutions for the uncontrolled problem, introduced a regularity criterion ensuring the passage to strong solutions (needed for differentiability of the control-to-state mapping), and then derived a first-order optimality system via a Lagrange multiplier argument in Banach spaces.
More recently, Braz e Silva et al. \cite{optimalbilinear_KellerSegel} analyzed a bilinear optimal control problem for the 2D KS model with a logistic term where the control acts multiplicatively in the chemical equation. The authors in \cite{optimalbilinear_KellerSegel} established the existence of a unique weak solution for every \(f \in L^{2+\alpha}(\Omega_T)\),
for $\alpha>0$ small enough, and proved the existence of an optimal control minimizing a quadratic cost functional, and derived the associated first-order optimality system.
Bilinear control structures have been further examined in chemo-repulsion models featuring nonlinear production terms of the form $u^p$ $(1<p\le2)$. In particular, Mallea-Zepeda \& Medina \cite{Mallea2023} studied an optimal control problem for a 2D parabolic-elliptic system combining nonlinear chemical production and bilinear control effects. Nonetheless, the interplay between logistic damping, sensitivity modulation, and weakly regularized dynamics has yet to be addressed in full generality.

Beyond optimal control, controllability results for chemotaxis systems have also been explored.
Okposo \& Willie \cite{Okposo} established well-posedness, finite-time blow-up bounds, and null or approximate controllability for the classical KS model using Carleman estimates.
In a related work, Guo \& Zhang \cite{GUO2014106} proved the local null controllability for a parabolic-elliptic chemotaxis system via a linearization approach.

Despite these advances in the field, most existing control formulations act indirectly on the chemical signal or through external sources. 
In contrast, in our model we propose to regulate the cell migration process more fundamentally by modulating the chemotactic sensitivity itself.

\subsection{Structure of the paper}
The rest of the paper is structured as follows.
In Section \ref{Func-Spa} we recall the functional setup and notations used throughout the paper. 
In Section \ref{sec:proof_theorem_wellp} we give the proof of  Theorem \ref{mainresult1}. Concerning the optimal control problem, in Section \ref{sec:existence_OC} we give the proof of Theorem \ref{thm:existence}, and in Section \ref{sec:lagrange_multipliers} we describe briefly the Lagrange multiplier method and use it as a tool to determine the necessary first-order optimality system. Moreover we describe a numerical algorithm in order to find the optimal control and the evolution of the controlled system. In Section \ref{sec:numerical_implementation} we describe the numerical implementation using spatial discretization given by a collocation method using B-splines and a Runge-Kutta method to numerically integrate the time evolution. The numerical examples are presented in Section \ref{sec:numerical_examples} and at the end of the paper we give four appendices containing technical proofs of some results that are crucial in the proofs of the main results.

\section{Notations, functional setup and some known-results}\label{Func-Spa}
We fix some notations, introduce the function spaces and recall some well-known results that are used throughout the paper. 

Let $X$ and $Y$ be two Banach spaces. If the injection of $X$ to $Y$ is continuous, 
we use the notation $X\hookrightarrow Y$. If the continuous injection is dense, then we denote $X\emb Y$. In addition, if the injection is compact we shall denote $X\embc Y$. We use the symbol $\to$ to indicate strong convergence, $\rightharpoonup$ for weak convergence and $\overset{\ast}{\rightharpoonup}$ for weak-$*$ convergence.

Next, let $\Omega\subset\mathbb R^N$ ($N\ge 1$) be a bounded open set. For every $m\in\mathbb N$ and $1\le p\le \infty$, we denote by $W^{m,p}(\Omega)$ the classical Sobolev spaces. If $m=1$ and $p=2$ we set $H^1(\Omega):=W^{1,2}(\Omega)$ which is a Hilbert space. We also denote by $(H^{1}(\Omega))'$ the dual space of $H^1(\Omega)$ with respect to the pivot space $L^2(\Omega)$ so that we have the continuous and dense injections
$$
    H^1(\Omega)\emb L^2(\Omega)\emb (H^{1}(\Omega))'.
$$

Let $0<s<1$ and $1\le p<\infty$ be real numbers. We define the fractional order Sobolev space
$$
    W^{s,p}(\Omega):=\left\{u\in L^p(\Omega):\; \int_{\Omega}\int_{\Omega}\frac{|u(x)-u(y)|^p}{|x-y|^{N+sp}}\;dxdy<\infty\right\},
$$
that we endow with the norm
\begin{equation}\label{norm}
    \|u\|_{W^{s,p}(\Omega)}:=\left(\int_{\Omega}|u|^p\;dx+ \int_{\Omega}\int_{\Omega}\frac{|u(x)-u(y)|^p}{|x-y|^{N+sp}}\;dxdy \right)^{\frac 1p}.
\end{equation}
With this definition, $W^{s,p}(\Omega)$ is a Banach space.
When $s>1$ and is not an integer, we write $s=m+\sigma$, where $m$ is an integer and $0<\sigma<1$. In this case, the space $W^{s,p}(\Omega)$ is defined by
$$W^{s,p}(\Omega):=\left \{u\in W^{m,p}(\Omega):\; D^\alpha u\in W^{\sigma,p}(\Omega) \;\mbox{ for all } \alpha \mbox{ such that } |\alpha|=m\right\},$$
which is a Banach space endowed with the norm
\begin{equation*}
    \|u\|_{W^{s,p}(\Omega)}:=\left(\|u\|_{W^{m,p}(\Omega)}^p+\sum_{|\alpha|=m}\|D^\alpha u\|_{W^{\sigma,p}(\Omega}^p\right)^{\frac 1p}.
\end{equation*}
With this definition, if $s=m$ is an integer, then $W^{s,p}(\Omega)$ coincides with $W^{m,p}(\Omega)$. 
We shall also denote by $W^{-s,q}(\Omega)$ the dual space of $W^{s,p}(\Omega)$ where $\frac 1p+\frac 1q=1$.
Moreover, if $\Omega$ is bounded and has a Lipschitz continuous boundary and $s_1 > s_2$, we have the compact embedding $W^{s_1,p}(\Omega) \embc W^{s_2,p}(\Omega)$.

For $\Omega\subset\mathbb R^N$ a bounded smooth open set, $s>1$ and $1\le p<\infty$, we use the notation 
$$
    W_{\mathbf{n}}^{s,p}(\Omega)) := \{ u \in W^{s,p}(\Omega): \; \partial_\nu u:=\nabla u\cdot\vec{n} = 0 \text{ on } \partial\Omega \}.
$$ 
For more details on classical and fractional order Sobolev spaces, we refer the interested reader to Di Nezza, Palatucci \& Valdinoci \cite{NPV}, Grisvard \cite{Gris}, Warma \cite{War-P} and their references.

Next, for $\alpha>0$ small enough, we use the notation
\begin{equation*}
    L^{p+}(\Omega):= L^{p+\alpha}(\Omega),\qquad  W^{s+,p+}(\Omega):=W^{s+\alpha,p+\alpha}(\Omega),
\end{equation*}
and
\begin{equation*}
    L^{p-}(\Omega):= L^{p-\alpha}(\Omega),\qquad  W^{s-,p-}(\Omega):=W^{s-\alpha,p-\alpha}(\Omega).
\end{equation*}

We recall some results that are used throughout the paper. We formulate the results for a smooth bounded open set $\Omega\subset\mathbb R^2$.
\begin{itemize}
    \item The two-dimensional Ladyzhenskaya/Gagliardo--Nirenberg inequality
    (see e.g. \cite{ladyzhenskaya1969}) is given by the following: there is a constant $C>0$ such that for every
    \(u\in H^1(\Omega)\),
    \[
        \|u\|_{L^4(\Omega)}
        \leq
        C\|u\|_{L^2(\Omega)}^{1/2}
        \|u\|_{H^1(\Omega)}^{1/2}.
    \]
    
    \item The 2D-Sobolev embedding \cite[Section 5.6]{Evans-98}. The following continuous embeddings hold.
    \[
        \begin{cases}
            W^{1,p}(\Omega)\hookrightarrow L^q(\Omega),\quad 1\leq q\leq \frac{2p}{2-p} &\mbox{ if } 1\leq p<2,\\
            H^{1}(\Omega)\hookrightarrow L^q(\Omega), \qquad  1\leq q<\infty &.
        \end{cases}
    \]
    \item We shall also use the fractional critical
    embedding (see e.g. \cite[Theorem~6.10]{NPV}),
    \[
        W^{\frac 12,4}(\Omega)\hookrightarrow L^q(\Omega)
        \qquad \text{for every } 1\leq q<\infty.
    \]
\end{itemize}

The following result known as the Aubin-Lions Lemma \cite[Lemma 7.7]{roubicek2013nonlinear} (see also \cite{simon1986compact}) will be frequently used. 

\begin{lemma}\label{lem-21}
    Let $X_0$, $X$ and $X_1$ be Banach spaces with $X_0 \subset X \subset X_1$. Suppose that $X_0\embc X$ and that $X\hookrightarrow X_1$. For $1 \leq p,q \leq \infty$, let
    $$
        W := \{ u \in L^p(0,T;X_0) \: : \: \partial_t u \in L^q(0,T;X_1) \}.
    $$
    If $1\le p < \infty$, then we have the compact embedding $W\embc L^p(0,T;X)$. If $p=\infty$ and $q>1$ then we also have the compact embedding $W\embc C([0,T];X)$.
\end{lemma}
 
Next, we recall known regularity results on linear parabolic equations which will be crucial in the proofs of our main results. Consider the linear system
\begin{equation}\label{eq:heat_neumann}
    \begin{cases}
        \partial_t u - \Delta u = f & \text{ in } \Omega_T, \\
        \partial_\nu u = 0 & \text{ on } \Gamma_T, \\
        u(\cdot,0) = u_0 & \text{ in } \Omega.
    \end{cases}
\end{equation}
The following result taken from \cite{feireisl2009singular} provides the optimal ($L^p-L^q$)-regularity for solutions of the heat equation \eqref{eq:heat_neumann}.

\begin{theorem}[\bf Maximal ($L^p - L^q$)-regularity]
    Let \(\Omega \subset \mathbb{R}^N\) ($N\ge 1$) be a bounded domain of class \(C^2\), and \(1 < p, q < \infty\). Suppose that
    $f \in L^p(0, T; L^q(\Omega))$, and $u_0 \in X_{p,q}$,
    where \(X_{p,q}\) is defined via a real interpolation (see e.g. \cite[page 344]{feireisl2009singular}) as
    \[
        X_{p,q} := \bigg\{ L^q(\Omega); D(\Delta_N) \bigg\}_{1-\frac 1p,p},
        \text{
        with } D(\Delta_N) = \bigg\{ v \in W^{2,q}(\Omega):\; \nabla v \cdot \mathbf{n}|_{\partial \Omega} = 0 \bigg\}.
    \]
    Then, the problem \eqref{eq:heat_neumann} admits a unique weak solution \(u\) satisfying the following:
    \[
        u \in L^p(0, T; W^{2,q}(\Omega))\cap C([0, T]; X_{p,q}), \quad \partial_t u \in L^p(0, T; L^p(\Omega)).
    \]
    Moreover, there is a constant \(C = C(p, q, \Omega, T) > 0\) such that for a.e. \(t \in [0, T]\),
    \[
        \|u(\cdot,t)\|_{X_{p,q}} + \|\partial_t u\|_{L^p(0,T; L^q(\Omega))} + \|\Delta u\|_{L^p(0,T; L^q(\Omega))} \leq C \left( \|f\|_{L^p(0,T; L^q(\Omega))} + \|u_0\|_{X_{p,q}} \right).
    \]
    \label{regular}
\end{theorem}


\section{Proof of the main Theorem \ref{mainresult1} }\label{sec:proof_theorem_wellp}

In order to give of the proof, we need some preparations. 

\subsection{Nonlinear $\varepsilon$–regularized system and associated auxiliary problems}
We introduce a regularized version of the original nonlinear system.  
For every fixed parameter $\varepsilon>0$, the regularization is performed through the truncation operator
\[
    a_\varepsilon(s):=\frac{s}{1+\varepsilon|s|},
\]
which is a Lipschitz continuous function that satisfies $|a_\varepsilon(s)|\le \min\{|s|,\varepsilon^{-1}\}$, and preserves positivity in the sense that $a_\varepsilon(s)\,s\ge0$.

We fix positive constants $D_u,D_v,\alpha,\beta,\mu,r$ and we assume that
\[
    f\in L^\infty(\Omega_T),\quad
    u_0\in L^4(\Omega),\quad
    v_0\in W_{\mathbf{n}}^{\frac{3}{2},4}(\Omega) \mbox{ with } u_0,v_0\ge 0  \mbox{ a.e. in }\Omega.
\]

For each $\varepsilon\in (0,1]$, we introduce the following $\varepsilon$–regularized system:
\begin{equation}\label{reg-eps}
    \begin{cases}
        \partial_t u_\varepsilon - D_u\Delta u_\varepsilon
        + \nabla\!\cdot\!\bigl(f\,a_\varepsilon(u_\varepsilon)\,\nabla v_\varepsilon\bigr)
        = \bigl(r-\mu\,a_\varepsilon(u_\varepsilon)\bigr)\,u_\varepsilon
        &\text{ in } \Omega_T,\\
        \partial_t v_\varepsilon - D_v\Delta v_\varepsilon + \alpha v_\varepsilon
        = \beta\,a_\varepsilon(u_\varepsilon)
        &\text{ in } \Omega_T,\\
        \partial_\nu u_\varepsilon = \partial_\nu v_\varepsilon = 0
        &\text{ on } \Gamma_T\\
        u_\varepsilon(0)=u_0,\; v_\varepsilon(0)=v_0
        &\text{ in }\Omega.
    \end{cases}
\end{equation}
This is the key nonlinear $\varepsilon$–regularized problem that we study the existence and regularity of solutions, and then we pass to the limit as $\varepsilon\to 0$ to obtain the solutions of the unregularized system \eqref{1}.

\subsection{Auxiliary frozen–coefficients problem}
\label{sec:frozen_coeff}
To construct the solutions of the system \eqref{reg-eps}, we first freeze the nonlinearity.
Given a function $\bar u\in L^4(\Omega_T)$, we set
\[
    a_\varepsilon:=a_\varepsilon(\bar u)\in L^\infty(\Omega_T)\cap L^4(\Omega_T)=L^\infty(\Omega_T).
\]
We then consider the following linear auxiliary system:
\begin{equation}\label{aux-frozen}
    \begin{cases}
        \partial_t u - D_u\Delta u
        + \nabla\!\cdot\!\bigl(f\,a_\varepsilon(\bar u)\,\nabla v\bigr)
        = \bigl(r - \mu\,a_\varepsilon(\bar u)\bigr)\,u\qquad 
        &\text{ in } \Omega_T,\\
        \partial_t v - D_v\Delta v + \alpha v = \beta\,a_\varepsilon(\bar u)
        &\text{ in } \Omega_T,\\
        \partial_\nu u=\partial_\nu v=0
        &\text{ on } \Gamma_T,\\
        u(0)=u_0,\quad v(0)=v_0
        &\text{ in }\Omega.
    \end{cases}
\end{equation}
For each fixed $\bar u\in L^4(\Omega_T)$ and $\varepsilon>0$, the system \eqref{aux-frozen} is linear in $(u,v)$,
with bounded coefficients depending only on the frozen function $a_\varepsilon(\bar u)$.

\begin{definition}
    A pair $(u,v)$ is called a solution to \eqref{aux-frozen} in the weak/strong sense if
    \begin{equation*}
        \begin{cases}
            u \in L^{\infty}(0, T; L^4(\Omega)) \cap L^2(0, T; H^1(\Omega)), \\
            v \in C([0, T]; W_{\mathbf{n}}^{\frac{3}{2}, 4}(\Omega)) \cap L^4(0, T; W^{2,4}(\Omega)), \\
            \partial_t u \in L^2(0, T; (H^{1}(\Omega))'), \quad \partial_t v \in L^4(\Omega_T),
        \end{cases}
    \end{equation*}
    and for all $\varphi\in L^4(0,T;W^{1,4}(\Omega)),$
    \begin{align*}
        \!\int_0^T\!\!\langle \partial_t u,\varphi\rangle_{(H^{1}(\Omega))',H^1(\Omega)}\,dt
        &+ D_u\!\int_{\Omega_T} \nabla u\!\cdot\!\nabla\varphi\;dtdx
        - \int_{\Omega_T} f\,a_\varepsilon(\bar u)\,\nabla v\!\cdot\!\nabla\varphi\;dtdx\\
        = &\int_{\Omega_T} (r-\mu a_\varepsilon(\bar u))\,u\,\varphi\;dtdx,
    \end{align*}
    while $v$ satisfies the following equation in the strong sense:
    \begin{equation*}
        \begin{cases}
            \partial_t v - D_v\Delta v + \alpha v=\beta\,a_\varepsilon(\bar u),\quad &\text{ in }\Omega_T,\\
            \partial_\nu v=0, &\mbox{ on }\Gamma_T,\\
            v(0)=v_0 &\mbox{ in }\Omega.
        \end{cases}
    \end{equation*}
\end{definition}

\begin{remark}
    The proof uses several steps.
    We explain our strategy for passing from the auxiliary to the original problem. 
    \begin{enumerate}
        \item[(a)] \textbf{Auxiliary regularized problem}: for every $\varepsilon\in(0,1]$ and each given $\bar u\in L^4(\Omega_T)$,
        we solve the $\varepsilon$–regularized frozen–coefficient system \eqref{aux-frozen}.
        \item[(b)] \textbf{Fixed point argument-positivity-uniqueness}: the map 
        \begin{equation}\label{R}
            \mathcal R:\bar u\mapsto u_\varepsilon
        \end{equation}
        is shown to be continuous and to have relatively compact image in $L^4(\Omega_T)$.  
        By Schauder's fixed point theorem, there exists a fixed point $u_\varepsilon\in L^4(\Omega_T)$ such that $\mathcal R(u_\varepsilon)=u_\varepsilon$.  
        This provides a solution $(u_\varepsilon,v_\varepsilon)$ to the auxiliary system for each $\varepsilon>0$. At this stage the nonnegativity of $(u_\varepsilon,v_\varepsilon)$ is established.
        \item[(c)] \textbf{Limit as $\varepsilon\to0$}:  
        uniform a priori estimates, independent of $\varepsilon$, allow us to extract a subsequence
        $(u_\varepsilon,v_\varepsilon)$ 
        converging to a limit $(u,v)$ as $\varepsilon\to0$.
        This limit satisfies the original unregularized problem \eqref{1}. The uniqueness of the solution is proved only after this passage to the limit.
    \end{enumerate} 
\end{remark}

Our approach builds upon the method proposed by Arumugam et al. \cite{ARUMUGAM2020103090} for other systems.

\subsection{Existence, uniqueness and regularity results for the auxiliary system}
\label{sec:existence_auxiliary}

For the $v$-equation we have the following situation.
Assuming that $v_0 \in W^{\frac{3}{2},4}_{\textbf{n}}(\Omega)$, since 
$a_\varepsilon(\bar u) \in L^\infty(\Omega_T)$, it follows that the $v$-equation admits a unique  
solution. Moreover, using the maximal ($L^4-L^q$)-regularity, we can deduce that
$v\in X_4$, where we recall that $X_4$ is given in \eqref{X4}. This shows that $v$ is a strong solution.

Next, assume that \(u_0 \in L^4(\Omega)\) and \(f \in L^\infty(\Omega_T)\). Since $v \in X_4$ and $\Omega \subset \R^2$, we have that \(\nabla v \in L^q(\Omega_T;\mathbb R^2)\), for any $1 \leq q < \infty$.
The frozen Neumann problem reads as
\begin{equation}\label{FP}
    \begin{cases}
        \partial_t u - D_u\Delta u + \nabla\!\cdot\!\big(f\,a_\varepsilon(\bar u)\,\nabla v\big)
        = \big(r-\mu a_\varepsilon(\bar u)\big)\,u \;&\mbox{ in } \Omega_T,\\
        \partial_\nu u=0 &\mbox{ on }\Gamma_T,\\
        u(\cdot,0)=u_0 &\mbox{ in }\Omega.
    \end{cases}
\end{equation}
We rewrite the first equation in \eqref{FP} as
\[
    \partial_tu-D_u\Delta u+Cu=-\nabla\!\cdot F, \quad F:=f\,a_\varepsilon(\bar u)\nabla v,
    \quad C:=-r+\mu a_\varepsilon(\bar u).
\]
Since \(f,a_\varepsilon(\bar u)\in L^\infty(\Omega_T)\) and
\(\nabla v\in L^4(\Omega_T;\mathbb R^2)\), we have
\[
    F\in L^4(\Omega_T;\mathbb R^2)\hookrightarrow L^2(\Omega_T;\mathbb R^2),
    \quad
    C\in L^\infty(\Omega_T),
\]
and, by continuity of
\(\nabla\!\cdot:L^2(\Omega;\mathbb R^2)\to (H^{1}(\Omega))'\),
\[
    \|\nabla\!\cdot F\|_{L^2(0,T;(H^{1}(\Omega))')}
    \le \|F\|_{L^2(\Omega_T)}.
\]
Since \(u_0\in L^4(\Omega)\hookrightarrow L^2(\Omega)\), standard
variational theory for linear parabolic equations with homogeneous
Neumann boundary conditions yields a unique weak solution satisfying
\[
    u\in L^2(0,T;H^1(\Omega)),\quad \partial_tu\in L^2(0,T;(H^{1}(\Omega))').
\]
In addition, using the variation-of-constants formula associated with the Neumann heat semigroup and the standard $L^4$-smoothing estimates, we obtain that
\begin{equation}
    \label{eq:u_C_L4}
    u\in C([0,T];L^4(\Omega)).    
\end{equation}

Moreover, there is a constant \(C_{T,\varepsilon}>0\) depending on \(T,\varepsilon,\) but not on $\bar u$, such that
\begin{equation}\label{EE}
    \begin{aligned}
        &\|u\|_{L^\infty(0,T;L^4(\Omega))}
        +\|u\|_{L^2(0,T;H^1(\Omega))}
        +\|\partial_t u\|_{L^2(0,T;(H^1(\Omega))')}
        \\
        &\quad\leq C_{T,\varepsilon}\left(\|u_0\|_{L^4(\Omega)}
        +\|F\|_{L^4(\Omega_T)}\right)
        \\
        &\quad\leq C_{T,\varepsilon}
        \left[ \|u_0\|_{L^4(\Omega)} +\frac{\|f\|_{L^\infty(\Omega_T)}}{\varepsilon}\left( \|v_0\|_{W^{3/2,4}_{n}}+\frac{\beta}{\varepsilon}|\Omega_T|^{1/4}\right)\right]
        =:K_{T,\varepsilon}.
    \end{aligned}
\end{equation}

We also record the following local-in-time regularity. Let
\(\tau\in(0,T)\) and choose \(\chi_\tau\in C^\infty([0,T])\) such that
\(\chi_\tau=0\) on \([0,\tau/2]\) and
\(\chi_\tau=1\) on \([\tau,T]\).
Setting \(z_\tau:=\chi_\tau u\), we have that
\[
    \partial_tz_\tau-D_u\Delta z_\tau+Cz_\tau = -\nabla\!\cdot(\chi_\tau F)+\chi_\tau'u, \quad z_\tau(0)=0.
\]
Since $\nabla v\in L^\infty(0,T;L^4(\Omega;\mathbb R^2))$
and \(f,a_\varepsilon(\bar u)\in L^\infty(\Omega_T)\), we have
that $F\in L^\infty(0,T;L^4(\Omega;\mathbb R^2)).$
Moreover, \(u\in C([0,T];L^4(\Omega))\) implies
\(z_\tau=\chi_\tau u\in C([0,T];L^4(\Omega))\). Hence, since
\(C\in L^\infty(\Omega_T)\), for every \(1<q<\infty\),
\[
    G_\tau:= -\nabla\!\cdot(\chi_\tau F)+\chi_\tau'u-Cz_\tau \in L^q(0,T;W^{-1,4}(\Omega)).
\]
Therefore, \(z_\tau\) satisfies
\[
    \partial_tz_\tau-D_u\Delta z_\tau=G_\tau, \quad z_\tau(0)=0.
\]
Let \(A_4\) denote the weak Neumann realization of
\(-D_u\Delta\) in \(W^{-1,4}(\Omega)\). Since \(\Omega\) is smooth,
standard elliptic \(W^{1,4}\)-regularity for the shifted Neumann
problem gives
\[
    A_4+I:W^{1,4}(\Omega)\longrightarrow W^{-1,4}(\Omega)
\]
as a topological isomorphism. Consequently,
\[
    D(A_4)=W^{1,4}(\Omega)
\]
with equivalent norms. Therefore, maximal regularity
\cite[Theorem~5.4 and Remarks~5.2(i), 5.14]{HALLERDINT},
applied with spatial exponent \(4\), together with uniqueness of the
weak solution, yields
\[
    z_\tau\in W^{1,q}(0,T;W^{-1,4}(\Omega))
    \cap L^q(0,T;W^{1,4}(\Omega)),
    \qquad 2\le q<\infty.
\]
The same conclusion for \(1<q<2\) follows from the case \(q=2\),
since \((0,T)\) has finite measure.
Since \(z_\tau=u\) on \([\tau,T]\), it follows that
\begin{equation}\label{local-u-regularity}
    u\in L^q(\tau,T;W^{1,4}(\Omega)), \quad
    \partial_tu\in L^q(\tau,T;W^{-1,4}(\Omega)),
    \quad 1<q<\infty.
\end{equation}
The corresponding estimate may depend on \(\tau\), \(\varepsilon\),
and \(q\).

\subsection{Existence of solutions to the nonlinear $\varepsilon-$regularized system via Schauder fixed point}
\label{subsec:schauder}

Our next aim is to establish an existence result for the $\varepsilon-$regularized system.
Here, to simplify the notations we omit the dependence of the solutions on the parameter \(\varepsilon\).
For each fixed \(\varepsilon>0\), we prove the existence of solutions to the system \eqref{reg-eps}
by applying Schauder's fixed-point theorem.
We introduce a closed and convex subset of the Banach space \(L^4(\Omega_T)\),
which simplifies the notation, and we show that the obtained solution indeed belongs to \(L^4(\Omega_T)\).
Let
\begin{equation}\label{eq:def_S_M4}
    S_{M}^{(4)}:=\Bigl\{\,u\in L^4(\Omega_T):\ \|u\|_{L^4(\Omega_T)}\le M\Bigr\}.    
\end{equation}

This set is nonempty, convex, closed and bounded in the Banach space \(L^4(\Omega_T)\).
By \eqref{EE}, for every $\bar u\in L^4(\Omega_T)$, the corresponding
solution $u=\mathcal R(\bar u)$ satisfies
\[
    \|\mathcal R(\bar u)\|_{L^4(\Omega_T)}
    \leq T^{1/4}
    \|u\|_{L^\infty(0,T;L^4(\Omega))}
    \leq T^{1/4}K_{T,\varepsilon}.
\]
Since $K_{T,\varepsilon}$ is independent of the frozen function
$\bar u$, we may choose
$M:=T^{1/4}K_{T,\varepsilon}.$
It then follows that
$\mathcal R\bigl(S_M^{(4)}\bigr)\subset S_M^{(4)},$
where $\mathcal R$ is the solution map defined in \eqref{R}.

\subsubsection{\bf Compactness of \(\mathcal R\) in \(L^4(\Omega_T)\).}

Let $(u_n)_n$ be an arbitrary sequence in $\mathcal{R}(S_M^{(4)})$. From the existence and regularity results for the auxiliary system and the uniform estimate \eqref{EE}, we have
\begin{align*}
    \sup_{n\in\mathbb{N}}
    \Big(
    \|u_n\|_{L^\infty(0,T;L^4(\Omega))}
    +\|u_n\|_{L^2(0,T;H^1(\Omega))} \quad
    +\|\partial_tu_n\|_{L^2(0,T;(H^1(\Omega))')}
    \Big)
    \leq K_{T,\varepsilon},
\end{align*}
where $K_{T,\varepsilon}$ is independent of $n$.
Since $\Omega\subset\mathbb{R}^2$, we have
\[
    X_0:=H^1(\Omega)
    \embc
    X:=L^4(\Omega)
    \hookrightarrow
    X_1:=(H^1(\Omega))'.
\]
Therefore, Lemma~\ref{lem-21} yields
\[
    \left\{
    w\in L^2(0,T;X_0):
    \partial_tw\in L^2(0,T;X_1)
    \right\}
    \embc
    L^2(0,T;L^4(\Omega)).
\]
Since $(u_n)_n$ is uniformly bounded in the space on the left-hand side, there exist a subsequence, still denoted by $(u_n)_n$, and $u\in L^2(0,T;L^4(\Omega))$ such that
\[
    u_n\longrightarrow u
    \quad\text{strongly in }L^2(0,T;L^4(\Omega)).
\]

We next show that the limit $u$ also belongs to $L^\infty(0,T;L^4(\Omega))$. Therefore, by the Banach--Alaoglu theorem and the metrizability of the weak-star topology on bounded subsets of the dual of a separable Banach space, there exist a further subsequence, still not relabeled, and
$\widetilde{u}\in L^\infty(0,T;L^4(\Omega))$
such that $u_n\rightharpoonup^\ast\widetilde{u}
\quad\text{in }L^\infty(0,T;L^4(\Omega)).$
Moreover, the weak-star lower semicontinuity of the norm yields
\[
    \|\widetilde{u}\|_{L^\infty(0,T;L^4(\Omega))} \leq \liminf_{n\to\infty} \|u_n\|_{L^\infty(0,T;L^4(\Omega))} \leq K_{T,\varepsilon}.
\]
It remains to identify $\widetilde{u}$ with the strong limit $u$. For every $\varphi\in C_c^\infty(\Omega_T)$, the weak-star convergence gives
\[
    \int_{\Omega_T}u_n\varphi\,dx\,dt
    \longrightarrow
    \int_{\Omega_T}\widetilde{u}\varphi\,dx\,dt.
\]
On the other hand, 
\begin{align*}
    \left|
    \int_{\Omega_T}(u_n-u)\varphi\,dx\,dt
    \right|
    &\leq
    \|u_n-u\|_{L^2(0,T;L^4(\Omega))}
    \|\varphi\|_{L^2(0,T;L^{4/3}(\Omega))}\longrightarrow 0.
\end{align*}
Thus, both sequences converge to the same limit in
$\mathcal{D}'(\Omega_T)$. Hence,
$u=\widetilde{u}
\quad\text{a.e. in }\Omega_T.$
Consequently,
\[
    u\in L^\infty(0,T;L^4(\Omega)),
    \quad
    \|u\|_{L^\infty(0,T;L^4(\Omega))}
    \leq K_{T,\varepsilon}.
\]
The triangle inequality then gives
\[
    \|u_n-u\|_{L^\infty(0,T;L^4(\Omega))}
    \leq
    \|u_n\|_{L^\infty(0,T;L^4(\Omega))}
    +
    \|u\|_{L^\infty(0,T;L^4(\Omega))}
    \leq 2K_{T,\varepsilon}.
\]
Finally, by H\"older's inequality in time,
\begin{align*}
    \|u_n-u\|_{L^4(\Omega_T)}^4
    &=
    \int_0^T
    \|u_n(t)-u(t)\|_{L^4(\Omega)}^4\,dt \\
    &\leq
    \|u_n-u\|_{L^\infty(0,T;L^4(\Omega))}^2
    \|u_n-u\|_{L^2(0,T;L^4(\Omega))}^2 \\
    &\leq
    4K_{T,\varepsilon}^2
    \|u_n-u\|_{L^2(0,T;L^4(\Omega))}^2
    \longrightarrow 0.
\end{align*}
Therefore,
$u_n\longrightarrow u
\quad\text{strongly in }L^4(\Omega_T).$
Since every sequence in $\mathcal{R}(S_M^{(4)})$ admits a
subsequence converging strongly in $L^4(\Omega_T)$, the set $\mathcal{R}(S_M^{(4)})$ is relatively compact in
$L^4(\Omega_T)$.

\subsubsection{\bf Continuity of $\mathcal R$ in $L^4(\Omega_T)$}

Let $(\bar u_{n})$ be a sequence in \(S_{M}^{(4)}\)
and let $ \bar u\in S_{M}^{(4)} $ be such that $\bar u_{n}\to\bar u$ in $L^4(\Omega_T)$, as $n\to\infty$.  We define $(u_n,v_n)$ such that \(\mathcal R(\bar u_{n})=u_n\) and $v_n$ the corresponding solution of the second system. The goal is to show that \(u_n=\mathcal R(\bar u_{n})\to \mathcal R(\bar u)=u\) in \(L^4(\Omega_T)\), as $n\to\infty$. This is achieved by using energy estimates and Sobolev embeddings. We refer the reader to Appendix \ref{appendix_a} for the full proof of the continuity of $\mathcal R$.
We have shown that the map \(\mathcal R:S_{M}^{(4)}\to S_{M}^{(4)}\) is continuous with a relatively compact image.
By Schauder’s fixed-point theorem, there exists a fixed point \(u=\mathcal{R}(u)\in S_{M}^{(4)}\).

Although $\mathcal R$ is regarded as a self-map of $S_M^{(4)}$ for
the application of Schauder's fixed-point theorem, the regularity
results obtained for the auxiliary system yield the stronger range
property
\[
    \mathcal R:
    S_M^{(4)}
    \longrightarrow
    W_4\cap C([0,T];L^4(\Omega))\cap S_M^{(4)}
    \hookrightarrow S_M^{(4)}.
\]
Indeed, for every $\bar u\in S_M^{(4)}$, one has $\bar u\in L^4(\Omega_T)$. The regularity results established above
for the corresponding auxiliary problem give
\[
    v\in X_4
    \quad\text{and}\quad
    \mathcal R(\bar u)\in W_4.
\]
Moreover, applying \eqref{eq:u_C_L4} to the auxiliary solution
$\mathcal R(\bar u)$ yields
$\mathcal R(\bar u)\in C([0,T];L^4(\Omega)).$
For this fixed point, the identity
$u=\mathcal R(u)\in\mathcal R\bigl(S_M^{(4)}\bigr)$
and the above range property immediately yield
$u\in W_4\cap C([0,T];L^4(\Omega)).$
Let $v$ be the solution of the linear $v$-equation obtained by replacing the frozen function $\bar u$ by the fixed point $u$. The auxiliary regularity result gives $v\in X_4$.  Moreover, since $u=\mathcal R(u)$ is the auxiliary solution
corresponding to the frozen function $u$, the local regularity result
\eqref{local-u-regularity} gives, for every $\tau\in(0,T)$ and every
$1<q<\infty$,
$u\in L^q(\tau,T;W^{1,4}(\Omega)),
\quad
\partial_tu\in L^q(\tau,T;W^{-1,4}(\Omega)).$
Therefore, by the definition of $\mathcal R$, $(u,v)$ is a weak
solution of the $\varepsilon$-regularized system satisfying
\[
    (u,v)\in W_4\times X_4,
    \quad
    u\in C([0,T];L^4(\Omega)),
\]
and, for every $\tau\in(0,T)$ and $1<q<\infty$,
\[
    u\in L^q(\tau,T;W^{1,4}(\Omega)),
    \quad
    \partial_tu\in L^q(\tau,T;W^{-1,4}(\Omega)).
\]

\subsection{Positivity and $\varepsilon$-independent uniform bounds for $(u_\varepsilon, v_\varepsilon)$.}

The main concern of this section is to find $\varepsilon$-independent bounds for the regularized solution $(u_\varepsilon, v_\varepsilon)$ of \eqref{reg-eps}. This will allow us to pass to the limit as $\varepsilon\to 0$ to get a solution of our original system \eqref{1}.

We have the following result.
\begin{lemma}\label{lem:uniform_bounds}
    Assume the hypotheses stated above.
    Let $v_0 \in W_{\mathbf{n}}^{\frac{3}{2}, 4}(\Omega)$ and $u_0\in L^4(\Omega)$ with $u_0\ge0$ and $v_0\ge0$ a.e. in $\Omega$.
    Then, the corresponding weak solution $(u_\varepsilon,v_\varepsilon)$ of the system \eqref{reg-eps} satisfies the following.
    
    \begin{enumerate}
        \item[(a)] \textbf{Positivity}: we have that $u_\varepsilon\ge0$ and $v_\varepsilon\ge0$ a.e.\ in $\Omega_T$.
        
        \item[(b)] \textbf{Uniform bounds for $u_\varepsilon$}: there are constants
        $K_1,K_2,K_3>0$, independent of $\varepsilon$, such that
        \begin{equation}\label{eq:uestimate}
        \|u_\varepsilon\|_{L^\infty(0,T;L^4(\Omega))}\le K_1,\quad
        \|u_\varepsilon\|_{L^2(0,T;H^1\Omega)}\le K_2,\quad
        \|\partial_t u_\varepsilon\|_{L^{2}(0,T;(H^{1}(\Omega))')}\le K_3.
        \end{equation}
        
        \item[(c)] \textbf{Maximal parabolic regularity}: 
        there is a constant $K_4>0$, independent of $\varepsilon$, such that
        \begin{equation}\label{eq:vestimate}
            \|v_\varepsilon\|_{L^\infty(0,T;W^{\frac{3}{2},4}(\Omega))}+\|\partial_t v_\varepsilon\|_{L^4(\Omega_T)}
            +\|v_\varepsilon\|_{L^4(0,T;W^{2,4}(\Omega))}\ \le\ K_4.
        \end{equation}
    \end{enumerate}
    All the above constants depend only on the data $(D_u,D_v,\alpha,\beta,\mu,\|f\|_{L^\infty(\Omega_T)},r,\Omega,T)$.
\end{lemma}

\begin{proof}
We proceed in two parts.

\textbf{Part 1.  Positivity}.
Assume that $u_0,v_0\ge0$ a.e. in $\Omega$.

Set $u_\varepsilon^-:=\min\{u_\varepsilon,0\}$  so that  $\nabla u_\varepsilon^-=\nabla u_\varepsilon|_{\{u_\varepsilon<0\}}$.

Testing the $u_\varepsilon$–equation with $u_\varepsilon^-$ and using the Neumann boundary conditions yields
\begin{align*}
    \frac12\frac{d}{dt}\|u_\varepsilon^-\|_{L^2(\Omega)}^2
    + D_u \|\nabla u_\varepsilon^-\|_{L^2(\Omega)}^2
    =& -\mu\int_\Omega a_\varepsilon(u^{-}_\varepsilon) |u_\varepsilon^-|^2\,dx+r\|u_\varepsilon^-\|_{L^2(\Omega)}^2\\
    &+\int_\Omega f\,a_\varepsilon(u_\varepsilon)\,\nabla v_\varepsilon\!\cdot\!\nabla u_\varepsilon^-\,dx.
\end{align*}
On $\{u_\varepsilon<0\}$ we have that $\displaystyle|a_\varepsilon(u_\varepsilon)|=\frac{|u_\varepsilon|}{1+\varepsilon|u_\varepsilon|}\le |u_\varepsilon|$.
Set
\[
    A_\varepsilon(t):=
    \|f(t)\|_{L^\infty(\Omega)}
    \|\nabla v_\varepsilon(t)\|_{L^4(\Omega)}.
\]
Using H\"older's inequality, the inhomogeneous
Gagliardo--Nirenberg inequality, and Young's inequality, we obtain
\[
    \left|
    \int_\Omega f\,a_\varepsilon(u_\varepsilon)
    \nabla v_\varepsilon\cdot\nabla u_\varepsilon^-\,dx
    \right|
    \le
    \frac{D_u}{2}\|\nabla u_\varepsilon^-\|_{L^2(\Omega)}^2
    + C_{\Omega,D_u}\bigl(1+A_\varepsilon(t)^4\bigr)
    \|u_\varepsilon^-\|_{L^2(\Omega)}^2.
\]
Moreover, since $|a_\varepsilon(s)|\le\varepsilon^{-1}$,
\[
    -\mu\int_\Omega a_\varepsilon(u_\varepsilon^-)
    |u_\varepsilon^-|^2\,dx
    \le
    \frac{\mu}{\varepsilon}
    \|u_\varepsilon^-\|_{L^2(\Omega)}^2.
\]
Therefore,
\[
    \frac12\frac{d}{dt}\|u_\varepsilon^-\|_{L^2(\Omega)}^2
    +\frac{D_u}{2}\|\nabla u_\varepsilon^-\|_{L^2(\Omega)}^2
    \le
    \left[
    r+\frac{\mu}{\varepsilon}
    +C_{\Omega,D_u}\bigl(1+A_\varepsilon(t)^4\bigr)
    \right]
    \|u_\varepsilon^-\|_{L^2(\Omega)}^2.
\]
Since $f\in L^\infty(\Omega_T)$ and
$\nabla v_\varepsilon\in L^4(0,T;L^4(\Omega))$, we have
$A_\varepsilon^4\in L^1(0,T).$ Moreover, since $u_\varepsilon^-(0)=0$, Gr\"onwall’s inequality \cite[Appendix B]{Evans-98}  yields $u_\varepsilon^-(t)\equiv0$ in $[0,T]$, i.e.\ $u_\varepsilon\ge0$ a.e. in $\Omega_T$.

Testing the $v_\varepsilon$–equation in \eqref{reg-eps} with $v_\varepsilon^-:=\min\{v_\varepsilon,0\}$ gives
\[
    \frac12\frac{d}{dt}\|v_\varepsilon^-\|_{L^2(\Omega)}^2
    + D_v\|\nabla v_\varepsilon^-\|_{L^2(\Omega)}^2
    + \alpha\|v_\varepsilon^-\|_{L^2(\Omega)}^2
    = \beta\!\int_\Omega a_\varepsilon(u_\varepsilon)\,v_\varepsilon^-\,dx \le 0,
\]
where we have used that $a_\varepsilon(u_\varepsilon)\ge0$ and $v_\varepsilon^-\le0$. With $v_\varepsilon^-(0)=0$, Gr\"onwall's inequality yields $v_\varepsilon^-\equiv0$. Hence, $v_\varepsilon\ge0$ a.e.\ in $\Omega_T$.\\

\subsection*{Part~2. Uniform bounds}
We establish a series of uniform estimates with respect to the
regularization parameter~$\varepsilon>0$. We proceed in three main steps.

\medskip
\noindent
\textbf{Step~1. Uniform $L^\infty(0,T;L^1(\Omega))$, $L^2(\Omega_T)$ and $L^2(0,T;H^1(\Omega))$ bounds for $u_\varepsilon$.}

\begin{itemize}
\item[(i)]
Integrating the $u_\varepsilon$–equation in \eqref{reg-eps} over~$\Omega$ and using the homogeneous Neumann boundary condition gives
\[
    \frac{d}{dt}\!\int_\Omega u_\varepsilon\;dx
    = r\!\int_\Omega u_\varepsilon \;dx- \mu\!\int_\Omega \frac{u_\varepsilon^2}{1+\varepsilon u_\varepsilon}\;dx.
\]
Since $\tfrac{u_\varepsilon^2}{1+\varepsilon u_\varepsilon}\ge0$ and $u_\varepsilon\ge 0$, we can deduce that
\[
    \frac{d}{dt}\!\int_\Omega u_\varepsilon\;dx
    \le r\!\int_\Omega u_\varepsilon\;dx.
\]
Setting $\displaystyle y_\varepsilon(t):=\int_\Omega u_\varepsilon(x,t)\,dx$ and applying Gr\"onwall’s lemma yields
\[
    0 \le y_\varepsilon(t) \le \bigg(\int_\Omega u_0\,dx\bigg)e^{rT}
    =:K_1(T), \quad \forall t\in[0,T],
\]
with $K_1(T)$ independent of~$\varepsilon$.
Consequently, 
\[
    u_\varepsilon \in L^\infty(0,T;L^1(\Omega))
    \quad\text{uniformly in }\varepsilon.
\]

\item[(ii)]
Integrating the mass identity in time over $(0,T)$ gives
\[
    \mu\!\int_0^T\!\!\int_\Omega \frac{u_\varepsilon^2}{1+\varepsilon u_\varepsilon}\;dtdx
    = \int_0^T\!\!\big(r\,y_\varepsilon(t)-y_\varepsilon'(t)\big)\,dt
    \le rTK_1(T) + \int_\Omega u_0\;dx.
\]
Hence, there is a constant $C_2(T)>0$ independent of $\varepsilon$ such that
\begin{equation}
    \int_{\Omega_T} \frac{u_\varepsilon^2}{1+\varepsilon u_\varepsilon}\;dx\;dt
    \le \frac{1}{\mu}\Big(\|u_0\|_{L^1(\Omega)} + rTK_1(T)\Big)
    =: C_2(T).
\label{eq:L2-aeps}
\end{equation}
Since $\displaystyle a_\varepsilon(s)^2
= \frac{s^2}{(1+\varepsilon s)^2}
\le \frac{s^2}{1+\varepsilon s}$,
the inequality \eqref{eq:L2-aeps} implies the uniform bound
\begin{equation}
    \|a_\varepsilon(u_\varepsilon)\|_{L^2(\Omega_T)}^2
    \le \int_{\Omega_T} \frac{u_\varepsilon^2}{1+\varepsilon u_\varepsilon}\;dtdx
    \le C_2(T),
    \quad\text{i.e.}\quad
    a_\varepsilon(u_\varepsilon)\in L^2(\Omega_T)
    \text{ uniformly in }\varepsilon.
\label{eq:aeps-L2}
\end{equation}

\item[(iii)] Here, we establish uniform $H^1$– and $H^2$–bounds for $v_\varepsilon$.
We consider the second equation in~\eqref{reg-eps}.
Testing with $v_\varepsilon$ and using the fact that $0\le |a_\varepsilon(u_\varepsilon)|\le |u_\varepsilon|$ gives
\[
    \frac12\frac{d}{dt}\|v_\varepsilon\|_{L^2(\Omega)}^2
    + D_v\|\nabla v_\varepsilon\|_{L^2(\Omega)}^2
    + \alpha\|v_\varepsilon\|_{L^2(\Omega)}^2
    \le \frac{\alpha}{2}\|v_\varepsilon\|_{L^2(\Omega)}^2
    + \frac{\beta^2}{2\alpha}\|a_\varepsilon(u_\varepsilon)\|_{L^2(\Omega)}^2.
\]
Integrating in time and using~\eqref{eq:aeps-L2} we get that there is a constant $C_1(T)>0$ independent of~$\varepsilon$ such that
\[
    \sup_{t\in[0,T]}\|v_\varepsilon(t)\|_{L^2(\Omega)}^2
    + \int_0^T\!\|\nabla v_\varepsilon\|_{L^2(\Omega)}^2\,dt
    \le C_1(T).
\]

Next, testing the $v_\varepsilon$–equation with $-\Delta v_\varepsilon$ we obtain that
\[
    \frac12\frac{d}{dt}\|\nabla v_\varepsilon\|_{L^2(\Omega)}^2
    + D_v\|\Delta v_\varepsilon\|_{L^2(\Omega)}^2
    + \alpha\|\nabla v_\varepsilon\|_{L^2(\Omega)}^2
    \le \frac{D_v}{2}\|\Delta v_\varepsilon\|_{L^2(\Omega)}^2
    + \frac{\beta^2}{2D_v}\|a_\varepsilon(u_\varepsilon)\|_{L^2(\Omega)}^2.
\]
Integrating in time and combining with the previous estimates we can deduce that
\begin{equation}\label{ee}
    \sup_{t\in[0,T]}\|\nabla v_\varepsilon(t)\|_{L^2(\Omega)}^2
    + \int_0^T\!\|\Delta v_\varepsilon(t)\|_{L^2(\Omega)}^2\,dt
    \le C_2(T),
\end{equation}
with the constant $C_2(T)$ independent of~$\varepsilon$.
Hence, by the elliptic regularity, we have that
\begin{equation}
    v_\varepsilon\in L^\infty(0,T;H^1(\Omega))
    \cap L^2(0,T;H^2(\Omega))
    \quad \text{uniformly in }\varepsilon.
\label{uniformv}
\end{equation}
In particular, in $2$D,
the Sobolev embedding yields
$\nabla v_\varepsilon\in L^4(\Omega_T;\mathbb R^2)$
with a norm bounded independently of $\varepsilon$.
In $2$D, the Ladyzhenskaya/Gagliardo-Nirenberg inequality 
gives 
\[
    \|\nabla v_\varepsilon\|_{L^4(\Omega)}^4
    \;\le\; C\,\|\nabla v_\varepsilon\|_{L^2(\Omega)}^2\,\|\nabla^2 v_\varepsilon\|_{L^2(\Omega)}^2
    \ \le\ C\,\|\nabla v_\varepsilon\|_{L^2(\Omega)}^2 \|\Delta v_\varepsilon\|_{L^2(\Omega)}^2.
\]
Integrating in time and using the uniform $H^1/H^2$ bounds given in \eqref{ee}, we get that there is a constant $C>0$ independent of~$\varepsilon$ such that
\[
    \int_0^T\!\|\nabla v_\varepsilon(t)\|_{L^4(\Omega)}^4\,dt \ \le C.
\]

\item[(iv)] 
Testing the $u_\varepsilon$-equation with $u_\varepsilon$ yields
\[
    \frac12\frac{d}{dt}\|u_\varepsilon\|_{L^2(\Omega)}^2 + D_u\|\nabla u_\varepsilon\|_{L^2(\Omega)}^2
    \leq \!\Bigg |\int_\Omega f\,a_\varepsilon(u_\varepsilon)\,\nabla v_\varepsilon\!\cdot\!\nabla u_\varepsilon\\;dx \Bigg |
    + r\|u_\varepsilon\|_{L^2(\Omega)}^2.
\]
Using the fact that $0\le a_\varepsilon(u_\varepsilon)\le u_\varepsilon$ and the $2$D-Ladyzhenskaya/Gagliardo-Nirenberg inequality, we obtain that
\begin{align*}
    \Big|\int_\Omega f\,a_\varepsilon(u_\varepsilon)\,\nabla v_\varepsilon\!\cdot\!\nabla u_\varepsilon\;dx \Big|
    \le& \|f\|_{L^\infty(\Omega)}\,\|u_\varepsilon\|_{L^4(\Omega)}\,\|\nabla v_\varepsilon\|_{L^4(\Omega)}\,\|\nabla u_\varepsilon\|_{L^2(\Omega)}\\
    \le& \frac{D_u}{2}\|\nabla u_\varepsilon\|_{L^2(\Omega)}^2
    + C\,\|f\|_{L^\infty(\Omega)}^4\,\|\nabla v_\varepsilon\|_{L^4(\Omega)}^4\,\|u_\varepsilon\|_{L^2(\Omega)}^2.
\end{align*}
Therefore,
\[
    \frac{d}{dt}\|u_\varepsilon\|_{L^2(\Omega)}^2 + D_u\|\nabla u_\varepsilon\|_{L^2(\Omega)}^2
    \le C\big(\|\nabla v_\varepsilon\|_{L^4(\Omega)}^4+1\big)\,\|u_\varepsilon\|_{L^2(\Omega)}^2.
\]
Since $\displaystyle\int_0^T\|\nabla v_\varepsilon\|_{L^4(\Omega)}^4\,dt\le C$ by (iii), Gr\"onwall’s lemma gives that
\begin{equation}
    \{ u_\varepsilon\} \text{ is bounded in } L^\infty(0,T;L^2(\Omega))\cap L^2(0,T;H^{1}(\Omega)) \ \text{uniformly in }\varepsilon.\label{reguH1}
\end{equation}
Moreover, the property \eqref{eq:u_C_L4} of the auxiliary system \eqref{FP}  is inherited by $u_\varepsilon$, i.e.,
$u_\varepsilon \in C([0,T];L^4(\Omega))$.

\end{itemize}

\noindent\textbf{Step 2. Uniform \(L^\infty(0,T;L^4(\Omega))\) bound for
\(u_\varepsilon\).} 
By \eqref{local-u-regularity}, for every $\tau>0$ and every
$1<q<\infty$, we have
$u_\varepsilon\in L^q(\tau,T;W^{1,4}(\Omega)).$
Thus, $u_\varepsilon(t)\in W^{1,4}(\Omega)$ for a.e.
$t\in(\tau,T)$. By the continuous embedding
$W^{1,4}(\Omega)\hookrightarrow L^\infty(\Omega)$ and the Sobolev
chain rule, we have
$u_\varepsilon^p(t)\in W^{1,4}(\Omega)$ for every
$1\leq p<\infty$ and for a.e. $t\in(\tau,T)$.
In particular,
$\|u_\varepsilon^3(t)\|_{W^{1,4}}
\leq C\|u_\varepsilon(t)\|_{W^{1,4}}^3.$
Taking $q=12$ in \eqref{local-u-regularity}, we obtain
\[
    \|u_\varepsilon^3\|_{L^4(\tau,T;W^{1,4})}^4
    \leq
    C\|u_\varepsilon\|_{L^{12}(\tau,T;W^{1,4})}^{12}<\infty.
\]
Therefore, for every fixed $\tau>0$, a standard time-regularization argument shows that $u_\varepsilon^3$ is an
admissible test function in the $u_\varepsilon$-equation on
$(\tau,T)$. 
Set
$w_\varepsilon:=u_\varepsilon^2.$
Using $u_\varepsilon^3$ as a test function, the chain rule, and
integration by parts, we obtain, for a.e. $t\in(\tau,T)$,
\begin{align*}
    \frac{1}{4}\frac{d}{dt}\int_\Omega u_\varepsilon^{4}\;dx
    + \frac{3D_u}{4}\int_\Omega |\nabla w_\varepsilon|^2\;dx
    &= 3\!\int_\Omega f\,a_\varepsilon(u_\varepsilon)\,u_\varepsilon^{2}\,
       \nabla v_\varepsilon\!\cdot\!\nabla u_\varepsilon \;dx\\
    &\quad + r\!\int_\Omega u_\varepsilon^{4}\;dx
     - \mu\!\int_\Omega a_\varepsilon(u_\varepsilon) u_\varepsilon^{4}\;dx.
\end{align*}
Since $0\le a_\varepsilon(u_\varepsilon)\le u_\varepsilon$ and $f\in L^\infty(\Omega)$, we have that
\begin{align}
    \nonumber
    \Big|\int_\Omega f\,a_\varepsilon(u_\varepsilon)\,u_\varepsilon^{2}\,\nabla v_\varepsilon\!\cdot\!\nabla u_\varepsilon\;dx\Big|
    \le &\|f\|_{L^\infty(\Omega)}\int_\Omega u_\varepsilon^{3}|\nabla v_\varepsilon||\nabla u_\varepsilon|\;dx\\
    \label{eq:esti_u3}
    = &\frac{1}{2}\|f\|_{L^\infty(\Omega)}\!\int_\Omega |\nabla v_\varepsilon|\,|w_\varepsilon|\,|\nabla w_\varepsilon|\;dx.
\end{align}
Using H\"older's inequality and the 
two-dimensional Gagliardo--Nirenberg inequality, we obtain that
\begin{align*}
    \int_\Omega
    |\nabla v_\varepsilon|\,|w_\varepsilon|\,
    |\nabla w_\varepsilon|\,dx
    &\le
    \|\nabla v_\varepsilon\|_{L^4(\Omega)}
    \|w_\varepsilon\|_{L^4(\Omega)}
    \|\nabla w_\varepsilon\|_{L^2(\Omega)}
    \\
    &\le
    C\|\nabla v_\varepsilon\|_{L^4(\Omega)}
    \Bigl(
    \|w_\varepsilon\|_{L^2(\Omega)}^{1/2}
    \|\nabla w_\varepsilon\|_{L^2(\Omega)}^{3/2}
    +\|w_\varepsilon\|_{L^2(\Omega)}
    \|\nabla w_\varepsilon\|_{L^2(\Omega)}
    \Bigr).
\end{align*}
Consequently, for every $\eta>0$ there is a constant $C_\eta>0$ such that
\begin{align*}
    &\frac{3}{2}\|f\|_{L^\infty(\Omega_T)}
    \int_\Omega
    |\nabla v_\varepsilon|\,|w_\varepsilon|\,
    |\nabla w_\varepsilon|\,dx
    \le
    \eta\|\nabla w_\varepsilon\|_{L^2(\Omega)}^2\\
    &+C_\eta
    \Bigl(
    \|f\|_{L^\infty(\Omega_T)}^4
    \|\nabla v_\varepsilon\|_{L^4(\Omega)}^4
    +
    \|f\|_{L^\infty(\Omega_T)}^2
    \|\nabla v_\varepsilon\|_{L^4(\Omega)}^2
    \Bigr)
    \|w_\varepsilon\|_{L^2(\Omega)}^2
    \\
    &\quad\le
    \eta\|\nabla w_\varepsilon\|_{L^2(\Omega)}^2
    +
    C_\eta
    \Bigl(
    1+
    \|f\|_{L^\infty(\Omega_T)}^4
    \|\nabla v_\varepsilon\|_{L^4(\Omega)}^4
    \Bigr)
    \|w_\varepsilon\|_{L^2(\Omega)}^2.
\end{align*}
Choosing $\eta:=3D_u/8$, recalling that
$\|w_\varepsilon\|_{L^2(\Omega)}^2
=\int_\Omega u_\varepsilon^4\,dx,$
absorbing the reaction term into the constant, and dropping the
nonpositive logistic term, we obtain, for a.e. $t\in(\tau,T)$,
\begin{equation}
\label{eq:est_u_eps_L4}
    \frac14\frac{d}{dt}
    \|u_\varepsilon(t)\|_{L^4(\Omega)}^4
    +\frac{3D_u}{8}
    \|\nabla(u_\varepsilon^2)(t)\|_{L^2(\Omega)}^2
    \leq
    C_0\left(
    1+\|\nabla v_\varepsilon(t)\|_{L^4(\Omega)}^4
    \right)
    \|u_\varepsilon(t)\|_{L^4(\Omega)}^4,    
\end{equation}
where $C_0>0$ is independent of $\tau$ and $\varepsilon$.
Applying Gr\"onwall's inequality on $(\tau,t)$ gives
\[
    \|u_\varepsilon(t)\|_{L^4(\Omega)}^4
    \leq
    \|u_\varepsilon(\tau)\|_{L^4(\Omega)}^4
    \exp\left(
    4C_0\int_\tau^t
    \left(
    1+\|\nabla v_\varepsilon(s)\|_{L^4(\Omega)}^4
    \right)\,ds
    \right).
\]
Since $u_\varepsilon\in C([0,T];L^4(\Omega))$ and
$u_\varepsilon(0)=u_0$, we have
\[
    \|u_\varepsilon(\tau)\|_{L^4(\Omega)}^4
    \longrightarrow
    \|u_0\|_{L^4(\Omega)}^4
    \qquad\text{as }\tau\downarrow0.
\]
Letting $\tau\downarrow0$, we obtain
\[
    \|u_\varepsilon(t)\|_{L^4(\Omega)}^4
    \leq
    \|u_0\|_{L^4(\Omega)}^4
    \exp\left(
    4C_0\int_0^t
    \left(
    1+\|\nabla v_\varepsilon(s)\|_{L^4(\Omega)}^4
    \right)\,ds
    \right).
\]
Since
\[
    \int_0^T
    \|\nabla v_\varepsilon(s)\|_{L^4(\Omega)}^4\,ds
    \leq C
\]
uniformly in $\varepsilon$, it follows that there is a constant $C>0$ independent of $\varepsilon$ such that
\begin{equation}
    \|u_\varepsilon\|_{L^\infty(0,T;L^4(\Omega))}
    \leq C.
\label{uniformbound}
\end{equation}

\noindent\textbf{Step 3. Uniform bound for $v_\varepsilon$}.
Since $u_\varepsilon$ is bounded in $L^4(\Omega_T)$ independently of $\varepsilon$, it follows from the  maximal parabolic regularity that there is a constant $K_4>0$, independent of $\varepsilon$, such that
\begin{equation}\label{eqvestimate}
    \|v_\varepsilon\|_{L^\infty(0,T;W^{\frac{3}{2},4}(\Omega))}+\|\partial_t v_\varepsilon\|_{L^4(\Omega_T)}
    +\|v_\varepsilon\|_{L^4(0,T;W^{2,4}(\Omega))}\ \le\ K_4,
\end{equation}
and the proof is finished.
\end{proof}

Now we are ready to give the proof of the first main theorem.

\subsection{Proof of Theorem \ref{mainresult1}}
\label{subsec:proof_main_result}

We proceed in two steps.

{\bf Step 1. Existence of solutions.} Recall the system \eqref{reg-eps}.
By \eqref{reguH1},  we have that
$D_u\Delta u_\varepsilon,\;ru_\varepsilon
\in L^2(0,T;(H^{1}(\Omega))')$.
Using \eqref{uniformbound} and \eqref{eqvestimate}, we have that
\[
    a_\varepsilon(u_\varepsilon)\in L^4(0,T;L^4(\Omega)),
    \qquad
    \nabla v_\varepsilon\in L^\infty(0,T;L^q(\Omega;\mathbb R^2))
    \quad \text{for every } 1\le q<\infty,
\]
where the latter follows from the two-dimensional embedding
\(W^{1/2,4}(\Omega)\hookrightarrow L^q(\Omega)\), \(1\le q<\infty\).
Hence, for every \(1\le p<4\), we have that
\[
    a_\varepsilon(u_\varepsilon)\nabla v_\varepsilon
    \in L^4(0,T;L^p(\Omega)^2)\;\mbox{ and }\;
    \nabla\!\cdot\!\bigl(f\,a_\varepsilon(u_\varepsilon)\nabla v_\varepsilon\bigr)\in L^{4}(0,T;W^{-1,p}(\Omega)).
\]  
Also there is a constant $C>0$ independent of $\varepsilon$ such that
\begin{equation}\label{eq:dtueps}
    \|\partial_t u_\varepsilon\|_{L^2(0,T;(H^{1}(\Omega))')}\ \le\ C.
\end{equation}

The compact embeddings $H^1(\Omega)\embc L^2(\Omega)\embc (H^{1}(\Omega))'$ together with \eqref{reguH1}, and 
the Aubin–Lions Lemma 
yields, as $\varepsilon\to 0$, 
\begin{equation}\label{eq:AL-strong}
    u_\varepsilon \to u \ \ \text{in } L^2(0,T;L^2(\Omega)),\quad
    u_\varepsilon \rightharpoonup u \ \ \text{in } L^2(0,T;H^1(\Omega)).
\end{equation}
Moreover, by \eqref{eq:dtueps}, after a subsequence if necessary, as $\varepsilon\to 0$, we have that 
\begin{equation}\label{eq:dt-weak}
    \partial_t u_\varepsilon \rightharpoonup \partial_t u
    \ \ \text{in } L^2(0,T;(H^{1}(\Omega))').
\end{equation}
Using in addition \eqref{reguH1}–\eqref{eq:dtueps}, we obtain that, as $\varepsilon\to 0$,
\begin{equation}\label{eq:u_compact}
    u_\varepsilon \to u \quad\text{in } L^{2}\bigl(0,T;L^4(\Omega)\bigr).
\end{equation}
Interpolating \eqref{eq:u_compact} with the uniform $L^\infty-L^4$ bound yields, as $\varepsilon\to 0$,
\begin{equation}\label{eq:u_L4plus_strong}
    u_\varepsilon \to u \quad\text{ in } L^4(\Omega_T).
\end{equation}
Indeed,  we have that there is a constant $C>0$, independent of $\varepsilon$, such that
\[
    \|u_\varepsilon\|_{ L^2(0,T;H^1(\Omega))}+\|u_\varepsilon\|_{ L^\infty(0,T;L^2(\Omega))}+\|\partial_t u_\varepsilon\|_{ L^2(0,T;(H^{1}(\Omega))')}\le C.
\] 
 Thus, $u_\varepsilon \to u$ in $L^2(0,T;L^4(\Omega))\), as $\varepsilon\to 0$. Moreover, since there is a constant $M>0$, independent of $\varepsilon$, such that $\displaystyle\|u_\varepsilon\|_{L^\infty(0,T;L^4(\Omega))} \le M$,  we have that, as $\varepsilon\to 0$,
\[
    u_\varepsilon \overset{\ast}{\rightharpoonup} u \in L^\infty(0,T; L^4(\Omega))\text{ and }
    \|u\|_{L^\infty(0,T;L^4(\Omega))}\le \lim\inf_{\varepsilon\to 0}\|u_\varepsilon\|_{L^\infty(0,T;L^4(\Omega))} \le M,
\]
and the already established strong convergence $u_\varepsilon\to u$ in $L^{2}(0,T;L^4(\Omega))$, as $\varepsilon\to 0$,  can be upgraded  
to $L^4(\Omega_T)$. In fact, setting 
$b_\varepsilon(t):=\|u_\varepsilon(t)-u(t)\|_{L^4(\Omega)}$, we have that
\[
    \int_0^T b_\varepsilon(t)^{4}\,dt
    =\int_0^T b_\varepsilon(t)^2\,b_\varepsilon(t)^2\,dt
    \le \Big(\sup_{t\in(0,T)} b_\varepsilon(t)\Big)^{2}\int_0^T b_\varepsilon(t)^2\,dt.
\]
Since $\displaystyle\sup_{t\in [0,T]} b_\varepsilon(t)\le 2M$ by the uniform $L^\infty- L^4$ bound, it follows that, as $\varepsilon\to 0$,
\[
    \|u_\varepsilon-u\|_{L^4(\Omega_T)}^{4}
    \le (2M)^{2} \|u_\varepsilon-u\|_{L^{2}(0,T;L^4(\Omega))}^{2}\;\longrightarrow\;0.
\]
Hence, we can deduce that \eqref{eq:u_L4plus_strong} holds.
Moreover, since $u_\varepsilon \in C([0,T];L^4(\Omega))$ we have, 
for any $\varepsilon>0$, 
\begin{equation}
    \label{eq:u_eps_u0}
    u_\varepsilon(0)=u_0 \text{ in } L^4(\Omega).
\end{equation}
From the maximal regularity of the $v_\varepsilon$-equation, we have that
\begin{equation}\label{eq:v_maxreg}
    v_\varepsilon \in L^4\bigl(0,T;W^{2,4}(\Omega)\bigr) \cap C([0,T]; W^{\frac{3}{2},4}(\Omega)),\quad
    \partial_t v_\varepsilon \in L^{4}\bigl(\Omega_T\bigr).
\end{equation}
Then $v_\varepsilon(0)=v_0 \in W_{\mathbf{n}}^{\frac{3}{2},4}(\Omega)$ for any $\varepsilon>0$. 
Moreover, up to a subsequence, as $\varepsilon\to 0$, we have that
\begin{equation}\label{eq:v_compact}
    v_\varepsilon \rightharpoonup v \ \text{ in } L^4(0,T;W^{2,4}(\Omega)) ,\quad
    \partial_t v_\varepsilon \rightharpoonup \partial_t v \ \text{ in } L^{4}(\Omega_T).
\end{equation}
and, using the compact embeddings (Aubin–Lions Lemma), we obtain that, as $\varepsilon\to 0$,
\begin{equation}\label{eq:nablav_strong}
    \nabla v_\varepsilon \to \nabla v \quad\text{in } L^{4}(\Omega_T;\mathbb R^2).
\end{equation}
\medskip
We observe that the regularizer $a_\varepsilon(s)=\frac{s}{1+\varepsilon|s|}$ is Lipschitz continuous and
$|a_\varepsilon(s)-s|\le \varepsilon |s|^2$. Hence, from \eqref{eq:u_compact}, we get the  estimate
\[
    \|a_\varepsilon(u_\varepsilon)-u\|_{L^4(\Omega_T)}
    \;\le\;
    \|a_\varepsilon(u_\varepsilon)-a_\varepsilon(u)\|_{L^4(\Omega_T)}
    \;+\;
    \|a_\varepsilon(u)-u\|_{L^4(\Omega_T)}.
\]
Since $a_\varepsilon$ is Lipschitz continuous, we have that the first term is bounded by $\|u_\varepsilon - u\|_{L^4(\Omega_T)}$ which converges to $0$, as $\varepsilon\to 0$.
Together with \eqref{eq:nablav_strong}, we obtain that, as $\varepsilon\to 0$,
\begin{equation}\label{eq:drift_strong}
    f\,a_\varepsilon(u_\varepsilon)\,\nabla v_\varepsilon \ \to\ f\,u\,\nabla v
    \quad\text{in }L^{2}(\Omega_T;\mathbb R^2).
\end{equation}
Moreover, using \eqref{eq:u_L4plus_strong} we have that, as $\varepsilon\to 0$, 
\begin{equation}
    a_\varepsilon(u_\varepsilon)
    =\frac{u_\varepsilon}{1+\varepsilon |u_\varepsilon|}
    \ \to\ u \quad\text{in }L^4(\Omega_T),
\end{equation}
and
\begin{equation}
    a_\varepsilon(u_\varepsilon)u_\varepsilon
    =\frac{u^2_\varepsilon}{1+\varepsilon |u_\varepsilon|}
    \ \to\ u^2 \quad\text{in }L^{2}(\Omega_T).
\end{equation}
\medskip
For any $\varphi\in L^{2}(0,T;H^{1}(\Omega)),$ the weak formulation for $u_\varepsilon$ and \eqref{eq:u_L4plus_strong}–\eqref{eq:drift_strong} allow passing to the limit, as $\varepsilon\to 0$, to get
\begin{align*}
    \!\int_{0}^{T}\!\!\langle \partial_t u,\varphi\rangle_{(H^{1}(\Omega))',H^1(\Omega)}\,dt
    &+ D_u\!\int_{\Omega_T}\nabla u\!\cdot\!\nabla \varphi\;dtdx
    - \int_{\Omega_T} f\,u\,\nabla v\!\cdot\!\nabla\varphi\;dtdx\\
    =& \int_{\Omega_T} r\,u\,\varphi\;dtdx - \mu\!\int_{\Omega_T} u^2\,\varphi\;dtdx.
\end{align*}
This is the \emph{weak form} of the limit $u$–equation with the Neumann boundary conditions.

Since $v_\varepsilon$ solves the equation in the strong form and
$a_\varepsilon(u_\varepsilon)\to u$ in $L^4(\Omega_T)$, as $\varepsilon\to 0$, passing to the limit as $\varepsilon\to 0$, we get from the regularity in \eqref{eq:v_compact} that $v$ indeed satisfies the equation in the strong form, that is,
\[
    \partial_t v - D_v\Delta v + \alpha v = \beta u \quad\text{ in }\Omega_T,\quad
    \partial_\nu v=0\ \text{on }\Gamma_T,\quad v(0)=v_0 \mbox{ in } \Omega.
\]

In fact, from \eqref{reguH1}–\eqref{eq:dtueps}, using \cite[Corollary 4]{simon1986compact}, with $X=L^2(\Omega)$ and $B=Y=(H^{1}(\Omega))'$,
we have that
$u_\varepsilon\to u$ in $C([0,T];(H^{1}(\Omega))')$, as $\varepsilon\to 0$.
Hence, from \eqref{eq:u_eps_u0}, and the fact that $u\in C([0,T];L^2(\Omega))$ (Lions-Magenes Lemma) we have $u(0)=u_0$ a.e. in $\Omega$.
By the strong formulation and the maximal regularity, $v_\varepsilon\to v$ in $C([0,T];W^{1,4}(\Omega))$, as $\varepsilon\to 0$. 
By \eqref{eq:v_maxreg} we have $v_\varepsilon(0) = v_0 
$ for any $\varepsilon>0$, then $v(0)=v_0$. The non-negativity follows by testing the $u_\varepsilon$–equation with the negative part (standard truncation)
and applying the parabolic maximum principle to the strong $v_\varepsilon$–equation. Then, passing to the limit, as $\varepsilon\to 0$, we can conclude that $u,v\ge0$ a.e. in $\Omega_T$. 

{\bf Step 2. Uniqueness of solutions.} 
Let $(u_1,v_1), (u_2,v_2) \in W_4 \times X_4$ be two solutions of our system. Denoting $(u,v)=(u_1-u_2,v_1-v_2)$ we obtain the system
\begin{equation}\label{1-1}
    \begin{cases}
        \partial_t u - D_u \Delta u + \nabla \cdot (f u_1 \nabla v + f u \nabla v_2) = r u - \mu u (u_1 + u_2), &  \mbox{ in  }\Omega_T, \\
        \partial_t v - D_v \Delta v + \alpha v = \beta u, &  \mbox{ in } \Omega_T, \\
        \partial_\nu u = \partial_\nu v = 0, & \mbox{ on  }\Gamma_T, \\
        u(\cdot,0) = 0, \; v(\cdot,0) = 0, &  \mbox{ in } \Omega.
    \end{cases}
\end{equation}
Testing the first equation in \eqref{1-1} with $u \in L^2(0,T; H^1(\Omega))$ and the second equation with $v - D_v \Delta v \in L^4(\Omega_T)$, we get
\begin{align*}
    &\frac{1}{2} \frac{d}{dt} \left( \|u\|_{L^2(\Omega)}^2 + \|v\|_{L^2(\Omega)}^2 + D_v \|\nabla v\|_{L^2(\Omega)}^2 \right) + D_u \|\nabla u\|_{L^2(\Omega)}^2 + D_v \|\nabla v\|_{L^2(\Omega)}^2 + D_v^2 \|\Delta v\|_{L^2(\Omega)}^2 \\
    &+ \mu \int_\Omega u^2(u_1 + u_2)\;dx + \alpha \|v\|_{L^2(\Omega)}^2 + \alpha D_v \|\nabla v\|_{L^2(\Omega)}^2\\
    =& r\|u\|_{L^2(\Omega)}^2 + \left( f u_1 \nabla v + f u \nabla v_2 , \nabla u \right)_{L^2(\Omega)} + \left(\beta u,v-D_u \Delta v \right)_{L^2(\Omega)}.
\end{align*}
Thus, there is a constant $C>0$ such that
\begin{align*}
    &\frac{1}{2} \frac{d}{dt} \left( \|u\|_{L^2(\Omega)}^2 + \|v\|^2_{H^1(\Omega)} \right) + \|\nabla u\|^2_{H^1(\Omega)} + \|\nabla v\|^2_{H^2(\Omega)} + \mu \int_\Omega u^2(u_1 + u_2)\;dx \\
    \leq &C \left(\|u\|_{L^2(\Omega)}^2 + \left( f u_1 \nabla v + f u \nabla v_2 , \nabla u \right)_{L^2(\Omega)} + \left(u,v-\Delta v \right)_{L^2(\Omega)} \right).
\end{align*}
Since $f \in L^\infty(\Omega)$, we have that
$$
    |(f u_1 \nabla v,\nabla u)_{L^2(\Omega)}| \leq \|f\|_{L^\infty(\Omega)} \|u_1\|_{L^{4}(\Omega)} \|\nabla v\|_{L^4(\Omega)} \|\nabla u\|_{L^2(\Omega)}.
$$
Using the Ladyzhenskaya and Young's inequalities we get, for every $\delta>0$,  
\begin{align*}
    \left|(u_1 \nabla v,\nabla u)_{L^2(\Omega)} \right|\leq& \|u_1\|_{L^4(\Omega)} \|\nabla v\|_{L^2(\Omega)}^{1/2} \|\nabla v\|_{H_1(\Omega)}^{1/2} \|\nabla u\|_{L^2(\Omega)} \\
    \leq &\delta \left(\|\nabla v\|_{H^1(\Omega)}^2 + \|\nabla u\|_{L_2(\Omega)}^2 \right) + C_\delta \|u_1\|_{L^4(\Omega)}^4 \|\nabla v\|_{L^2(\Omega)}^2.
\end{align*}
Analogously, 
\begin{align*}
    |(f u \nabla v_2,\nabla u)_{L^2(\Omega)}| \leq &\|f\|_{L^\infty(\Omega)} \|u\|_{L^{4}(\Omega)} \|\nabla v_2\|_{L^4(\Omega)} \|\nabla u\|_{L^2(\Omega)} \\
    \leq & C \|u\|_{L^2(\Omega)}^{1/2} \|\nabla v_2\|_{L^4(\Omega)} \|u\|_{H^1(\Omega)}^{3/2}.
\end{align*}
Thus, for every $\alpha>0$, we have that
$$
    |(f u \nabla v_2,\nabla u)_{L^2(\Omega)}| \leq \alpha \|u\|_{H^1}^2 + C_\alpha \|\nabla v_2\|_{L^4(\Omega)}^4 \|u\|^2_{L^2(\Omega)}.
$$
On the other hand, 
$$
    |\left(u,v-\Delta v \right)_{L^2(\Omega)}| \leq \alpha (\|v\|^2 + \|\Delta v\|^2 ) + C_\alpha \|u\|^2, 
$$
and taking $\alpha$ small enough, after rearranging the terms, we get
\begin{align}\label{eq:gronwald_unicity}
    &\frac{1}{2} \frac{d}{dt} \left( \|u\|_{L^2(\Omega)}^2 + \|v\|^2_{H^1(\Omega)} \right) + \|u\|^2_{H^1(\Omega)} + \|v\|^2_{H^2(\Omega)} \notag\\
    &\leq C_2 \left(\|u\|_{L^2(\Omega)}^2 + \|u_1\|^4_{L^{4}(\Omega)}\|\nabla v\|^2_{L^2(\Omega)} + \|\nabla v_2\|_{L^{4}(\Omega)}^4 \|u\|_{L^2(\Omega)}^2 \right).
\end{align}
Letting $U(t) := \|u(\cdot,t)\|_{L^2(\Omega)}^2 + \|v(\cdot,t)\|^2_{H^1(\Omega)}$, the  inequality \eqref{eq:gronwald_unicity} implies that
$$
    \frac{1}{2} \frac{d}{dt} U(t) \leq C_2\big( 1 + \|u_1\|^4_{L^{4}(\Omega)} + \|\nabla v_2\|_{L^{4}(\Omega)}^4\big) U(t).
$$
Thus, by Gr\"onwall's inequality, since $U(0)=0$ we get $U(t)=0$ for a.e. $t \in [0,T]$, which implies that $u(t)=0$ and $v(t)=0$ for a.e. $t \in [0,T]$, giving the uniqueness. The proof is finished.

\section{Proof of the second main Theorem}\label{sec:existence_OC}

In this section we give the proof of the existence of optimal solutions.

\begin{proof}[\bf of Theorem \ref{thm:existence}]
Recall that the admissible set $\mathcal S_{ad}$ for the optimal control problem \eqref{1}-\eqref{cont-prob} is given in \eqref{Sad}.
From Theorem \ref{mainresult1}, one has that $\mathcal S_{ad} \neq \emptyset$. Let $\{s_m\}_{m \in \N} := \{(u_m, v_m, f_m)\}_{m \in \N} \subset \mathcal S_{ad}$ be a minimizing sequence of the functional $J$. That is,
\[
    \lim_{m \to +\infty} J(s_m) = \inf_{s:=(u,v,f) \in \mathcal S_{ad}} J(s).
\]
The existence of a minimizing sequence follows from the fact that the functional $J$ is bounded from below by $0$.
By the definition of $\mathcal S_{ad}$, for each $m \in \mathbb{N}$, we have that $(u_m,v_m)$ satisfies the system \eqref{1} with control function $f_m$. Then, for all $\varphi\in L^4(0,T,W^{1,4}(\Omega))$, we have that
\begin{align}\label{eq:system_m}
    \displaystyle\int_0^T\!\!\langle \partial_t u_m,\varphi\rangle_{(H^{1}(\Omega))',H^1(\Omega)}\,dt
    &+ D_u\!\int_{\Omega_T} \nabla u_m\!\cdot\!\nabla\varphi\;dx dt
    - \int_{\Omega_T} f_m u_m\,\nabla v_m\!\cdot\!\nabla\varphi\;dxdt\notag\\
    =& \int_{\Omega_T} (r-\mu u_m)\,u_m\,\varphi\;dxdt,
\end{align}
and 
\begin{equation}\label{eq:system_m-1}
    \displaystyle\D\partial_t v_m - D_v \Delta v_m + \alpha v_m = \beta u_m \; \mbox{ in } \Omega_T.
\end{equation}

In addition we have the boundary and initial conditions
\begin{equation}\label{eq:system_m-2}
    \begin{cases}
        \displaystyle\partial_\nu u_m = \partial_\nu v_m = 0, & \mbox{ on  }\Gamma_T, \\
        \displaystyle\D u_m(\cdot,0) = u_0, \; v_m(\cdot,0) = v_0, & \mbox{ in  }\Omega.
    \end{cases}
\end{equation}

From the definition of $J$ and the assumptions, we can deduce that
\begin{equation}\label{22}
    \{f_m\}_{m \in \N} \text{ is bounded in } L^{4}(\Omega_T). 
\end{equation}
From Theorem \ref{mainresult1}, we have that there is a constant $C > 0$, independent of $m$, such that
\begin{equation}\label{23}
    \| (u_m, v_m)\|_{W_4 \times X_{4}} \leq C. 
\end{equation}
Therefore, using \eqref{22}-\eqref{23}, and taking into account that
$\mathcal{F} \subset L^4(\Omega_T)$ is a closed and convex subset,
we can deduce that for all $f \in \mathcal{F}$, we have that $|f(x,t)| \leq M(x,t)$ for a.e.  $(x,t)$ in $\Omega_T$,
where \(M \in L^\infty(\Omega_T)\). Therefore, \(\mathcal{F}\) is bounded in \(L^4(\Omega_T)\), and since the space \(L^4(\Omega_T)\) is reflexive, it follows that \(\mathcal{F}\) is weakly compact. Hence, there exists $\tilde{s} = (\tilde{u}, \tilde{v}, \tilde{f}) \in W_4 \times X_{4} \times \mathcal{F}$ such that, for some subsequence of $\{s_m\}_{m \in \N}$ (still denoted by $\{s_m\}_{m \in \N}$), the following convergences hold, as $m \to +\infty$:
\begin{align}
    u_m \rightharpoonup \tilde{u} 
    & \quad \text{in } L^4(\Omega_T), \label{24} \\
    u_m \rightharpoonup \tilde{u} 
    & \quad \text{in } L^2(0,T; H^{1}(\Omega)), \label{24b} \\
    v_m \rightharpoonup \tilde{v} 
    & \quad \text{in } L^4(0,T; W^{2,4}(\Omega)), \label{25} \\
    \partial_t u_m \rightharpoonup \partial_t \tilde{u} 
    & \quad \text{in } L^2(0,T;(H^{1}(\Omega))'), \label{26} \\
    \partial_t v_m \rightharpoonup \partial_t \tilde{v} 
    & \quad \text{in } L^{4}(\Omega_T), \label{27} \\
    f_m \rightharpoonup \tilde{f} 
    & \quad \text{in } L^{4}(\Omega_T), \quad \tilde{f} \in \mathcal{F}. \label{28}
\end{align}
From the convergences \eqref{24}-\eqref{27}, using the Sobolev embeddings and the Aubin-Lions compactness results, 
we obtain that the pair \((\tilde{u}, \tilde{v})\) satisfies the weak formulation described in \eqref{W1} and \eqref{W2}. Then, one has that, as $m\to\infty$,
\begin{align}
    u_m &\to \tilde{u} \quad \text{in } L^4(\Omega_T), \label{29} \\
    v_m &\to \tilde{v} \quad \text{in } L^{4}(\Omega_T), \label{30} \\
    \nabla v_m &\to \nabla \tilde{v} \quad \text{in } L^{4}(\Omega_T;\mathbb R^2), \label{31} \\
    u_m \nabla v_m &\to \tilde{u} \nabla \tilde{v} \quad \text{in } L^2(\Omega_T;\mathbb R^2). \label{32}
\end{align}
Using the Aubin-Lions Lemma for $u_m$ with $X_0=H^1(\Omega)$, $X=L^4(\Omega)$, $X_1=(H^{1}(\Omega))'$, we obtain that $u_m \to \tilde{u}$ in $L^2(0,T;L^4(\Omega))$, as $m\to\infty$,  and as above it implies that $u_m \to \tilde{u}$ in $L^4(\Omega_T)$, as $m\to\infty$. 
It follows that $u_m \to \tilde{u}$ in $L^2(\Omega_T)$, as $m\to\infty$. In addition, we have that
\begin{align}\label{conv-4}
    \int_{\Omega_T}|u_m^2-\tilde u^2|^2\;dtdx=&\int_{\Omega_T}|u_m-\tilde u|^2|u_m+\tilde u|^2\;dtdx\notag\\
    \le &\|u_m-\tilde u\|_{L^4(\Omega_T)}^2\|u_m+\tilde u\|_{L^4(\Omega_T)}^2\notag\\
    \le & C(\tilde u)\|u_m-\tilde u\|_{L^4(\Omega_T)}^2,
\end{align}
where $C(\tilde u)$ is a constant depending only on $\tilde u$ but is independent of $m$ (here we have used the fact that $(u_m)$ is bounded in $L^4(\Omega_T))$. Taking the limit of \eqref{conv-4} as $m\to\infty$, we can deduce that $u_m^2\to \tilde u^2$ in $L^2(\Omega_T)$, as $m\to\infty$, hence weakly. That is,
\begin{equation*}
    \lim_{n\to\infty}\int_{\Omega_T}u_m^2\varphi\;dtdx=\int_{\Omega_T}\tilde u^2\varphi\;dtdx .
\end{equation*}

Let $\hat v$ be the solution of the second equation in \eqref{1} with source $\tilde u$ and initial datum $v_0$.  
Set $w_m:=v_m-\hat v$. Then, $w_m$ satisfies
\[
    \partial_t w_m - D_v \Delta w_m + \alpha w_m = \beta(u_m-\tilde u),\quad
    w_m(\cdot,0)=0,\;\;\partial_n w_m=0.
\]
By the maximal ($L^4-L^2$)–regularity, we have that, there is a constant $C>0$ independent of $m$ such that
\[
    \|w_m\|_{L^4(0,T;W^{2,4}(\Omega))}+\|\partial_t w_m\|_{L^4(\Omega_T)}
    \;\le C\,\|u_m-\tilde u\|_{L^4(\Omega_T)}.
\]
Since $u_m\to\tilde u$ in $L^4(\Omega_T)$, as $m\to\infty$, it follows that $w_m\to 0$ in $L^4(0,T;W^{2,4}(\Omega))$ and hence, in $L^4(\Omega_T)$, as $m\to\infty$. In addition, we have that $\Delta w_m\to 0$ and $\partial_tw_n\to 0$, in $L^4(\Omega_T)$, as $m\to\infty$. 
Hence, as $m\to\infty$,
\[
    v_m\to \hat v \;\text{ in } L^4(0,T;W^{2,4}(\Omega)) \mbox{ and in }L^4(\Omega_T) \quad\text{and}\quad \nabla v_m\to \nabla\hat v
    \quad\text{in } L^4(\Omega_T;\mathbb R^2)
\]
and
\[
    \Delta v_m\to\Delta \hat v\;\mbox{ and } \partial_t v_m\to \partial_t\hat v\;\mbox{ in } L^4(\Omega_T).
\]
The uniqueness of the limit shows that $\hat v=\tilde v$ a.e. in $\Omega_T$.

In particular, using \eqref{28}, \eqref{29}, \eqref{31} and \eqref{32}, the limit of the nonlinear terms in \eqref{eq:system_m} can be controlled as follows:
\begin{align}\label{33} 
    f_m u_m \nabla v_m &\rightharpoonup \tilde{f} \tilde{u} \nabla \tilde{v} \quad \text{in } L^{\frac{4}{3}}(\Omega_T;\mathbb R^2), \mbox{ as } m\to\infty.
\end{align}
We conclude that, as $m\to\infty$,
\[
    \int_{\Omega_T} f_m u_m \nabla v_m. \nabla\varphi \;dtdx  \to \int_{\Omega_T}\tilde{f} \tilde{u} \nabla \tilde{v}. \nabla\varphi\;dt  dx, \quad \forall \varphi \in L^4(0,T;W^{1,4}(\Omega)).
\] 

Finally, combining all the above convergences, we have that, for all  $\varphi \in L^4(0,T;W^{1,4}(\Omega))$,
\begin{align*}
    \int_0^T \langle \partial_t \tilde{u}, \varphi \rangle_{(H^{1}(\Omega))',H^1(\Omega)}\, dt 
    &+ \int_{\Omega_T}  D_{u} \nabla \tilde{u} \cdot \nabla \varphi  \, dtdx \\
    =& \int_{\Omega_T} (r \tilde{u}-\mu\tilde u^2) \varphi\, dtdx 
    + \int_{\Omega_T} f \tilde{u} \nabla \tilde{v} \cdot \nabla \varphi\, dtdx,
\end{align*}
and
\[
    \partial_t \tilde{v}-D_v \Delta\tilde{v}+\alpha \tilde{v}=\tilde{u}.
\]
Moreover, using also the Aubin-Lions and Simon \cite{simon1986compact} results mentioned above, 
we obtain that \( (u_m(0),  v_m(0)) \to (\tilde{u}(0), \tilde{v}(0)) \) in \( (H^{1}(\Omega))' \times L^2(\Omega) \), as $m\to\infty$. 
Since \( u_m(0) = u_0 \text{ in } L^4(\Omega) \), \( v_m(0) = v_0 \in W_{\textbf{n}}^{\frac{3}{2},4}(\Omega) \), for any $m \geq 1$, it follows that \( \tilde{u}(0) = u_0 \text{ a.e. in } \Omega\)  and \( \tilde{v}(0) = v_0 \text{ in } W_{\textbf{n}}^{\frac{3}{2},4}(\Omega)\). 
Thus, \( (\tilde u,\tilde v) \) satisfies the initial conditions given in \eqref{1}.
Passing to the limit in the weak formulation of the state system,
we conclude that $(\widetilde u,\widetilde v)$ is a weak solution
of \eqref{1} associated with the control $\widetilde f$. Hence,
$(\widetilde u,\widetilde v,\widetilde f)\in\mathcal S_{ad}.$
Moreover, by continuity of the trace operator,
$u_m(T)\rightharpoonup\widetilde u(T)
\quad\text{in }L^2(\Omega),$
and therefore
\[
    \|\widetilde u(T)-u_T\|_{L^2(\Omega)}^2
    \le
    \liminf_{m\to\infty}
    \|u_m(T)-u_T\|_{L^2(\Omega)}^2.
\]
Combining this inequality with the convergence and weak lower semicontinuity properties established above for the other terms of $J$, we obtain
\begin{equation}\label{34}
    \inf_{(u,v,f)\in\mathcal S_{ad}}J(u,v,f)
    \le J(\widetilde u,\widetilde v,\widetilde f)
    \le \liminf_{m\to\infty}J(s_m)
    =\inf_{(u,v,f)\in\mathcal S_{ad}}J(u,v,f).
\end{equation}
Consequently,
\[
    J(\widetilde u,\widetilde v,\widetilde f)=\inf_{(u,v,f)\in\mathcal S_{ad}}J(u,v,f),
\]
and $(\widetilde u,\widetilde v,\widetilde f)$ is a global optimal
state-control triplet.
\end{proof}

\section{The Lagrange Multiplier Method: General Theorem and Applications}
\label{sec:lagrange_multipliers}

In the previous section we proved the existence of at least one solution of the optimal control problem. In this section we derive the optimality system using the Lagrange multiplier method in Banach spaces. This method has the advantage of being a systematic and rigorous way to include state constraints to the optimization problem. Moreover, it allows the study of the regularity of the dual variables (which correspond to the Lagrange multipliers). We refer to Zowe \& Kurcyusz \cite{Zowe_multipliers} and Tr\"oltzsch \cite[Chapter 6]{Troltzsch_book} for more details on this topic.

\subsection{Deriving the optimality system}
To obtain the first-order necessary optimality conditions we use the theory of Lagrange multipliers on Banach spaces. 
First we recall some general definitions.

Consider the following optimization problem:
\begin{equation}\label{eq:probl_lagrange_optim}
    \min_{s \in \mathbb{M}} \mathbb J(s) \quad \text{subject to } G(s)=0,    
\end{equation}
where $\mathbb J : \mathbb{M} \to \R$ is a functional, $G:\mathbb{X} \to \mathbb{Y}$ is an operator, $\mathbb{X}$ and $\mathbb{Y}$ are Banach spaces and $\mathbb{M}$ is a non-empty closed and convex subset of $\mathbb{X}$. 

The \emph{admissible set} $\mathcal{S}$ of the problem \eqref{eq:probl_lagrange_optim} is defined by
$$
    \mathcal{S} = \{ s \in \mathbb{M} \ : \ G(s)=0 \}.
$$

The associated \emph{Lagrangian functional}  $\mathcal{L}:\mathbb{X} \times \mathbb{Y}' \to \R$ related to \eqref{eq:probl_lagrange_optim} is given by
$$
    \mathcal{L}(s,\xi) = \mathbb J(s) - \langle \xi, G(s) \rangle_{\mathbb{Y}',\mathbb{Y}},
$$
where $\mathbb{Y}'$ denotes the dual of $\mathbb{Y}$. In this section, we use $\langle \cdot, \cdot \rangle_{\mathbb{Y}',\mathbb{Y}}$ to denote the duality and $(\cdot,\cdot)_{L^2(\Omega_T)}$ to denote the $L^2$-scalar product.

\begin{definition}
    We introduce the following well-known notions.
    \begin{enumerate}
        \item[(a)]  We say that \(\tilde s\in\mathcal S\) is a local minimizer of \eqref{eq:probl_lagrange_optim} if there exists a neighbourhood \(U\) of \(\tilde s\) in \(\mathbb X\) such that
        \[
            \mathbb J(\tilde s)\leq \mathbb J(s)
            \quad
            \text{for all } s\in \mathcal S\cap U .
        \]
        \item[(b)] Let $\tilde s \in \mathcal{S}$ be a local optimal solution of  \eqref{eq:probl_lagrange_optim}. Suppose that $\mathbb{J}$ and $G$ are Fréchet differentiable at $\tilde s$. Then, $\xi \in \mathbb{Y}'$ is called a \emph{Lagrange multiplier at} $\tilde s$ if
        $$
            \mathbb J'(\tilde s)[z] - \langle \xi, G'(\tilde s)[z] \rangle_{\mathbb{Y}',\mathbb{Y}} \geq 0,  \quad \forall z \in \mathcal{C}(\tilde s),
        $$
        where $\mathcal{C}(\tilde s):= \{ \theta(s-\tilde s): \ s \in \mathbb{M}, \theta \geq 0 \}$ denotes the conical hull of $\tilde s$ in $\mathbb{M}$.     
        \item[(c)] A point $\tilde s \in \mathcal{S}$ is said to be a \emph{regular point} if
        $$
            G'(\tilde s)[\mathcal{C}(\tilde s)] = \mathbb{Y}.
        $$
    \end{enumerate}
\end{definition}

The notion of regular points is important because, under the Fréchet differentiability of $\mathbb J$ and $G$, the general theory implies the existence of Lagrange multipliers at that point (see e.g. \cite[Theorem 3.1]{Zowe_multipliers} and \cite[Chapter 6]{Troltzsch_book}).

We will apply the Lagrange multiplier method to our problem. 
We recall that the spaces $W_4$ and $X_4$ were given in Definition \ref{def:weak-solution} and $\mathcal{F} \subset L^{4}(\Omega_T)$ was defined in \eqref{eq:def_F}. Our optimal control problem is the following.
\begin{equation}\label{eq:optimal_problem}
    \begin{cases}
        \text{Find } (u,v,f) \in W_4 \times X_4 \times \mathcal{F} \text{ minimizing the functional } J \text{ defined in \eqref{cout}} \\
        \text{subject to the condition that } (u,v)  \text{ is a weak solution of   \eqref{1}} \text{ with control } f.    
    \end{cases}
\end{equation}

Throughout the rest of this section, to simplify the notations we shall simply denote the duality pairing $ \langle \cdot,\cdot \rangle_{L^2(0,T;(H^{1}(\Omega))'),L^2(0,T;H^{1}(\Omega))}$ by $\langle \cdot,\cdot\rangle$.

We consider the Banach spaces
$$
    \mathbb{X} := W_4 \times \widetilde X_{4} \times L^4(\Omega_T), \quad \mathbb{Y}:= L^2(0,T; (H^{1}(\Omega))') \times L^4(\Omega_T),
$$
where
$$
    \widetilde X_{4} = \{ v \in X_4 \ : \ \partial_\nu v |_{\partial \Omega} =0 \},
$$
and the operator $G=(G_1,G_2) : \mathbb{X} \to \mathbb{Y}$, i.e. $G_1 : \mathbb{X} \to L^2(0,T; (H^{1}(\Omega))')$ and $G_2 : \mathbb{X} \to L^4(\Omega_T)$, are given for every $s=(u,v,f)\in\mathbb X$ and $\varphi \in L^2(0,T;H^{1}(\Omega))$
by
\begin{equation}\label{eq:def_G1_G2}
    \begin{cases}
        \langle G_1(s),\varphi \rangle = \langle \partial_t u, \varphi \rangle  + (D_u \nabla u - f u \nabla v, \nabla \varphi)_{L^2(\Omega_T)} + (-ru +\mu u^2,\varphi)_{L^2(\Omega_T)}, \\
        G_2(s) = \partial_t v - D_v \Delta v + \alpha v - \beta u.
    \end{cases}  
\end{equation}

The set $\mathbb{M} := (\hat u,\hat v,0) + \hat W_4 \times \hat X_{4} \times \mathcal{F}$ is a closed, convex subset of $\mathbb{X}$,
where $(\hat u,\hat v)$ is the global weak solution of \eqref{1} without control, that is, taking $f=0$, and
$$
    \hat W_4 = \{ u \in W_4 \ : \ u(0)=0
    \}, \quad \hat X_{4} = \{ v \in X_{4} \ : \ v(0)=0, \ \partial_\nu v |_{\partial \Omega} = 0 \}.
$$

Thus, for our optimal control problem \eqref{eq:optimal_problem} the admissible set $\mathcal{S}_{ad}$ is given by
\begin{equation}\label{eq:admissible_set_lagran}
    \mathcal{S}_{ad} = \{ s=(u,v,f) \in \mathbb{M} \ : \ G(s)=0 \}.    
\end{equation}

First we need to show the Fréchet differentiability of the functional $J$ and the operator $G$. This is proved in Appendices \ref{appendix_b_J} and \ref{appendix_c_G}. Then, Lemma \ref{lem:regular_point} in Appendix \ref{appendix_reg_point} 
shows that a local optimal solution $\tilde s=(\tilde u,\tilde v, \tilde f) \in \mathcal{S}_{ad}$ is a regular point. 
The proof consists in showing the existence of solutions of a PDE in the space $\hat W_4 \times \hat X_4$ as we did  in Sections \ref{sec:introduction} and \ref{sec:proof_theorem_wellp} for the nonlinear system \eqref{1}.

All these ingredients are used in the following theorem that allows us to prove the existence of a Lagrange multiplier associated to a local  optimal solution $\tilde s=(\tilde u,\tilde v, \tilde f) \in \mathcal{S}_{ad}$ of the control problem \eqref{eq:optimal_problem}.

\begin{theorem}\label{teo:exist_lagr}
    Let $\tilde s=(\tilde u, \tilde v,\tilde f) \in \mathcal{S}_{ad}$ be a local optimal solution of  \eqref{eq:optimal_problem}. Then, there exists a Lagrange multiplier $(\lambda,\eta) \in \mathbb{Y}' = L^2(0,T; H^1(\Omega))\times L^{4/3}(\Omega_T)$ such that, for all $(U,V,F) \in \hat W_4 \times \hat X_4 \times \mathcal{C}(\tilde f)$, we have that
    \begin{equation}\label{eq:full_condition}
        \begin{aligned}
            &2\gamma_u \int_{\Omega_T} (\tilde u - u_d)\, U \ dx dt
            \;+\; 4\gamma_f \int_{\Omega_T} \tilde f^{\,3} F \ dx dt
             \\
             &+\, 2\gamma_{\text{taxis}} 
            \int_{\Omega_T} (\tilde f \tilde u \nabla \tilde v)\cdot
            \big( \tilde f \tilde u \nabla V 
            + \tilde f U \nabla \tilde v
            + F \tilde u \nabla \tilde v \big) \ dx dt \\
            &+\, 2\gamma_T \int_\Omega (\tilde u(T) - u_d(T))\, U(T) \ dx \\
            &- \int_0^T \langle \partial_t U , \lambda \rangle \ dt - \int_{\Omega_T}\big( D_u \nabla U 
            - \tilde f U \nabla \tilde v
            - \tilde f \tilde u \nabla V
            - F \tilde u \nabla \tilde v \big)\cdot
            \nabla \lambda\;dxdt \\
            & + \int_{\Omega_T}(-r U + 2\mu \tilde u U) \lambda\;dxdt - \int_{\Omega_T} 
            \big( \partial_t V - D_v \Delta V + \alpha V - \beta U \big) \eta \ dx dt \;\ge 0 .
        \end{aligned}
    \end{equation}
\end{theorem}

\begin{proof}
    From Lemma \ref{lem:regular_point} we have that $\tilde s \in \mathcal{S}_{ad}$ is a regular point. By \cite[Theorem 6.3, p. 330]{Troltzsch_book}, if $J$ is Fréchet differentiable and $G$ is continuously-Fréchet differentiable, and if $\tilde s$ is a regular point, 
    then there exists a Lagrange multiplier $(\lambda,\eta) \in \mathbb{Y}'$ at $\tilde s$ such that
    $$
        J'(\tilde s)[z] - \langle \lambda, G_1'(\tilde s)[z] \rangle_{L^2(0,T;H^{1}(\Omega)),L^2(0,T;(H^{1}(\Omega))')} - \langle \eta, G_2'(\tilde s)[z] \rangle_{L^{4/3}(\Omega_T),L^4(\Omega_T)} \geq 0,
    $$
    for all $z=(U,V,F) \in \hat W_4 \times \hat X_4 \times \mathcal{C}(\tilde f)$. Substituting the derivatives of $J$ and $G$ obtained in  Lemmas \ref{lem:der_J} and \ref{lem:der_G}, we get the result.
\end{proof}

From Theorem \ref{teo:exist_lagr}, we derive an optimality system.

\begin{corollary}\label{cor:opt_system1}
    Let $(\tilde u,\tilde v, \tilde f) \in \mathcal{S}_{ad}$ be a local optimal solution of \eqref{eq:optimal_problem}. Then, any Lagrange multiplier $(\lambda,\eta) \in \mathbb{Y}' = L^2(0,T;H^1(\Omega))\times L^{4/3}(\Omega_T)$ provided by Theorem \ref{teo:exist_lagr} satisfies the following system, for all $(U,V)\in \hat W_4 \times \hat X_4$:
    \begin{equation}\label{eq:optimality}
        \begin{cases}
            \displaystyle    \int_0^T \langle \partial_t U,\lambda \rangle dt + \int_{\Omega_T} (D_u \nabla U - \tilde f U \nabla \tilde v)\cdot \nabla \lambda \ dx dt \\
            \displaystyle    \qquad + \int_{\Omega_T}(-rU+2\mu\tilde u U)\lambda\;dxdt
            - \beta\int_{\Omega_T} U\eta \ dx dt
            = 2\gamma_u \int_{\Omega_T} (\tilde u-u_d)U \ dx dt \\
            \displaystyle   \qquad + 2 \gamma_{\text{taxis}} \int_{\Omega_T} |\tilde f|^2 \tilde u U |\nabla \tilde v|^2 dx dt + 2\gamma_T \int_\Omega (\tilde u(T) - u_d(T))U(T) \ dx,  \vspace{2mm}\\
            \displaystyle    \int_{\Omega_T} (\partial_t V - D_v \Delta V + \alpha V)\eta \ dx dt - \int_{\Omega_T} \tilde f \tilde u \nabla V \cdot \nabla \lambda \ dx dt \\
            \displaystyle   \qquad = 2 \gamma_{\text{taxis}} \int_{\Omega_T} |\tilde f|^2 |\tilde u|^2 \nabla \tilde v \cdot \nabla V \ dx dt, 
        \end{cases}
    \end{equation}
    and the optimality condition
    \begin{equation}\label{eq:optimality_cond}
        \int_{\Omega_T} \big(4 \gamma_f \tilde f^3 + \tilde u \nabla \tilde v \cdot \nabla \lambda + 2\gamma_{\text{taxis}} \tilde f |\tilde u|^2 |\nabla \tilde v|^2\big)(f-\tilde f) \ dx dt \geq 0, \quad f \in \mathcal{F}.    
    \end{equation}
\end{corollary}

\begin{proof}
    Taking $(V,F)=(0,0)$ in \eqref{eq:full_condition} we obtain the first equation in \eqref{eq:optimality}. Similarly, taking $(U,F)=(0,0)$ in \eqref{eq:full_condition} we get the second equation of \eqref{eq:optimality}. Finally, taking $(U,V)=(0,0)$ in \eqref{eq:full_condition}  we get
    \begin{equation}\label{optimalitycond}
        4\gamma_f \int_{\Omega_T} F \tilde f^3 \ dx dt + \int_{\Omega_T} F \tilde u \nabla \tilde v \cdot \nabla \lambda \ dx dt + 2\gamma_{\text{taxis}} \int_{\Omega_T} F \tilde f |\tilde u|^2 |\nabla \tilde v|^2 \ dx dt \geq 0, \quad \forall F \in C(\tilde f).
    \end{equation}
    Taking $F=\theta(f-\tilde f) \in \mathcal{C}(\tilde f)$, for $f \in \mathcal{F}$ and $\theta \geq 0$ in \eqref{optimalitycond}, we get  \eqref{eq:optimality_cond}.
\end{proof}

Now we show that the Lagrange-multiplier $(\lambda,\eta) \in \mathbb{Y}'$ has more regularity. In fact, \eqref{eq:optimality} corresponds to the notion of a very weak solution of the system \eqref{eq:optimality3} bellow.

\begin{proposition} \label{prop:adjoint_regularity}
    Let $(\tilde u,\tilde v, \tilde f) \in \mathcal{S}_{ad}$ be a local optimal solution of  \eqref{eq:optimal_problem}. Then, the system
    \begin{equation}\label{eq:optimality3}
    \begin{cases}
        - \partial_t \lambda - D_u \Delta \lambda  - \tilde f \nabla \tilde v \cdot \nabla \lambda -r\lambda +2\mu\tilde u \lambda 
        - \beta \eta = 2\gamma_u (\tilde u-u_d) +
        2 \gamma_{\text{taxis}} |\tilde f|^2 \tilde u  |\nabla \tilde v|^2 & \text{in } \Omega_T,\\ 
        -\partial_t \eta - D_v \Delta \eta + \alpha \eta + \nabla \cdot (\tilde f \tilde u \nabla \lambda) = - 2 \gamma_{\text{taxis}} \nabla \cdot ( |\tilde f|^2 |\tilde u|^2 \nabla \tilde v) & \text{in } \Omega_T,\\
        \partial_\nu \lambda = 0, \quad \partial_\nu \eta = 0 & \text{on } \Gamma_T,\\
        \lambda(T) = 2\gamma_T(\tilde u(T)-u_d),\quad \eta(T) = 0 & \text{in } \Omega,
    \end{cases}
    \end{equation}
    has a unique weak solution $(\lambda,\eta)$ satisfying
    \begin{equation}\label{eta-lam}
        \lambda \in L^\infty(0,T;H^1(\Omega)) \cap L^2(0,T;H^2(\Omega)), \quad  \eta \in L^\infty(0,T;L^{2}(\Omega)) \cap L^{2}(0,T;H^{1}(\Omega)).
    \end{equation}
\end{proposition}

\begin{proof}
    Let $\tau=T-t$. Then the system \eqref{eq:optimality3} is equivalent to
    \begin{equation}\label{eq:optimality_s}
        \begin{cases}
            \partial_\tau \lambda - D_u \Delta \lambda  - \tilde f \nabla \tilde v \cdot \nabla \lambda -r\lambda +2\mu\tilde u \lambda 
            - \beta \eta = 2\gamma_u (\tilde u-u_d) +
            2 \gamma_{\text{taxis}} |\tilde f|^2 \tilde u  |\nabla \tilde v|^2, & \text{in } \Omega_T,\\ 
            \partial_\tau \eta - D_v \Delta \eta + \alpha \eta + \nabla \cdot (\tilde f \tilde u \nabla \lambda) = - 2 \gamma_{\text{taxis}} \nabla \cdot ( |\tilde f|^2 |\tilde u|^2 \nabla \tilde v), & \text{in } \Omega_T,\\
            \partial_\nu \lambda = 0, \quad \partial_\nu \eta = 0, & \text{on } \Gamma_T,\\
            \lambda(0) = 2\gamma_T(\tilde u(T)-u_d),\quad \eta(0) = 0, & \text{in } \Omega, 
        \end{cases}
    \end{equation}
    where, for simplicity, we used the same notations for the unknowns $(\lambda,\eta)$.
    The existence of solution to \eqref{eq:optimality_s} is obtained by the method of Faedo-Galerkin, using the following energy estimate.
    Multiplying the first equation in \eqref{eq:optimality_s} by $-\Delta \lambda$ and the second one by $\eta$, we get 
    \begin{equation*}
        \begin{split}
            &\frac{d}{d\tau}(\|\lambda\|_{H^1(\Omega)}^2 + \|\eta\|_{L^2(\Omega)}^2)+C(\|\lambda\|_{H^2(\Omega)}^2 + \|\eta\|_{H^1(\Omega)}^2)\\
            &\qquad \leq C(\|\eta\|_{L^2(\Omega)}^2 +  
            \|\tilde u\|_{L^4(\Omega)}\|\lambda\|_{L^4(\Omega)}\|\Delta \lambda\|_{L^2(\Omega)} + 
            \|\tilde f\|_{L^\infty(\Omega)}\|\lambda\|_{L^4(\Omega)} \|\Delta \tilde v\|_{L^4(\Omega)} \|\Delta \lambda\|_{L^2(\Omega)}) \\
            &\qquad \quad+ C\|\tilde u - u_d\|_{L^2(\Omega)}^2 + C\|\tilde f\|_{L^\infty(\Omega)}^2\|\tilde u\|_{L^4(\Omega)} \|\nabla \tilde v\|_{L^8(\Omega)}^2 \|\lambda \|_{H^2(\Omega)}\\
            &\qquad \quad + C\|\tilde f\|_{L^\infty(\Omega)} \|\tilde u\|_{L^4(\Omega)} \|\nabla \lambda\|_{L^4(\Omega)} \|\eta\|_{L^2(\Omega)} + C\|\tilde f\|_{L^\infty(\Omega)}^2\|\tilde u\|_{L^6(\Omega)}^2 \|\nabla \tilde v\|_{L^6(\Omega)} \|\nabla \eta \|_{L^2(\Omega)}. 
        \end{split}    
    \end{equation*}
    Thus, using Young's inequality and the embedding $H^1(\Omega) \hookrightarrow L^q(\Omega)$, for $1 \leq q \leq 6$, we obtain
    \begin{equation*}
        \begin{split}
            &\frac{d}{d\tau}(\|\lambda\|_{H^1(\Omega)}^2 + \|\eta\|_{L^2(\Omega)}^2)+C(\|\lambda\|_{H^2(\Omega)}^2 + \|\eta\|_{H^1(\Omega)}^2)\\
            &\qquad \leq C(\|\eta\|_{L^2(\Omega)}^2 + 
            \|\tilde u\|_{L^4(\Omega)}^2\|\lambda\|_{H^1(\Omega)}^2 + 
            \|\tilde f\|_{L^\infty(\Omega)}^2\|\lambda\|_{H^1(\Omega)}^2 \|\Delta \tilde v\|_{L^4(\Omega)}^2  \\
            &\qquad \quad+ C\|\tilde u - u_d\|_{L^2(\Omega)}^2 + C\|\tilde f\|_{L^\infty(\Omega)}^4\|\tilde u\|_{L^4(\Omega)}^2 \|\nabla \tilde v\|_{L^8(\Omega)}^4 \\
            &\qquad \quad + C\|\tilde f\|_{L^\infty(\Omega)}^2 \|\tilde u\|_{L^4(\Omega)}^2  \|\eta\|_{L^2(\Omega)}^2 + C\|\tilde f\|_{L^\infty(\Omega)}^4\|\tilde u\|_{L^6(\Omega)}^4 \|\nabla \tilde v\|_{L^6(\Omega)}^2 \\
            &\qquad\leq C(1+\|\tilde f\|_{L^\infty(\Omega)}^2 \|\tilde u\|_{L^4(\Omega)}^2)\|\eta\|_{L^2(\Omega)}^2 + C(\|\tilde u\|_{L^4(\Omega)}^2 + \| f\|_{L^\infty(\Omega)}^2 \|\Delta \tilde v\|_{L^4(\Omega)}^2 )\|\lambda\|_{H^1(\Omega)}^2 \\
            &\qquad \quad+ C\|\tilde u - u_d\|_{L^2(\Omega)}^2 + C\|\tilde f\|_{L^\infty(\Omega)}^4\|\tilde u\|_{L^4(\Omega)}^2 \|\nabla \tilde v\|_{L^8(\Omega)}^4 + C\|\tilde f\|_{L^\infty(\Omega)}^4\|\tilde u\|_{H^1(\Omega)}^4 \|\nabla \tilde v\|_{L^6(\Omega)}^2.
        \end{split}    
    \end{equation*}
    Recalling that $\nabla \tilde v \in L^q(\Omega_T;\mathbb R^2)$ for any $q\geq 1$, from Grönwall's inequality we can deduce that the pair $(\lambda,\eta)$ satisfies \eqref{eta-lam}.
\end{proof}

\subsection{Control characterization}
\label{subsec:control_characterization}

Let $(\widetilde u,\widetilde v,\widetilde f)
\in\mathcal S_{\rm ad}$
be a local optimal state-control triple and let $(\lambda,\eta)$ be the
associated adjoint pair. Corollary~\ref{cor:opt_system1} shows that the control
component of the first-order optimality condition is represented by
\begin{equation}\label{eq:reduced_gradient}
    G_f
    =
    \widetilde u\nabla\widetilde v\cdot\nabla\lambda
    +
    2\gamma_{\rm taxis}\widetilde f\,\widetilde u^2
    |\nabla\widetilde v|^2
    +
    4\gamma_f\widetilde f^3 .
\end{equation}
More precisely, $G_f\in L^{4/3}(\Omega_T)$ is paired with admissible control
variations in $L^4(\Omega_T)$.

Since $\mathcal F$ is the box-constrained control set defined in
\eqref{eq:def_F}, the optimal control satisfies
\begin{equation}
\label{eq:box_optimality_condition}
    \int_{\Omega_T}
    G_f\,(g-\widetilde f)\,dx\,dt
    \geq0,
    \qquad
    \forall g\in\mathcal F .
\end{equation}
Equivalently, for every $\tau>0$,
\begin{equation}
\label{eq:projected_gradient_condition}
    \widetilde f
    =
    \operatorname{Proj}_{[\xi_1,\xi_2]}
    \bigl(\widetilde f-\tau G_f\bigr)
    \qquad\text{a.e. in }\Omega_T,
\end{equation}
where the projection is understood pointwise. In the numerical experiments,
the bounds $\xi_1$ and $\xi_2$ are taken to be constants.

On the inactive set
\[
    \{(x,t)\in\Omega_T:
    \xi_1(x,t)<\widetilde f(x,t)<\xi_2(x,t)\},
\]
one has $G_f=0$, and hence
\begin{equation}
\label{eq:cubic_control}
    4\gamma_f\widetilde f^3
    +
    2\gamma_{\rm taxis}
    \widetilde u^2|\nabla\widetilde v|^2\widetilde f
    +
    \widetilde u\nabla\widetilde v\cdot\nabla\lambda
    =0.
\end{equation}
Assume that $\gamma_f>0$ and set
\[
    A=\widetilde u^2|\nabla\widetilde v|^2,
    \qquad
    B=\widetilde u\nabla\widetilde v\cdot\nabla\lambda.
\]
Then,
\[
    \widetilde f^3+p_c\widetilde f+q_c=0,
    \;\; \mbox{ where } \;
    p_c=\frac{\gamma_{\rm taxis}}{2\gamma_f}A,
    \qquad
    q_c=\frac{B}{4\gamma_f}.
\]
Since $A\geq0$, the cubic polynomial is monotone and has a unique real zero
given by
\begin{equation}
\label{eq:cardano_explicit}
    \begin{aligned}
        f_{\rm un}
        &=
        \operatorname{cbrt}\left(
        -\frac{B}{8\gamma_f}
        +
        \sqrt{
        \left(\frac{B}{8\gamma_f}\right)^2
        +
        \left(\frac{\gamma_{\rm taxis}}{6\gamma_f}A\right)^3
        }
        \right)
        \\
        &\quad+
        \operatorname{cbrt}\left(
        -\frac{B}{8\gamma_f}
        -
        \sqrt{
        \left(\frac{B}{8\gamma_f}\right)^2
        +
        \left(\frac{\gamma_{\rm taxis}}{6\gamma_f}A\right)^3
        }
        \right),
    \end{aligned}
\end{equation}
where
\[
    \operatorname{cbrt}(z)
    =
    \operatorname{sign}(z)|z|^{1/3}.
\]
The corresponding box-constrained pointwise candidate is
\begin{equation}
    \label{eq:update_control}
    f_{\rm car}
    =
    \operatorname{Proj}_{[\xi_1,\xi_2]}(f_{\rm un}).
\end{equation}
There is no additional global minus sign in
\eqref{eq:cardano_explicit}. In the numerical algorithm,
$f_{\rm car}$ is used as a candidate control and is accepted only after
projection and line search.

\section{Numerical implementation}
\label{sec:numerical_implementation}

This section explains how the continuous problem is implemented in practice and records the discretization choices used in the simulations of Section~\ref{sec:numerical_examples}.

The analytical state space remains $W_4\times X_4$, while the admissible
controls belong to $\mathcal F\subset L^4(\Omega_T)$. The finite-dimensional
approximation uses tensor-product B-spline collocation in space and either a
semi-implicit IMEX Euler method or an explicit Runge--Kutta method in time.

Let
\[
    t_k=k\Delta t,\qquad k=0,\ldots,N,\qquad T=N\Delta t.
\]
The discrete control $f^{j,k}$ is kept constant on each interval
$[t_k,t_{k+1})$, where $j$ denotes the optimization iteration.

\subsection{Tensor-product B-spline collocation}
\label{subsec:bspline_collocation}

Let $d\geq1$ be the B-spline degree and let
$\{B_m^x\}_{m=1}^{N_x}$ and
$\{B_\ell^y\}_{\ell=1}^{N_y}$ be the corresponding univariate bases.
For $w\in\{u,v,\lambda,\eta\}$, we use the tensor-product approximation
\begin{equation}
\label{eq:bspline_expansion}
    w_B(x,y,t)
    =
    \sum_{m=1}^{N_x}
    \sum_{\ell=1}^{N_y}
    a_{m,\ell}^{w}(t)
    B_m^x(x)B_\ell^y(y).
\end{equation}
The collocation nodes are the Greville points associated with the knot
vectors. The one-dimensional differentiation matrices induce the discrete
operators
\[
    \nabla_B,\qquad
    \operatorname{div}_B,\qquad
    \Delta_B,
\]
with the same sign convention for $\Delta_B$ as for the continuous Laplacian.

Homogeneous Neumann boundary conditions are incorporated through constrained
maps between nodal values and B-spline coefficients. The taxis term is
evaluated in divergence form:
\begin{equation}
\label{eq:bspline_taxis_flux}
    \operatorname{div}_B(fu\nabla_Bv).
\end{equation}

All spatial integrals are approximated by a tensor-product quadrature on the
Greville grid. For a nodal array $W=(W_{\ell,m})$, we set
\begin{equation}
\label{eq:bspline_quadrature}
    \mathcal Q_B(W)
    =
    \sum_{m=1}^{N_x}
    \sum_{\ell=1}^{N_y}
    \omega_m^x\omega_\ell^yW_{\ell,m}.
\end{equation}

After a forward time step, the state may be projected onto the Neumann-compatible B-spline space and onto the nonnegative cone at the collocation points. These projections are numerical safeguards and are not imposed on the adjoint variables.

In the conservative case, the implementation also allows the optional
rescaling
\begin{equation}
\label{eq:optional_mass_correction}
    u_B^{k+1}
    \leftarrow
    \frac{\mathcal Q_B(u_B^0)}
     {\mathcal Q_B(u_B^{k+1})}
    u_B^{k+1}.
\end{equation}
No mass rescaling is applied in the logistic case.

\subsection{Time discretization}
\label{subsec:time_discretization}

For a fixed control value $f$, define
\begin{align}
    \mathcal F_{u,B}(u,v,f)
    &=
    D_u\Delta_Bu
    -
    \operatorname{div}_B(fu\nabla_Bv)
    +
    ru-\mu u^2,
    \label{eq:Fu_bspline}
    \\
    \mathcal F_{v,B}(u,v)
    &=
    D_v\Delta_Bv-\alpha v+\beta u.
\label{eq:Fv_bspline}
\end{align}

For the IMEX Euler option, diffusion is treated implicitly and the taxis and
reaction terms explicitly:
\begin{equation}
\label{eq:forward_u_discrete}
    \left(I-\Delta t\,D_u\Delta_B\right)u^{j,k+1}
    =
    u^{j,k}
    +
    \Delta t
    \left[
    -\operatorname{div}_B
    \left(
    f^{j,k}u^{j,k}\nabla_Bv^{j,k}
    \right)
    +
    ru^{j,k}-\mu(u^{j,k})^2
    \right],
\end{equation}
and
\begin{equation}
\label{eq:forward_v_discrete}
    \left(I-\Delta t\,D_v\Delta_B\right)v^{j,k+1}
    =
    v^{j,k}
    +
    \Delta t
    \left[
    -\alpha v^{j,k}+\beta u^{j,k+1}
    \right].
\end{equation}

For the explicit option, write
\[
    Z=
    \begin{pmatrix}u\\ v\end{pmatrix},
    \qquad
    \mathcal F_B(Z,f)
    =
    \begin{pmatrix}
        \mathcal F_{u,B}(u,v,f)\\
        \mathcal F_{v,B}(u,v)
    \end{pmatrix}.
\]
Keeping $f^{j,k}$ fixed on $[t_k,t_{k+1})$, the classical RK4 update is
\begin{equation}
\label{eq:rk4_state_update}
    Z^{j,k+1}
    =
    Z^{j,k}
    +
    \frac{\Delta t}{6}
    \left(
    K_1+2K_2+2K_3+K_4
    \right),
\end{equation}
where
\[
    K_1=\mathcal F_B(Z^{j,k},f^{j,k}),
    \qquad
    K_2=\mathcal F_B
    \left(
    Z^{j,k}+\frac{\Delta t}{2}K_1,f^{j,k}
    \right),
\]
\[
    K_3=\mathcal F_B
    \left(
    Z^{j,k}+\frac{\Delta t}{2}K_2,f^{j,k}
    \right),
    \qquad
    K_4=\mathcal F_B
    \left(
    Z^{j,k}+\Delta tK_3,f^{j,k}
    \right).
\]
Thus, $u$ and $v$ are advanced simultaneously at every Runge--Kutta stage.
Internal substeps may be used for stability, while the control remains
constant over the full control interval.

By Proposition~\ref{prop:adjoint_regularity}, the adjoint system admits the additional regularity
\[
    \lambda
    \in
    L^\infty(0,T;H^1(\Omega))
    \cap
    L^2(0,T;H^2(\Omega)),
    \qquad
    \eta
    \in
    L^\infty(0,T;L^2(\Omega))
    \cap
    L^2(0,T;H^1(\Omega)).
\]
This justifies the strong-form B-spline discretization of the adjoint system.
Its terminal conditions are
\begin{equation}
\label{eq:terminal_adjoints_algorithm}
    \lambda^{j,N}
    =
    2\gamma_T(u^{j,N}-u_T),
    \qquad
    \eta^{j,N}=0.
\end{equation}
For the IMEX option, the adjoint diffusion terms are treated with the same implicit matrices as in the state equations. For the explicit option, the adjoint is integrated in reverse time with the same RK4 formula, using an intervalwise reconstruction of the stored forward state.

\subsection{Projected-gradient algorithm}
\label{sec:algorithm}

The characterization in
Subsection~\ref{subsec:control_characterization} provides the continuous
first-order condition and the Cardano candidate. We now describe their
discrete use.

For a given control $f^j\in\mathcal F$, the state system is solved forward,
the adjoint system is solved backward, and the reduced cost
\[
    \mathcal J(f^j)=J(u_{f^j},v_{f^j},f^j)
\]
is evaluated. The discrete control derivative is
\begin{equation}
\label{eq:discrete_reduced_gradient}
    G_B^{j,k}
    =
    u^{j,k}
    \nabla_Bv^{j,k}\cdot\nabla_B\lambda^{j,k}
    +
    2\gamma_{\rm taxis}
    f^{j,k}(u^{j,k})^2|\nabla_Bv^{j,k}|^2
    +
    4\gamma_f(f^{j,k})^3.
\end{equation}

The implementation tests the following candidate directions:
\begin{itemize}
    \item the normalized gradient direction
    \[
        d_{\rm grad}^j=-\sigma_jG_B^j,
        \qquad
        \sigma_j
        =
        \frac{\kappa\Delta_\xi}
        {\max\{\|G_B^j\|_\infty,10^{-14}\}},
    \]
    where $\Delta_\xi=\xi_2-\xi_1$ for constant bounds;
    
    \item the bang-bang direction
    \[
        f_{\rm bb}^{j,k}
        =
        \begin{cases}
            \xi_1, & G_B^{j,k}>0,\\
            \xi_2, & G_B^{j,k}<0,\\
            f^{j,k}, & G_B^{j,k}=0,
        \end{cases}
        \qquad
        d_{\rm bb}^j=f_{\rm bb}^j-f^j;
    \]
    
    \item when $\gamma_f>0$, the Cardano direction
    \[
        d_{\rm car}^j=f_{\rm car}^j-f^j,
    \]
    where $f_{\rm car}^j$ is obtained from
    \eqref{eq:cardano_explicit}--\eqref{eq:update_control}.
\end{itemize}

An additional damped projected-gradient candidate may be used in the logistic computation.

For each candidate direction $d^j$ and
\[
    \alpha_\ell=q_{\rm ls}^{\,\ell},
    \qquad
    \ell=0,\ldots,L_{\max}-1,
    \qquad
    0<q_{\rm ls}<1,
\]
the projected trial control is
\begin{equation}
\label{eq:line_search_control}
    f_{\ell,d}^{j+1}
    =
    \operatorname{Proj}_{[\xi_1,\xi_2]}
    \left(
    f^j+\alpha_\ell d^j
    \right).
\end{equation}
The state equations are solved again for every trial control. Among all tested
candidates, the algorithm accepts the one with the smallest reduced cost,
provided that
\begin{equation}
\label{eq:line_search_decrease}
    \mathcal J(f_{\ell,d}^{j+1})
    <
    \mathcal J(f^j)
    -
    \varepsilon_{\rm dec}
\   max\{1,|\mathcal J(f^j)|\}.
\end{equation}

The discrete projected-gradient mapping is
\begin{equation}
\label{eq:projected_gradient_norm}
    \mathcal G_{\tau_j}(f^j)
    =
    \frac{
    f^j-
    \operatorname{Proj}_{[\xi_1,\xi_2]}
    \left(
    f^j-\tau_jG_B^j
    \right)
    }{
    \tau_j
    },
    \qquad
    \tau_j
    =
    \frac{1}{
    \max\{\|G_B^j\|_\infty,10^{-14}\}
    }.
\end{equation}
For the numerical stopping criterion, we use the discrete $L^2$-type norm
\begin{equation}
\label{eq:discrete_L2_space_time_norm}
    \|W\|_{B,2}^2
    =
    \Delta t
    \sum_{k=0}^{N-1}
    \mathcal Q_B(|W^k|^2).
\end{equation}
This is a numerical stationarity norm; the analytical control space remains $L^4(\Omega_T)$.

After an accepted update, define
\[
    r_J^j
    =
    \frac{
    |\mathcal J(f^{j+1})-\mathcal J(f^j)|
    }{
    \max\{1,|\mathcal J(f^j)|\}
    },
    \qquad
    r_f^j
    =
    \frac{
    \|f^{j+1}-f^j\|_{B,2}
    }{
    \max\{1,\|f^j\|_{B,2}\}
    }.
\]
The algorithm stops when
\begin{equation}
\label{eq:actual_stopping_criterion}
    \left(
    r_J^j<\varepsilon_J
    \quad\text{and}\quad
    r_f^j<\varepsilon_f
    \right)
    \qquad\text{or}\qquad
    \|\mathcal G_{\tau_j}(f^j)\|_{B,2}
    <
    \varepsilon_{\rm PG}.
\end{equation}
A failed line search is recorded separately and is not interpreted as satisfaction of the first-order optimality condition.

\subsection{Diagnostics}
\label{sec:diagnostics}

The main terminal diagnostic is the relative final error
\begin{equation}
\label{eq:relative_final_error}
    E_T
    =
    \frac{
    \|u(T)-u_T\|_{L^2(\Omega)}
    }{
    \|u_T\|_{L^2(\Omega)}
    }.
\end{equation}
This quantity measures the accuracy with which the controlled trajectory reaches the prescribed terminal target.

We also monitor the terminal mass discrepancy
\begin{equation}
\label{eq:relative_mass_error}
    E_M
    =
    \frac{
    \left|
    \displaystyle\int_\Omega u(T)\,dx
    -
    \displaystyle\int_\Omega u_T\,dx
    \right|
    }{
    \left|
    \displaystyle\int_\Omega u_T\,dx
    \right|
}.
\end{equation}
In the numerical implementation, the denominators in \eqref{eq:relative_final_error} and \eqref{eq:relative_mass_error} are replaced
by their maximum with a small positive threshold in order to avoid division by
zero.

The lower and upper active-set rates of the control are defined by
\begin{equation}
\label{eq:saturation_rates}
    S_-
    =
    \frac{
    \#\{(m,\ell,k): f_{m,\ell}^k \simeq \xi_1\}
    }{
    N_xN_yN
    },
    \qquad
    S_+
    =
    \frac{
    \#\{(m,\ell,k): f_{m,\ell}^k \simeq \xi_2\}
    }{
    N_xN_yN
    }.
\end{equation}
Here \(f_{m,\ell}^k \simeq \xi_i\) means that the value of the control is equal
to the bound \(\xi_i\) up to the numerical tolerance used in the implementation.
These quantities indicate whether the optimal control is mostly interior or
whether it saturates one of the constraints.

We also record the taxis contribution
\begin{equation}
\label{eq:taxis_gradient_diagnostic}
    Q_{\rm taxis}
    =
    \int_0^T
    \int_\Omega
    |f u\nabla v|^2\,dx\,dt.
\end{equation}
In the discrete computation, this quantity is approximated using the B-spline
gradient \(\nabla_B\) and the quadrature \(\mathcal Q_B\) associated with the
Greville collocation points.

In the conservative case \(r=\mu=0\), the continuous model satisfies the mass
balance
\begin{equation}
\label{eq:mass_balance_conservative}
    \frac{d}{dt}
    \int_\Omega u(x,t)\,dx
    =
    0.
\end{equation}
We therefore monitor the discrete total mass
\begin{equation}
\label{eq:discrete_mass}
    M^k
    =
    \mathcal Q_B(u^k),
\end{equation}
together with the absolute and relative mass drifts
\begin{equation}
\label{eq:conservative_mass_diagnostics}
    M^k-M^0,
    \qquad
    \frac{|M^k-M^0|}{|M^0|}.
\end{equation}
The discrete derivative of \(M^k\) is also plotted in order to check the numerical counterpart of \eqref{eq:mass_balance_conservative}.

When the optional mass rescaling is activated, these diagnostics describe the
post-correction mass. The intrinsic mass drift of the uncorrected B-spline scheme should therefore be assessed either before the rescaling or in a separate computation without mass correction.

In the logistic case, the mass is not conserved. Instead, it satisfies
\begin{equation}
\label{eq:mass_balance_logistic}
    \frac{d}{dt}
    \int_\Omega u(x,t)\,dx
    =
    \int_\Omega
    \left(
    ru-\mu u^2
    \right)\,dx .
\end{equation}
Accordingly, we compute the discrete reaction contribution
\begin{equation}
\label{eq:discrete_logistic_reaction_integral}
    R^k
    =
    \mathcal Q_B
    \left(
    ru^k-\mu(u^k)^2
    \right).
\end{equation}
The numerical diagnostic compares the discrete derivative of the computed mass
with \(R^k\). For the Runge--Kutta implementation, the predicted mass increment
is evaluated with the same stage weights as the state update. For the IMEX
Euler implementation, the reaction contribution is evaluated at the explicit
time level.

Overall, the numerical diagnostics report the terminal error, the terminal mass
discrepancy, the active-set rates of the control, the taxis contribution, the
mass-balance behavior and the evolution of the optimization quantities such as
the reduced cost and the projected-gradient norm.

\section{Numerical examples}
\label{sec:numerical_examples}

We present two numerical experiments illustrating the action of the optimized
taxis sensitivity. The first one concerns the conservative regime
$r=\mu=0$, whereas the second one includes logistic growth and saturation,
$r>0$ and $\mu>0$. The state system, the cost functional and the numerical
optimization procedure have already been introduced and are not repeated here.

In both experiments, the spatial discretization uses tensor-product cubic
B-splines on a $35\times35$ Greville collocation grid. The desired trajectory
is the smooth interpolation
\begin{equation}
\label{eq:desired_path_numerical_examples}
    u_d(t)
    =
    \bigl(1-q(t/T)\bigr)u_0
    +
    q(t/T)u_T,
    \qquad
    q(s)=3s^2-2s^3.
\end{equation}
Thus, $u_d(0)=u_0$ and $u_d(T)=u_T$, while its time derivative vanishes at both endpoints.

The principal numerical parameters are summarized in
Table~\ref{tab:numerical_parameters}. To limit redundancy, the main text
retains the state evolution, the terminal comparison, the control structure,
the mechanism and optimization diagnostics, and the mass-balance verification.

\begin{table}[H]
    \centering
    \small
    \begin{tabular}{lcc}
        \hline
        Parameter & Conservative case & Logistic case\\
        \hline
        B-spline degree & $3$ & $3$\\
        Collocation grid & $35\times35$ & $35\times35$\\
        Final time $T$ & $7$ & $10$\\
        Number of time intervals & $300$ & $220$\\
        Time integrator & IMEX Euler & explicit RK4\\
        Control bounds & $[0,0.90]$ & $[0,1.10]$\\
        $(D_u,D_v)$ & $(1.8\times10^{-2},3.5\times10^{-2})$
                     & $(1.45\times10^{-2},4.0\times10^{-2})$\\
        $(\alpha,\beta)$ & $(1.2\times10^{-1},1.0\times10^{-1})$
                         & $(8.0\times10^{-2},1.45\times10^{-1})$\\
        $(r,\mu)$ & $(0,0)$ & $(8.0\times10^{-2},2.0\times10^{-1})$\\
        $\gamma_T$ & $8.0\times10^2$ & $4.0\times10^3$\\
        $\gamma_u$ & $1.0\times10^{-3}$ & $2.0\times10^{-4}$\\
        $\gamma_f$ & $8.0\times10^{-7}$ & $2.0\times10^{-8}$\\
        $\gamma_{\rm taxis}$ & $2.0\times10^{-5}$ & $6.0\times10^{-6}$\\
        \hline
    \end{tabular}
    \caption{Main parameters used in the conservative and logistic experiments.
    The choice of time discretization is practical rather than comparative:
    IMEX Euler is used in the conservative case, whereas RK4 is used in the
    logistic case with internal substeps to respect the explicit diffusion
    stability restriction.}
    \label{tab:numerical_parameters}
\end{table}

\subsection{Conservative experiment}
\label{subsec:example1_conservative}

In the conservative regime, the total mass of the population is invariant. The independently prescribed target profile is therefore normalized according to
\begin{equation}
\label{eq:target_rescaling_example1}
    u_T
    =
    \frac{\displaystyle\int_\Omega u_0\,dx}
         {\displaystyle\int_\Omega u_T^{\rm raw}\,dx}
    u_T^{\rm raw}.
\end{equation}
The optional mass correction described in
\eqref{eq:optional_mass_correction} is activated in this computation.

Figure~\ref{fig:ex1_state_evolution} shows the controlled, uncontrolled and
desired evolutions at selected times. The uncontrolled density becomes
progressively flatter because of diffusion, whereas the controlled solution
remains more localized and is guided toward the target region.

\begin{figure}[H]
    \centering
    \begin{subfigure}[t]{0.92\textwidth}
        \centering
        \includegraphics[width=\linewidth]
        {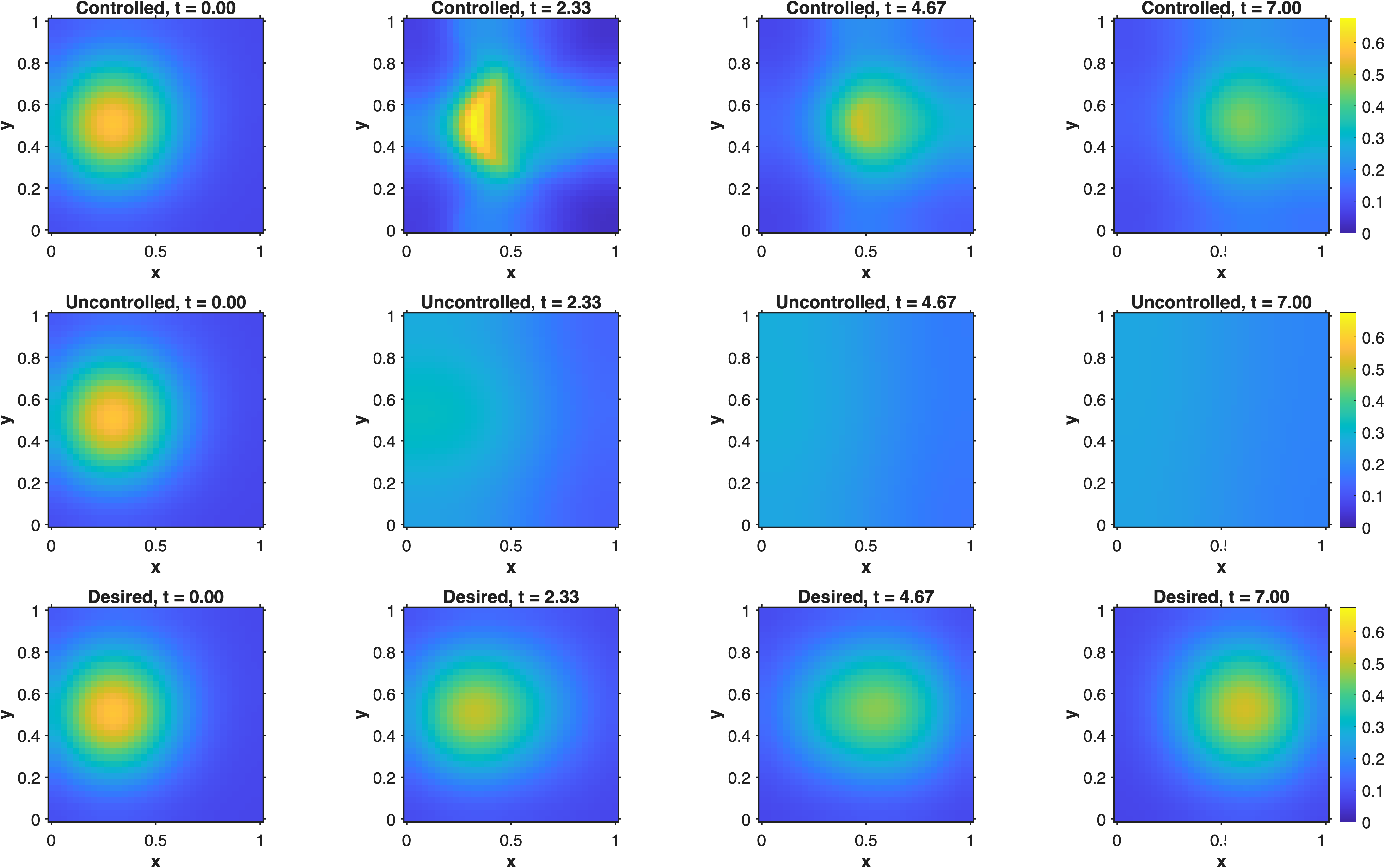}
        \caption{Controlled, uncontrolled and desired density profiles at selected
        times.}
        \label{fig:ex1_state_evolution_sub}
    \end{subfigure}
    \caption{Conservative experiment: state evolution.}
    \label{fig:ex1_state_evolution}
\end{figure}

Figure~\ref{fig:ex1_state_control} summarizes the principal effect of the control. Without control, diffusion produces an almost spatially uniform final density and the relative terminal error is
\[
    E_T^{\rm free}=5.10\times10^{-1}.
\]
The optimized taxis sensitivity preserves a localized density and moves its
maximum toward the target region, reducing the error to
\[
    E_T^{\rm ctrl}=1.66\times10^{-1}.
\]
This represents a reduction of approximately $67.5\%$.

The control is strongly shaped by the box constraints. Its support and intensity vary substantially in time, while the time-averaged map identifies the regions where chemotactic sensitivity is repeatedly activated.

\begin{figure}[H]
    \centering
    
    \begin{subfigure}[t]{0.82\textwidth}
        \centering
        \includegraphics[width=\linewidth]
        {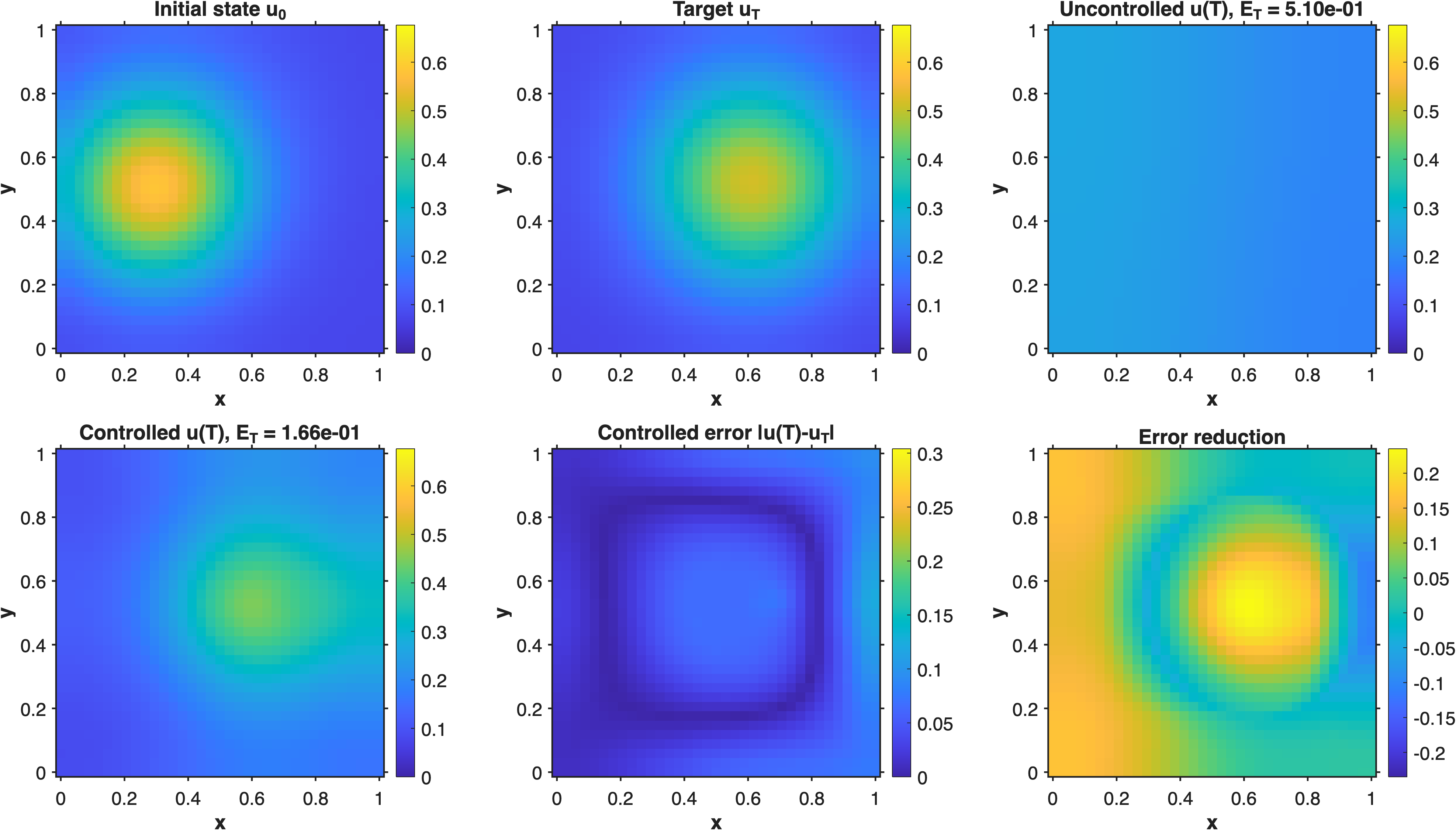}
        \caption{Initial state, target, free and controlled terminal states, terminal
        error and error reduction.}
        \label{fig:ex1_final_summary}
    \end{subfigure}
    
    \vspace{0.5em}
    
    \begin{subfigure}[t]{0.82\textwidth}
        \centering
        \includegraphics[width=\linewidth]
        {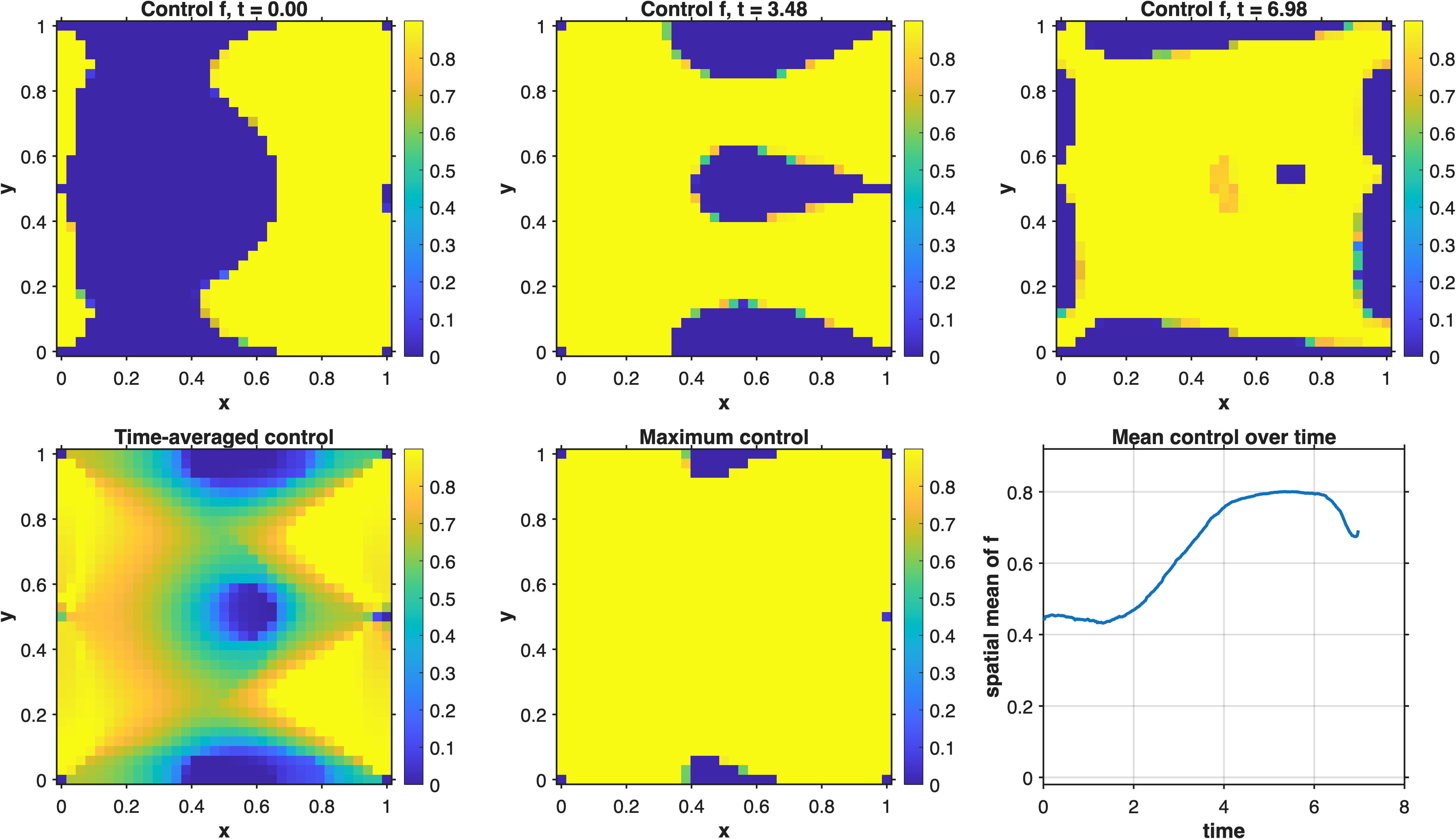}
        \caption{Control snapshots, time-averaged control, pointwise maximum and
        spatial mean over time.}
        \label{fig:ex1_control_structure}
    \end{subfigure}
    
    \caption{Conservative experiment: terminal action of the optimized taxis
    control and structure of the resulting admissible control.}
    \label{fig:ex1_state_control}
\end{figure}

The mechanism and numerical diagnostics are collected in Figure~\ref{fig:ex1_diagnostics}. The taxis contribution is concentrated where the chemical gradient and the controlled population overlap. The final midline profile confirms that the controlled state is substantially closer to the target than the free state.

The reduced cost and terminal error decrease throughout the optimization iterations. The projected-gradient residual, however, does not reach the
prescribed stationarity tolerance in the displayed run. The result should therefore be interpreted as a substantially improved admissible control rather than as a numerically certified stationary point.

The total mass is constant up to machine precision. Since the optional mass rescaling is active, this figure verifies the post-correction conservation implemented in the algorithm; it does not independently measure the intrinsic mass conservation error of the uncorrected collocation scheme.

\begin{figure}[H]
    \centering
    
    \begin{subfigure}[t]{0.90\textwidth}
        \centering
        \includegraphics[width=\linewidth]
        {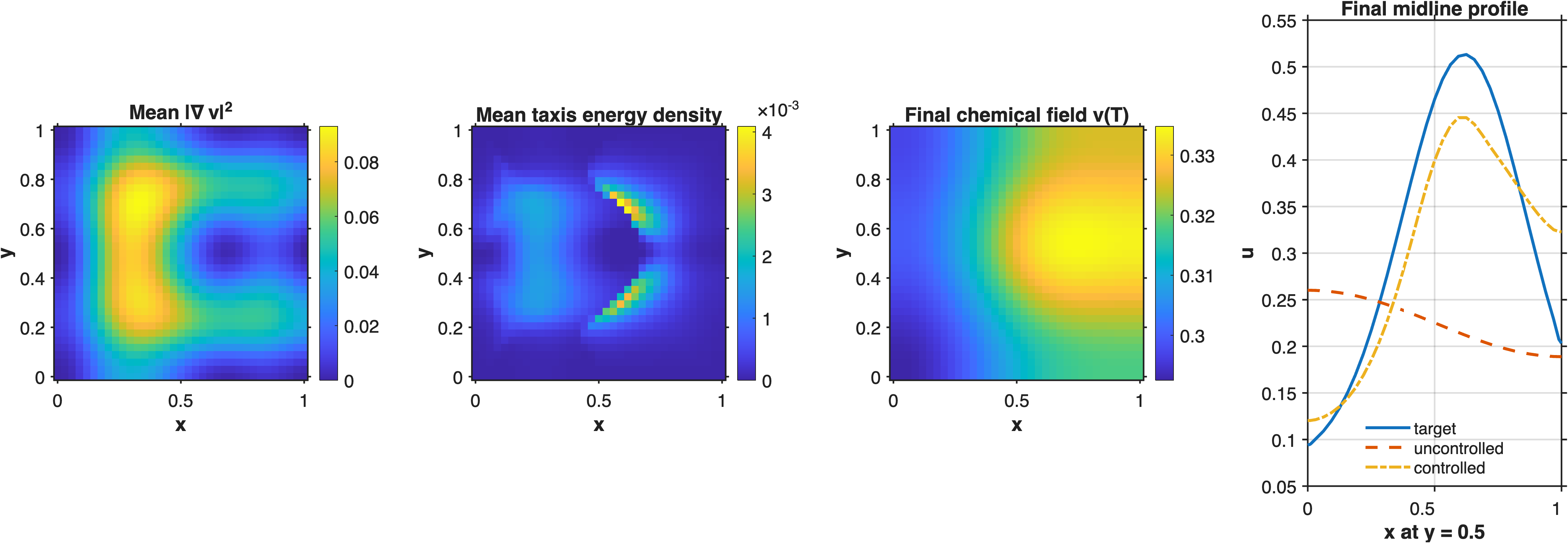}
        \caption{Mean chemical-gradient intensity, mean taxis energy density, final
        chemical field and final midline profile.}
        \label{fig:ex1_mechanism}
    \end{subfigure}
    
    \vspace{0.5em}
    
    \begin{subfigure}[t]{0.49\textwidth}
        \centering
        \includegraphics[width=\linewidth]
        {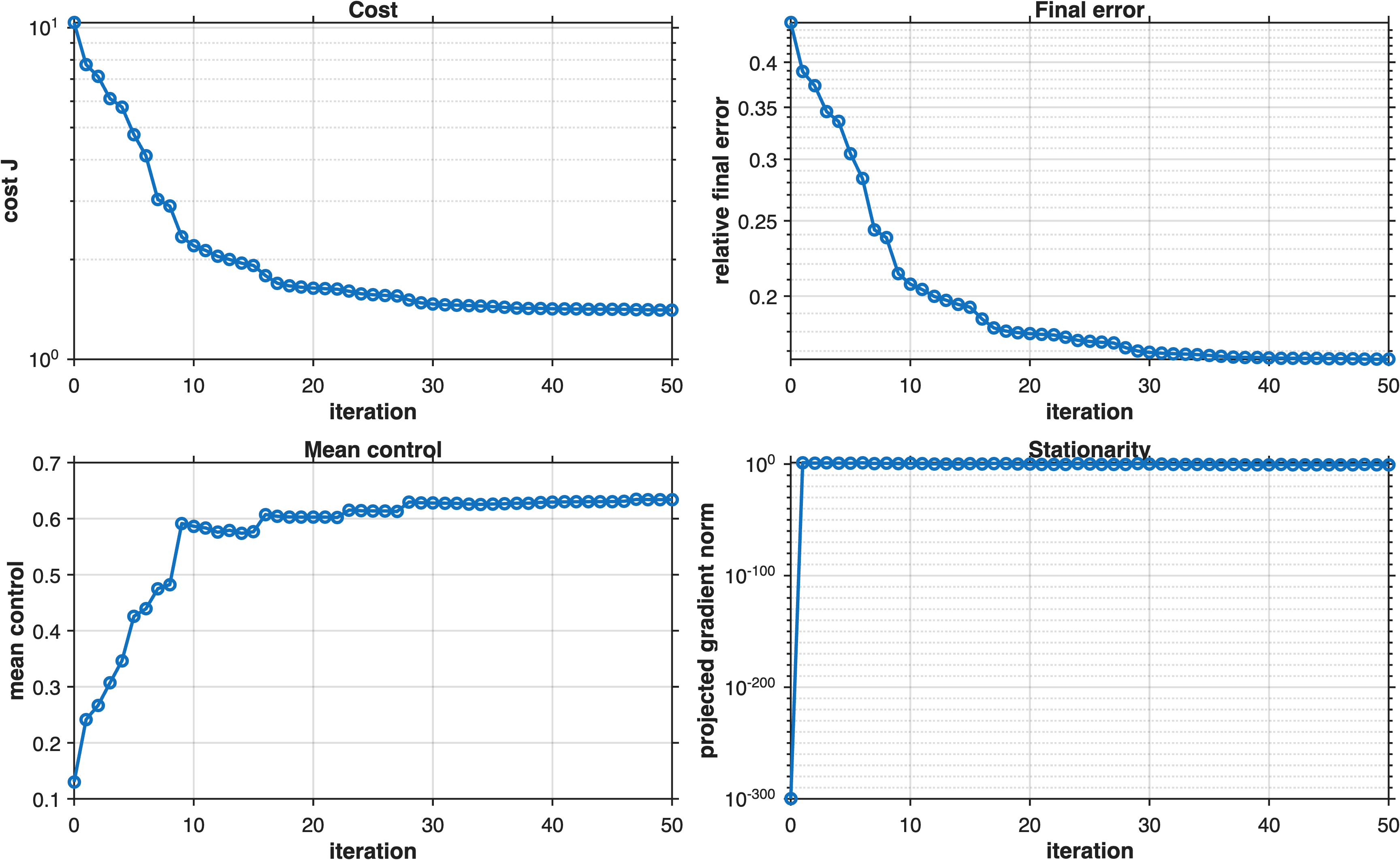}
        \caption{Reduced cost, final error, mean control and projected-gradient
        residual.}
        \label{fig:ex1_optimization}
    \end{subfigure}
    \hfill
    \begin{subfigure}[t]{0.49\textwidth}
        \centering
        \includegraphics[width=\linewidth]
        {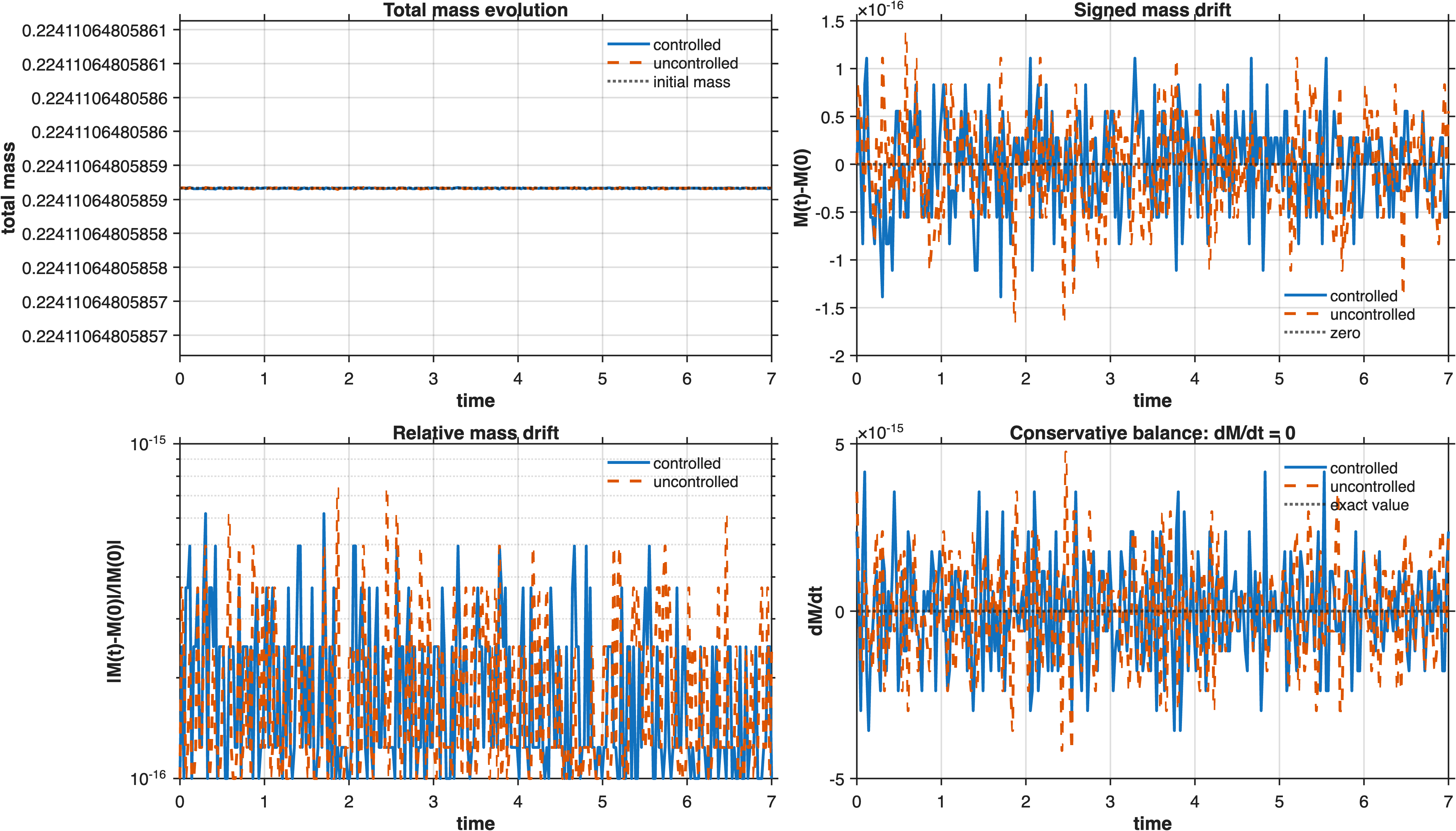}
        \caption{Total mass, absolute and relative mass drift, and discrete mass
        derivative.}
        \label{fig:ex1_mass}
    \end{subfigure}
    
    \caption{Conservative experiment: taxis mechanism, optimization history and
    mass-conservation diagnostics.}
    \label{fig:ex1_diagnostics}
\end{figure}

\FloatBarrier

\subsection{Logistic experiment}
\label{subsec:example2_logistic}

In the logistic regime, the total population mass evolves according to
\[
    \frac{d}{dt}\int_\Omega u(x,t)\,dx
    =
    \int_\Omega\bigl(ru-\mu u^2\bigr)\,dx.
\]
The target shape remains independently prescribed, but its mass is rescaled to the terminal mass of the uncontrolled logistic trajectory. This avoids an artificial mass mismatch without constructing the target from a reference control.

Figure~\ref{fig:ex2_state_evolution} shows the controlled, uncontrolled and desired evolutions in the logistic regime. Here the density is affected by the combined action of diffusion, taxis transport and logistic growth-saturation.
The optimized control creates a localized terminal distribution, whereas the uncontrolled solution becomes comparatively flatter.

\begin{figure}[H]
    \centering
    \begin{subfigure}[t]{0.92\textwidth}
        \centering
        \includegraphics[width=\linewidth]
        {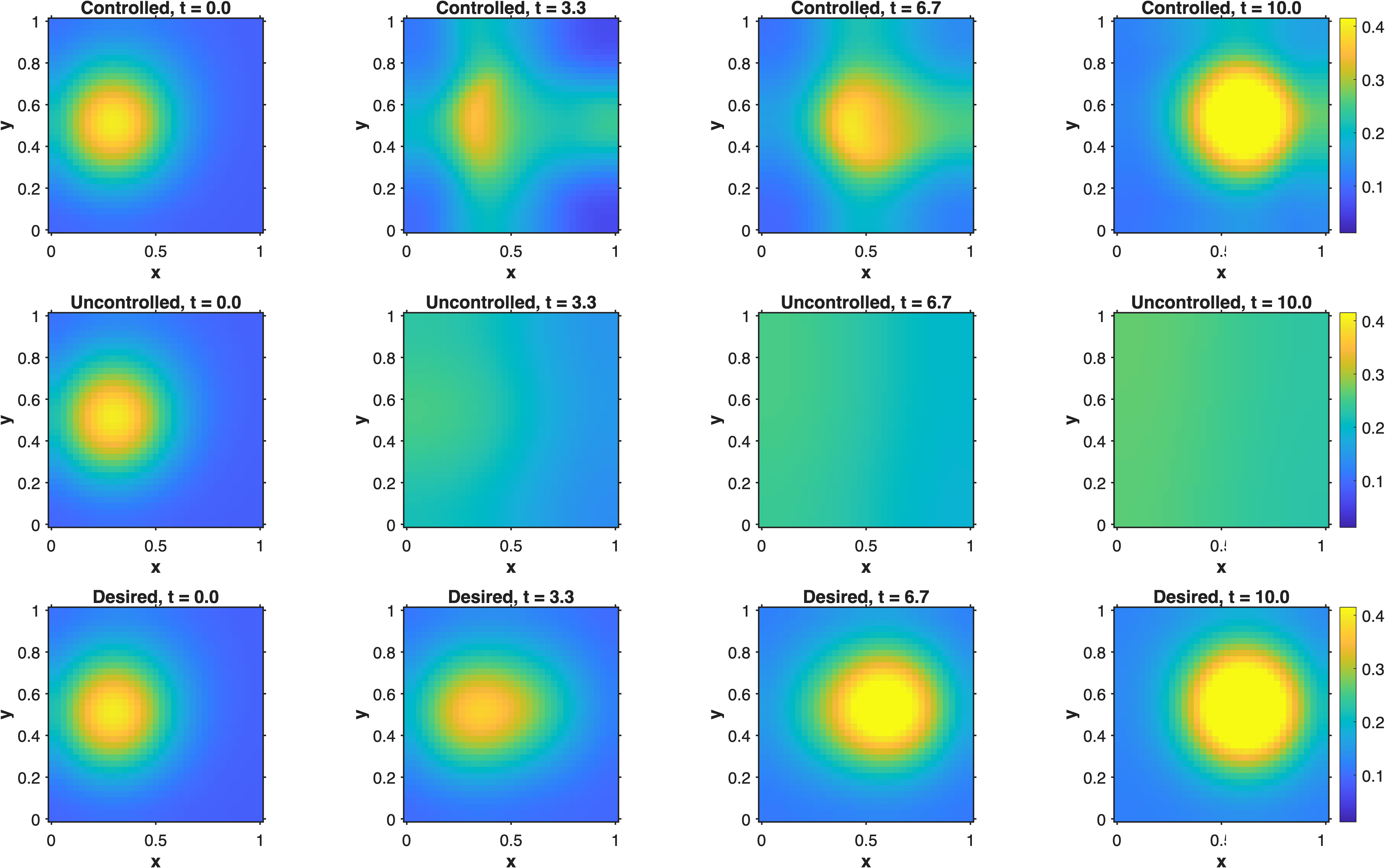}
        \caption{Controlled, uncontrolled and desired density profiles at selected
        times.}
        \label{fig:ex2_state_evolution_sub}
    \end{subfigure}
    \caption{Logistic experiment: state evolution.}
    \label{fig:ex2_state_evolution}
\end{figure}

Figure~\ref{fig:ex2_state_control} shows a stronger terminal improvement than in the conservative experiment. The uncontrolled solution becomes nearly spatially homogeneous and has relative terminal error
\[
    E_T^{\rm free}=4.51\times10^{-1}.
\]
The optimized control produces a localized final density close to the target and reduces the error to
\[
    E_T^{\rm ctrl}=6.10\times10^{-2},
\]
corresponding to a reduction of approximately $86.5\%$.

As in the conservative case, the computed control has a pronounced active-set structure. Its spatial mean first decreases, then increases during the main steering phase, and finally relaxes near the terminal time.

\begin{figure}[H]
    \centering
    
    \begin{subfigure}[t]{0.82\textwidth}
        \centering
        \includegraphics[width=\linewidth]
        {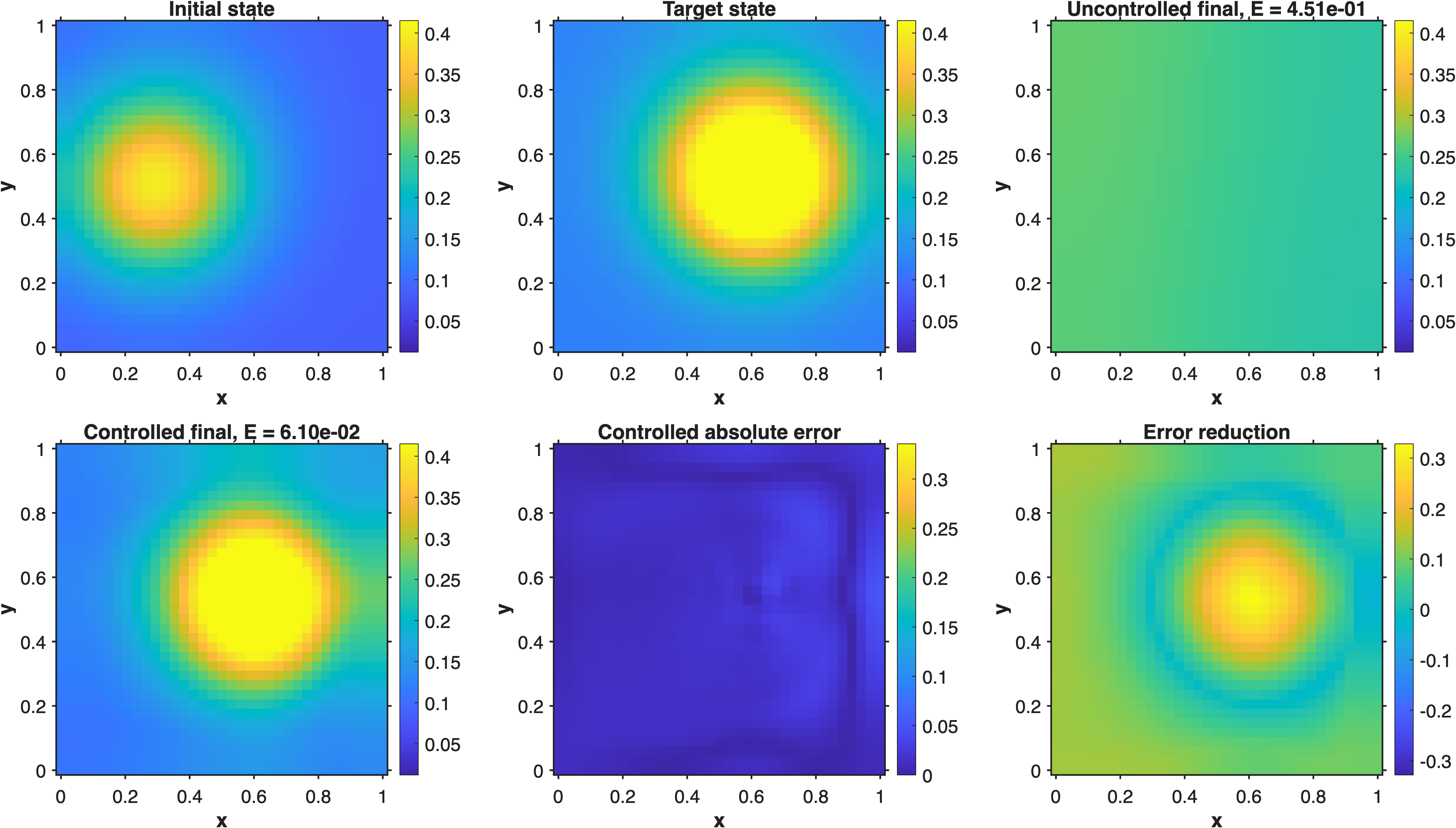}
        \caption{Initial state, target, free and controlled terminal states, terminal
        error and error reduction.}
        \label{fig:ex2_final_summary}
    \end{subfigure}
    
    \vspace{0.5em}
    
    \begin{subfigure}[t]{0.82\textwidth}
        \centering
        \includegraphics[width=\linewidth]
        {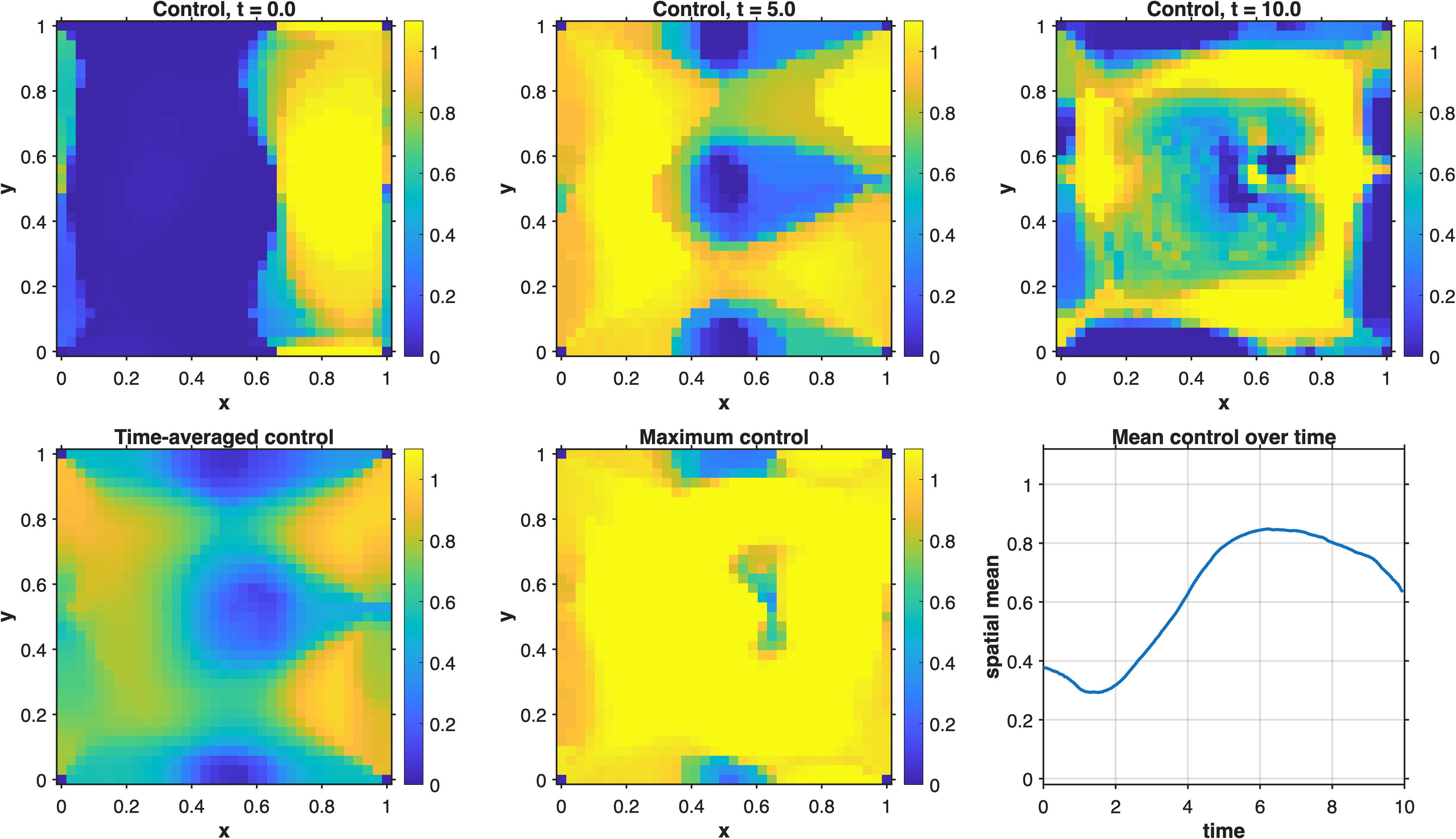}
        \caption{Control snapshots, time-averaged control, pointwise maximum and
        spatial mean over time.}
        \label{fig:ex2_control_structure}
    \end{subfigure}
    
    \caption{Logistic experiment: terminal action of the optimized taxis control
    and structure of the resulting admissible control.}
    \label{fig:ex2_state_control}
\end{figure}

Figure~\ref{fig:ex2_diagnostics} combines the optimization, taxis and mass-balance information. The mean taxis density is concentrated near the right-central region where the controlled density is being formed. The final midline profile is very close to the target, consistently with the small terminal error.

The reduced cost and terminal error decrease over the optimization iterations.
As in the conservative computation, the displayed projected-gradient residual does not provide a strict stationarity certificate; the numerical result is therefore reported as a locally improved admissible control.

The controlled and uncontrolled mass curves differ because the control changes the spatial distribution of $u$ and therefore modifies \(\int_\Omega u^2\,dx\). In particular, matching the target mass with the uncontrolled terminal mass does not force the controlled terminal mass to be
the same. The computed and reaction-predicted mass curves remain close, and the comparison between the numerical mass derivative and the reaction integral confirms the expected logistic balance. The larger residual in the controlled computation reflects the stronger spatial dynamics and remains small relative to the overall mass scale.

\begin{figure}[H]
    \centering
    
    \begin{subfigure}[t]{0.82\textwidth}
        \centering
        \includegraphics[width=\linewidth]
        {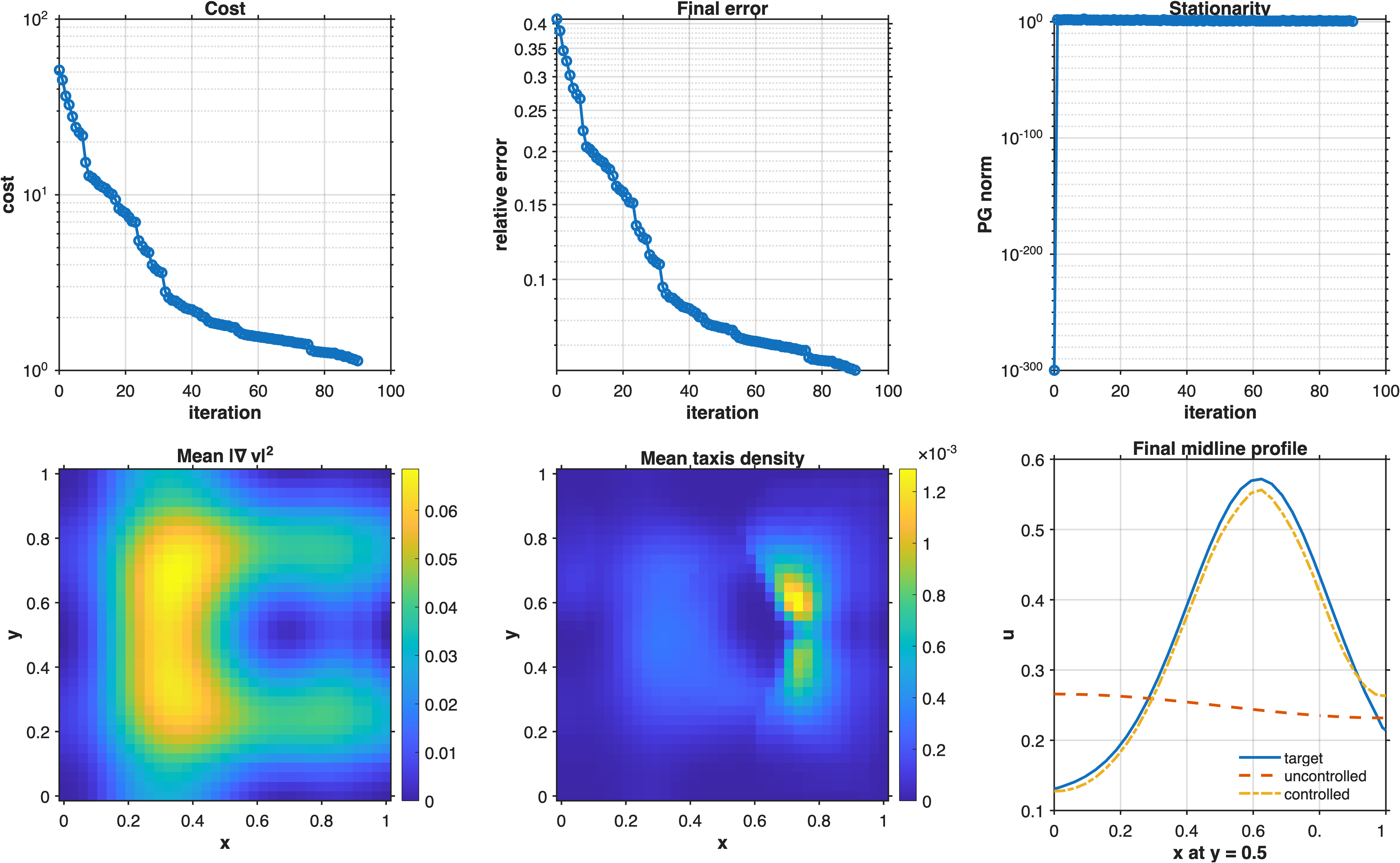}
        \caption{Cost and terminal-error histories, projected-gradient residual, mean
        chemical-gradient intensity, mean taxis density and final midline profile.}
        \label{fig:ex2_optimization_mechanism}
    \end{subfigure}
    
    \vspace{0.5em}
    
    \begin{subfigure}[t]{0.82\textwidth}
        \centering
        \includegraphics[width=\linewidth]
        {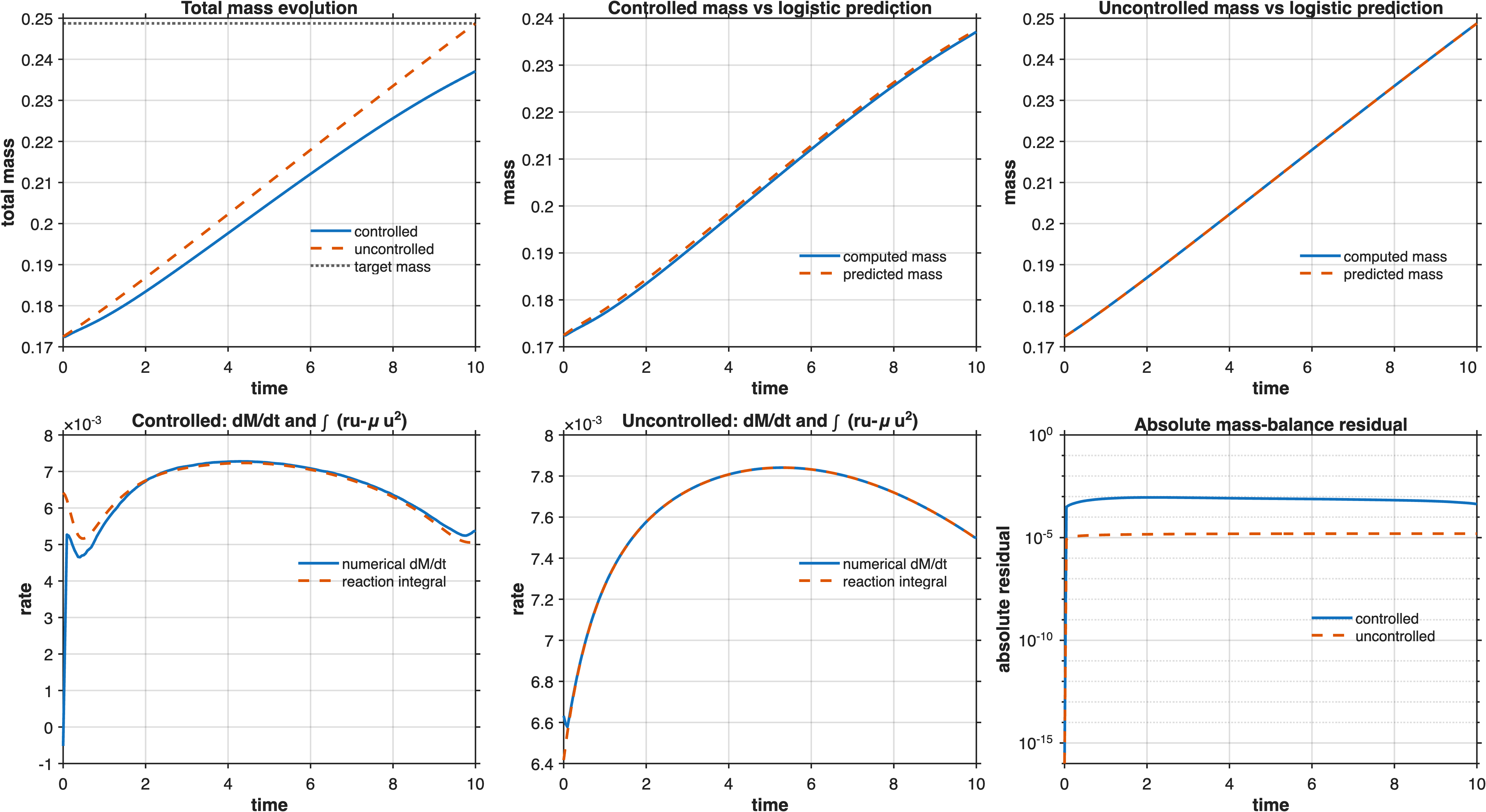}
        \caption{Computed and predicted mass curves, comparison of the mass derivative
        with the reaction integral, and absolute mass-balance residual.}
        \label{fig:ex2_mass_balance}
    \end{subfigure}
    
    \caption{Logistic experiment: optimization, taxis mechanism and numerical
    verification of the reaction-driven mass balance.}
    \label{fig:ex2_diagnostics}
\end{figure}

The two experiments show that the optimized taxis sensitivity can substantially improve terminal tracking in both the conservative and logistic regimes. The control does not act as an external source: it modifies the
transport induced by the chemical gradient. Consequently, the attainable terminal profile depends on the chemical field, the box constraints and, in the logistic case, the reaction-driven mass evolution. The numerical results support the effectiveness of the computed controls without asserting global optimality.

\FloatBarrier

\section{Acknowledgments}

The authors are grateful to Enrique Zuazua for his interest in this work, helpful discussions, and encouragement throughout the development of the paper.

Y.S. and L.Y. acknowledge that this work was carried out during a research stay at the Chair for Dynamics, Control, Machine Learning and Numerics, Friedrich-Alexander-Universität Erlangen–Nürnberg. Y.S. was supported by the Alexander von Humboldt Foundation through an Alexander von Humboldt Fellowship and L.Y. was supported by CAPES-Brazil and partially by the Alexander von Humboldt Foundation as  visiting scientist.

\appendix

\section{ Continuity of the map $\mathcal R$ given in \eqref{R}}\label{appendix_a}

Let $(\bar u_{n})$ be a sequence in \(S_{M}^{(4)}\) (defined in \eqref{eq:def_S_M4}) and let $ \bar u\in S_{M}^{(4)} $ be such that $\bar u_{n}\to\bar u$ in $L^4(\Omega_T)$, as $n\to\infty$.  Moreover, we define $(u_n,v_n)$ such that \(\mathcal R(\bar u_{n})=u_n\) and $v_n$ the corresponding solution of the second system. The goal is to show that \(u_n=\mathcal R(\bar u_{n})\to \mathcal R(\bar u)=u\) in \(L^4(\Omega_T)\), as $n\to\infty$.  For this, let us define
\[
    a_n:=a_\varepsilon(\bar u_{n})\;\mbox{ and } \quad a:=a_\varepsilon(\bar u).
\]
Since $a_\varepsilon$ is bounded and Lipschitz continuous, we have that $a_n\to a$ in $L^4(\Omega_T)$ and weak-$*$ in $L^\infty(\Omega_T)$, as $n\to\infty$.
We notice that $(u_n,v_n)$ and $(u,v)$ are the solutions associated to $a_n$ and $a$, respectively. Let
\[
    w_n:=e^{-\gamma t}(u_n-u)\; \mbox{ and } \; z_n:=v_n-v.
\]
Then, $w_n(\cdot,0)=0$, $z_n(\cdot,0)=0$ in $\Omega$, we have the Neumann boundary conditions on $\Gamma_T$, and  on $\Omega_T$ we have that
\begin{align}
    &\partial_t w_n - D_u\Delta w_n + e^{-\gamma t} \nabla\!\cdot\!\big(f\,(a_n\nabla v_n - a\nabla v)\big)
    +\gamma w_n= r\,w_n - \mu \, e^{-\gamma t} (a_n u_n - a u), \label{diff-w}\\
    &\partial_t z_n - D_v\Delta z_n + \alpha z_n = \beta\,(a_n-a). \label{diff-z}
\end{align}

We further split
\[
    a_n\nabla v_n - a\nabla v = (a_n-a)\,\nabla v_n + a\,\nabla z_n,
    \quad
    a_n u_n - a u = (a_n-a)\,u_n + a\, e^{\gamma t} w_n.
\]

Testing \eqref{diff-w} with $w_n$ and integrating by parts in the drift term, we get the following:
\begin{align*}
    &\frac12\frac{d}{dt}\|w_n\|_{L^2(\Omega)}^2 + D_u\|\nabla w_n\|_{L^2(\Omega)}^2+\gamma \|w_n\|_{L^2(\Omega)}^2 \\
    =& \int_\Omega e^{-\gamma t} f\,(a_n\nabla v_n - a\nabla v)\cdot\nabla w_n\;dx
    + r\|w_n\|_{L^2(\Omega)}^2 
    - \mu\!\int_\Omega e^{-\gamma t} (a_n u_n - a u)\,w_n\;dx.
\end{align*}
We introduce the explicit pieces
\[
    I_1:=\int_\Omega |f|\,|a_n-a|\,|\nabla v_n|\,|\nabla w_n|\;dx,\quad
    I_2:=\int_\Omega |f|\,|a|\,|\nabla z_n|\,|\nabla w_n|\;dx,
\]
\[
    J_1:=\int_\Omega |a_n-a|\,|u_n|\,|w_n|\;dx,\quad
    J_2:=\int_\Omega a\,|w_n|^2\;dx.
\]
Then,
\[
    \frac12\frac{d}{dt}\|w_n\|_{L^2(\Omega)}^2 + D_u\|\nabla w_n\|_{L^2(\Omega)}^2 +\gamma \|w_n\|_{L^2(\Omega)}^2
    \le I_1 + I_2 + r\|w_n\|_{L^2(\Omega)}^2 + \mu J_1 + \mu J_2.
\]

Since $(v_n)_n$ is uniformly bounded in $X_4$, we have
$\sup_n\|\nabla v_n\|_{L^\infty(0,T;L^4(\Omega))}\leq C.$
Hence, using H\"older's and Young's inequalities, we obtain
\begin{align*}
    I_1
    &\leq
    \|f\|_{L^\infty(\Omega_T)}
    \|a_n-a\|_{L^4(\Omega)}
    \|\nabla v_n\|_{L^4(\Omega)}
    \|\nabla w_n\|_{L^2(\Omega)}
    \\
    &\leq
    \frac{D_u}{6}\|\nabla w_n\|_{L^2(\Omega)}^2
    +
    C\|a_n-a\|_{L^4(\Omega)}^2.
\end{align*}
Since $a\in L^\infty(\Omega_T)$, it follows that
\begin{align*}
    I_2 \le &\|f\|_{L^\infty(\Omega)}\|a\|_{L^\infty(\Omega}\,\|\nabla z_n\|_{L^2(\Omega)}\,\|\nabla w_n\|_{L^2(\Omega)}\\
    \le &\frac{D_u}{6}\|\nabla w_n\|_{L^2(\Omega)}^2 + C\,\|\nabla z_n\|_{L^2(\Omega)}^2.
\end{align*}
Moreover,
\[
    \mu J_2 = \mu\int_\Omega a\,|w_n|^2\;dx \le \mu\|a\|_{L^\infty(\Omega)}\,\|w_n\|_{L^2(\Omega)}^2.
\]
From the uniform estimate obtained in Section~\ref{sec:existence_auxiliary}, we have $\sup_n \|u_n\|_{L^\infty(0,T;L^4(\Omega))}\leq C.$
Hence,
\begin{align*}
    J_1 \le &\|a_n-a\|_{L^4(\Omega)}\,\|u_n\|_{L^4(\Omega)}\,\|w_n\|_{L^2(\Omega)}\\
    \le &C\,\|a_n-a\|_{L^4(\Omega)}\,\|w_n\|_{L^2(\Omega)}\\
    \le& \tfrac12\|w_n\|_{L^2(\Omega)}^2 + C\,\|a_n-a\|_{L^4(\Omega)}^2.
\end{align*}

Gathering all the bounds, for a.e.\ $t\in(0,T)$, we obtain that
\begin{align}\label{w-ineq}
    &\frac{d}{dt}\|w_n\|_{L^2(\Omega)}^2 + \frac{2D_u}{3}\|\nabla w_n\|_{L^2(\Omega)}^2+\gamma \|w_n\|_{L^2(\Omega)}^2 \notag\\
    \le& C\,\|w_n\|_{L^2(\Omega)}^2 + C\,\|a_n-a\|_{L^2(\Omega)}^2 + C\,\|a_n-a\|_{L^4(\Omega)}^2 + C\,\|\nabla z_n\|_{L^2(\Omega)}^2.
\end{align}
Since $|\Omega|<\infty$, we have that $\|a_n-a\|_{L^2(\Omega)}\le |\Omega|^{\frac12-\frac14}\,\|a_n-a\|_{L^4(\Omega)}$. We may keep only the $L^4$-term.

Testing \eqref{diff-z} with $z_n$, we get
\begin{align*}
    &\frac12\frac{d}{dt}\|z_n\|_{L^2(\Omega)}^2 + D_v\|\nabla z_n\|_{L^2(\Omega)}^2 + \alpha\|z_n\|_{L^2(\Omega)}^2\\
    = &\beta\int_\Omega (a_n-a )z_n\;dx
    \le \frac{\alpha}{2}\|z_n\|_{L^2(\Omega)}^2 + \frac{\beta^2}{2\alpha}\|a_n-a\|_{L^2(\Omega)}^2.
\end{align*}
Thus,
\begin{align}\label{zn-grad}
    &\frac{d}{dt}\|z_n\|_{L^2(\Omega)}^2 + 2D_v\|\nabla z_n\|_{L^2(\Omega)}^2 + \alpha\|z_n\|_{L^2(\Omega)}^2
    \le \frac{\beta^2}{\alpha}\|a_n-a\|_{L^2(\Omega)}^2,\notag\\
    &\Longrightarrow\quad
    \int_0^t\|\nabla z_n\|_{L^2(\Omega)}^2\;dt \le C\int_0^t\|a_n-a\|_{L^4(\Omega)}^2\;dt.
\end{align}
Integrating \eqref{w-ineq} over $(0,t)$ and using \eqref{zn-grad} as well as $w_n(0)=0$, we get
\begin{align*}
    &\|w_n(t)\|_{L^2(\Omega)}^2 + \frac{2D_u}{3}\int_0^t\|\nabla w_n\|_{L^2(\Omega)}^2\;dt+(\gamma-C)\int_0^t\|w_n\|_{L^2(\Omega)}^2\;dt \\
    \le& C\int_0^t\|a_n-a\|_{L^2(\Omega)}^2 \;dt+ C\int_0^t\|a_n-a\|_{L^4(\Omega)}^2\;dt.
\end{align*}
Choosing $\gamma=C+1,$ we get
\begin{equation}\label{w-strong}
    \sup_{t\in[0,T]}\|w_n(t)\|_{L^2(\Omega)}^2 + \int_0^T\|\nabla w_n\|_{L^2(\Omega)}^2\;dt
    +\int_0^T \|w_n\|_{L^2(\Omega)}^2 \;dt\le C\,\|a_n-a\|_{L^4(\Omega_T)}^2 \xrightarrow[n\to\infty]{} 0.
\end{equation}
From \eqref{w-strong}, we have that $w_n\to0$ in $L^2(0,T;H^1(\Omega))$ and in $L^\infty(0,T;L^2(\Omega))$, as $n\to\infty$.
By the 2D Sobolev embedding, for each $t$, $\|w_n(t)\|_{L^4(\Omega)}\le C\|w_n(t)\|_{H^1(\Omega)}$. Hence,
\[
    \int_0^T \|w_n(t)\|_{L^4(\Omega)}^{2}\,dt \le C \int_0^T \|w_n(t)\|_{H^1(\Omega)}^{2}\,dt \xrightarrow[n\to\infty]{}0.
\]
Moreover, by the uniform bound of $u_n$ and the fact that $u\in L^\infty(0,T;L^4(\Omega))$, we obtain that
\begin{align*}
    \|w_n\|_{L^4(\Omega_T)}^{4}=\int_0^T \|w_n(t)\|_{L^4(\Omega)}^{4}\;dt
    \le &\Big(\sup_{t\in[0,T]}\|w_n(t)\|_{L^4(\Omega)}^{\,2}\Big)\,\int_0^T \|w_n(t)\|_{L^4(\Omega)}^{2}\;dt\\
    \le &C\,\int_0^T \|w_n(t)\|_{L^4(\Omega)}^{2}\;dt\xrightarrow[n\to\infty]{}0.
\end{align*}
Therefore, $w_n\to0$ in $L^4(\Omega_T)$, as $n\to\infty$. That is, as $n\to\infty$,
\[
    \mathcal R(\bar u_{n}) = u_n \longrightarrow u = \mathcal R(\bar u)\quad\text{in }L^4(\Omega_T),
\]
and we have shown that $\mathcal R$ is continuous.

\section{Fréchet differentiability of the functional $J$}
\label{appendix_b_J}

\begin{lemma}\label{lem:der_J}
    The functional $J:\mathbb{X}\to \R$ is Fréchet differentiable and its Fréchet derivative at $s=(u, v, f) \in \mathbb{X}$ in the direction $z=(U,V,F) \in \mathbb{X}$ is given by
    \begin{align*}
        J'(s)[z] = &2 \gamma_u \int_{\Omega_T} ( u - u_d) U \ dx dt + 2 \gamma_{\text{taxis}} \int_{\Omega_T} (f u \nabla v) \cdot (fu\nabla V + fU \nabla v + F u \nabla v) \ dx dt \notag\\
        &+ 4 \gamma_f \int_{\Omega_T} f^3 F \ dx dt + 2 \gamma_T \int_\Omega (u(x,T)-u_T(x))U(x,T) \ dx.
    \end{align*}
\end{lemma}

\begin{proof}
    We derive each term of the functional $J$. Letting
    $\displaystyle J_1(u,v,f) := \int_0^T \| u(t) - u_d(t) \|^2_{L^2(\Omega)} \ dt$, a simple calculation gives, for every $\lambda\in (0,1]$,
    \begin{align*}
        &J_1(u+\lambda U, v+\lambda V, f+\lambda F) - J_1(u,v,f) \\
        &= \int_{\Omega_T} \left( 2 \lambda (u(x,t)-u_d(x,t)) U(x,t) + \lambda^2 U(x,t)^2 \right) dxdt. 
    \end{align*}
    This implies that
    $$
        \lim_{\lambda \to 0} \frac{J_1(u+\lambda U, v+\lambda V, f+\lambda F) - J_1(u,v,f)}{\lambda} = \int_{\Omega_T} 2 (u(x,t)-u_d(x,t)) U(x,t) \ dxdt.
    $$
    Next, letting
    $\displaystyle J_2(u,v,f) = \int_{\Omega_T} |f u \nabla v|^2 \ dxdt$,
    we have, for every $\lambda\in (0,1]$, 
    \begin{align*}
        &J_2(u+\lambda U, v+\lambda V, f+\lambda F) \\
        &= \int_{\Omega_T} \left(f u \nabla v + \lambda (f u \nabla V + f U \nabla v + Fu\nabla v)+\lambda^2 \mathcal{C}_1(u,v,f,U,V,F,\lambda) \right)^2 \ dxdt \\
        &= \int_{\Omega_T} \left( |fu\nabla v|^2 + 2\lambda ((fu\nabla v) \cdot (f u \nabla V + f U \nabla v + Fu\nabla v)) + \lambda^2 \mathcal{C}_2(u,v,f,U,V,F,\lambda) \right) dxdt, 
    \end{align*}
    where $\mathcal{C}_i$, $i=1,2$, are expressions depending on $(u,v,f,U,V,F,\lambda)$ and bounded when $\lambda \to 0$. Thus,
    $$
        \lim_{\lambda \to 0} \frac{J_2(u+\lambda U, v+\lambda V, f+\lambda F) - J_2(u,v,f)}{\lambda} = 2 \int_{\Omega_T} (fu\nabla v) \cdot (f u \nabla V + f U \nabla v + Fu\nabla v) \ dxdt.
    $$
     The other terms of $J'(s)[z]$ are computed analogously. This shows that $J$ is Gateaux differentiable at $s$. 
    
    To show that $J$ is Fréchet differentiable, first we observe that the functional $J$, with the exception of the term $J_2$, is a sum of terms given by $L^p$-norms of the variables. 
    Following the argument of \cite{Guillen-Vianna-23} and using the theory of Nemytskii operators \cite[Subsection 4.3.3.]{Troltzsch_book} in $L^p$-spaces, it is known that those terms are Fréchet differentiable.
    
    Now, we analyze the term $J_2'$ given by
    $$
        J_2'(s)[z]= 2 \int_{\Omega_T} \left( |f|^2 |u|^2 \nabla v \cdot \nabla V + |f|^2 u U |\nabla v|^2 + f F |u|^2 |\nabla v|^2 \right) dxdt =: J_{2,1}' + J_{2,2}' + J_{2,3}'.
    $$ 
    
    Let $(u_n,v_n,f_n)$ be a sequence in $\mathbb{X}$ converging to $(u,v,f) \in \mathbb{X}$, as $n\to\infty$. For $J_{2,1}'$ we have that
    \begin{equation*}
        \begin{split}
            &\Big|J_{2,1}'(u_n,v_n,f_n)[U,V,F]-J_{2,1}'(u,v,f)[U,V,F]\Big| \\
            &\quad \leq 2 \int_{\Omega_T} \left( |f_n - f| |f_n + f| |u|^2 |\nabla v| |\nabla V| + |f|^2 |u_n - u| |u_n + u| |\nabla v| |\nabla V| \right. \\
            &\quad\quad\quad\quad\quad \left. + |f|^2 |u|^2 |\nabla v_n - \nabla v| |\nabla V| \right) \ dxdt \\
            &\quad \leq \left( 2 \| f_n + f\|_{L^\infty(\Omega_T)} \|u\|_{L^4(\Omega_T)}^2 \|\nabla v\|_{L^4(\Omega_T)} \|\nabla V\|_{L^4(\Omega_T)} \right. \\
            &\left. \quad\quad + 2\|f\|_{L^\infty(\Omega_T)}^2 \|u_n + u\|_{L^4(\Omega_T)} \|\nabla v\|_{L^4(\Omega_T)} \|\nabla V \|_{L^4(\Omega_T)} \right. \\
            &\left. \quad\quad + 2 \|f\|_{L^\infty(\Omega_T)}^2 \|u\|_{L^4(\Omega_T)}^2  \|\nabla V \|_{L^4(\Omega_T)} \right) \|(u_n-u,v_n-v,f_n-f)\|_{\mathbb{X}},
        \end{split}
    \end{equation*}
    which converges to zero, as $n \to \infty$. The estimates for $J_{2,2}'$ and $J_{2,3}'$ follow similar. Thus, the map $s \mapsto J_2'(s)$ is continuous from $\mathbb{X} \to \mathcal{L}(\mathbb{X,}\R)$. Therefore, $J_2$, hence $J$, is Fréchet differentiable at $s$.
\end{proof}

\section{Fréchet differentiability of the map $G$}
\label{appendix_c_G}

\begin{lemma}\label{lem:der_G}
    The operator $G:\mathbb{X}\to \mathbb{Y}$ is Fréchet differentiable and its Fréchet derivative at  $s=(u,v,f) \in \mathbb{X}$ in the direction $z=(U,V,F) \in \mathbb{X}$ is given, for every $\varphi \in L^2(0,T;H^{1}(\Omega))$, by
    \begin{align*}
        \langle G_1'(s)[z],\varphi \rangle =& \langle \partial_t U,\varphi \rangle + D_u (\nabla U,\nabla \varphi)_{L^2(\Omega_T)} - (fu\nabla V + f U \nabla v + F u \nabla v, \nabla \varphi)_{L^2(\Omega_T)} \\
        &\qquad\qquad\qquad+ (-\lambda r U + 2 \mu u U,\varphi)_{L^2(\Omega_T)},  \\
        G_2'(s)[z] =& \partial_t V - D_v \Delta V + \alpha V - \beta U, \qquad \text{in } L^4(\Omega_T).
    \end{align*}
\end{lemma}

\begin{proof}
    
    For $\epsilon>0$ small and, if necessary, by modifying $\xi_i \in L^{\infty}(\Omega_T)$, $i=1,2$, in the definition of $\mathcal{F}$ in \eqref{eq:def_F}, we can ensure that $f + \epsilon F \in \mathcal{F}$. Then, we compute, from the definition of $G$ in \eqref{eq:def_G1_G2}, for $\varphi \in L^2(0,T;H^{1}(\Omega))$, 
    $$
        G_{1,\epsilon} :=\frac{1}{\epsilon}\langle G_1((u,v,f)+\epsilon(U,V,F))-G_1(u,v,f),\varphi \rangle.
    $$
    We get
    \begin{equation*}
        \begin{split}
            G_{1,\epsilon} =& \langle \partial_t U,\varphi \rangle + ( D_u \nabla U -  f U \nabla v - f u \nabla V - F u \nabla v - \epsilon (f U \nabla V + F u \nabla V + F U \nabla v), \nabla \varphi)_{L^2(\Omega_T)} \\
            &- \epsilon^2 (F U \nabla V, \nabla \varphi)_{L^2(\Omega_T)} - (r U - 2 \mu U u - \epsilon \mu U^2,\varphi)_{L^2(\Omega_T)}.
        \end{split}
    \end{equation*}
    
    Taking the limit, as $\epsilon \to 0$,  we obtain the Gateaux derivative of $G_1$ at $s=(u,v,f)$ in the direction $(U,V,F)$. On the other hand, the operator $G_2$ is linear and we can deduce that $G$ is Gateaux differentiable, with Gateaux derivative given, for any $\varphi \in L^2(0,T;H^{1}(\Omega))$, by
    \begin{align*}
        \langle G_1'(s)[z],\varphi \rangle =& \langle \partial_t U,\varphi \rangle + D_u (\nabla U,\nabla \varphi)_{L^2(\Omega_T)} - (fu\nabla V + f U \nabla v + F u \nabla v, \nabla \varphi)_{L^2(\Omega_T)} \\
        &+ (-\lambda r U + 2 \mu u U,\varphi)_{L^2(\Omega_T)}, \\
        G_2'(s)[z] =& \partial_t V - D_v \Delta V + \alpha V - \beta U, \quad \text{in } L^4(\Omega_T).
    \end{align*}
    
    To prove that $G$ is Fréchet differentiable, we show that the mapping $s \mapsto G'(s)$ is continuous from $\mathbb{X}$ into $\mathcal{L}(\mathbb{X},\mathbb{Y})$. 
    
    Let $(u_n,v_n,f_n)$ be a sequence in $\mathbb{X}$ converging to $(u,v,f) \in \mathbb{X}$, as $n\to\infty$. 
    We show that for any $(U,V,F) \in \mathbb{X}$, we have
    \begin{equation}\label{eq:DG1_cont}
        \| (G_1'(u_n,v_n,f_n) - G_1'(u,v,f)) (U,V,F)\|_{L^2(0,T;(H^{1}(\Omega))'} \to 0\quad\text{as}\quad n \to \infty.
    \end{equation}
    Indeed, for any $\varphi \in L^2(0,T;H^{1}(\Omega))$, 
    \begin{equation}\label{eq:DG1n-DG1}
        \begin{split}
            \langle (G_1'(u_n,v_n,f_n) - G_1'(u,v,f)) (U,V,F),\varphi\rangle = &- (f_n u_n \nabla V + f_n U \nabla v_n + F u_n \nabla v_n, \nabla \varphi)_{L^2(\Omega_T)} \\
            & + (f u \nabla V + f U \nabla v + F u \nabla v, \nabla \varphi)_{L^2(\Omega_T)} \\
            &+ (2 \mu u_n U,\varphi)_{L^2(\Omega_T)} - (2 \mu u U,\varphi)_{L^2(\Omega_T)}.
        \end{split}
    \end{equation}
    Let us focus on the term
    $$
        (F u_n \nabla v_n - F u \nabla v, \nabla \varphi)_{L^2(\Omega_T)} = (F(u_n-u)\nabla v_n,\nabla \varphi)_{L^2(\Omega_T)} + (F u (\nabla v_n - \nabla v),\nabla \varphi)_{L^2(\Omega_T)} =: J_1 + J_2.
    $$

    First we estimate $J_1$. Using that $F \in \mathcal{F}$ is bounded in $\Omega_T$, and using H\"older's inequality in space and time, we get
    \begin{equation}\label{eq:est_J1_1}
        \begin{split}
            |J_1| &\leq \int_{\Omega_T} |F| |u_n-u| |\nabla v_n| |\nabla \varphi| \ dxdt \\
            &\leq \|F\|_{L^\infty(\Omega_T)} 
            \int_0^T \|u_n(t)-u(t)\|_{L^4(\Omega)} \| \nabla v(t) \|_{L^4(\Omega)} \| \nabla \varphi(t) \|_{L^2(\Omega)} \ dt \\
            &\leq \|F\|_{L^\infty(\Omega_T)} \|u_n-u\|_{L^4(0,T;L^4(\Omega))} \| \nabla v \|_{L^4(\Omega_T)} \| \nabla \varphi \|_{L^2(\Omega_T)} \\
            &\leq C_1 \|u_n-u\|_{L^4(0,T;L^4(\Omega))}.
        \end{split}
    \end{equation}
    On the other hand, using the continuous embeddings 
    $$
        H^1(\Omega) \hookrightarrow L^4(\Omega) \;\mbox{ and }\; W^{s,p}(\Omega) \hookrightarrow L^{q}(\Omega),
    $$ 
    for any $1 \leq q < \infty$ if $s \, p=d$, with $d$ the spacial dimension, i.e. $\Omega \subset \mathbb R^d$, we have that
    \begin{equation}\label{eq:est_J1_2}
        \begin{split}
            |J_1| &\leq C \|F\|_{L^\infty(\Omega_T)} 
            \int_0^T \|u_n(t)-u(t)\|_{H^1(\Omega)} \| \nabla v(t) \|_{L^4(\Omega)} \| \nabla \varphi(t) \|_{L^2(\Omega)} \ dt \\
            &\leq C \|F\|_{L^\infty(\Omega_T)} 
            \| \nabla v \|_{L^\infty(0,T;L^4(\Omega))} 
            \int_0^T \|u_n(t)-u(t)\|_{H^1(\Omega)} \| \nabla \varphi(t) \|_{L^2(\Omega)} \ dt \\
            &\leq C \|F\|_{L^\infty(\Omega_T)}  \| v \|_{C(0,T;W_{\mathbf{n}}^{3/2,4}(\Omega))} \|u_n-u\|_{L^2(0,T;H^1(\Omega))}  \| \nabla \varphi \|_{L^2(\Omega_T)} \\
            &\leq C_2 \|u_n-u\|_{L^2(0,T;H^1(\Omega))}.
        \end{split}
    \end{equation}
    
    Thus, from \eqref{eq:est_J1_1} and \eqref{eq:est_J1_2}, we get  $|J_1| \leq C_3 \|(u_n-u,v_n-v,f_n-f)\|_{\mathbb{X}} \to 0$, as $n\to \infty$.

    For $J_2$, proceeding similarly, using Hölder's inequality and the continuous embeddings $H^1(\Omega) \hookrightarrow L^4(\Omega)$ and $L^4([0,T]) \hookrightarrow L^2([0,T])$, we obtain that
    \begin{equation}\label{eq:est_J2_1}
        \begin{split}
            |J_2| &\leq \int_{\Omega_T} |F| |u| |\nabla v_n - \nabla v| |\nabla \varphi| \ dxdt \\
            &\leq C \|F\|_{L^\infty(\Omega_T)} \|u_n\|_{L^\infty(0,T;L^4(\Omega))} \int_0^T \| v_n(t) - v(t)\|_{W^{2,4}(\Omega)} \| \nabla \varphi(t) \|_{L^2(\Omega)} \ dt\\
            &\leq C \|F\|_{L^\infty(\Omega_T)} \|u_n\|_{L^\infty(0,T;L^4(\Omega))} \| v_n - v\|_{L^2(0,T;W^{2,4}(\Omega))} \| \nabla \varphi \|_{L^2(\Omega_T)} \\
            &\leq C_1 \| v_n - v\|_{L^4(0,T;W^{2,4}(\Omega))}.
        \end{split}
    \end{equation}
    
    On the other hand, the continuous embedding $L^2([0,T]) \hookrightarrow L^1([0,T])$ gives
    \begin{equation}\label{eq:est_J2_2}
        \begin{split}
            |J_2| 
            &\leq C \|F\|_{L^\infty(\Omega_T)} \|u_n\|_{L^\infty(0,T;L^4(\Omega))} \sup_{t \in [0,T]} \| \nabla v_n(t) - \nabla v(t) \|_{L^4(\Omega)} \int_0^T \| \nabla \varphi(t) \|_{L^2(\Omega)} \ dt\\
            &\leq C \|F\|_{L^\infty(\Omega_T)} \|u_n\|_{L^\infty(0,T;L^4(\Omega))}\| \nabla v_n - \nabla v\|_{L^\infty(0,T;L^{4}(\Omega))} \| \nabla \varphi \|_{L^2(\Omega_T)} \\
            &\leq C_2 \| \nabla v_n - \nabla v\|_{L^\infty(0,T;L^4(\Omega))} \\
            &\leq C_2 \| v_n - v\|_{L^\infty(0,T;W_{\mathbf{n}}^{3/2,4}(\Omega))},
        \end{split}
    \end{equation}
    where in the last step we used as in \eqref{eq:est_J1_2} the continuous embedding $W_{\mathbf{n}}^{\frac{1}{2},4}(\Omega) \hookrightarrow L^{q}(\Omega)$, for any $1 \leq q < \infty$, when $\Omega \subset \mathbb R^2$.
    Thus, from \eqref{eq:est_J2_1} and \eqref{eq:est_J2_2}, we get
    $$
        |J_2| \leq C_3 \|(u_n-u,v_n-v,f_n-f)\|_{\mathbb{X}} \to 0 \quad \text{as } n \to \infty.
    $$
    
    The other terms in \eqref{eq:DG1n-DG1} are estimated analogously. 
    In particular, the estimates of the term
    $$
        (f_n u_n \nabla V, \nabla \varphi)_{L^2(\Omega_T)} - (f u \nabla V , \nabla \varphi)_{L^2(\Omega_T)}
    $$
    are also bounded by the norm $\|f_n -f\|_{L^4(\Omega_T)}$.
    This proves \eqref{eq:DG1_cont}.
    
    For $G_2$, we observe that the expression of $G_2'(s)[z]$ does not depend on $s$, thus $s \mapsto G_2'(s)$ is continuous. Together with \eqref{eq:DG1_cont}, this implies that $s \mapsto G'(s)$ is continuous and therefore $G$ is Fréchet differentiable at $s$.  
\end{proof}

\section{Regularity of elements in $\mathcal{S}_{ad}$}
\label{appendix_reg_point}

\begin{lemma}\label{lem:regular_point}
    If $\tilde s=(\tilde u, \tilde v, \tilde f) \in \mathcal{S}_{ad}$, then $\tilde s$ is a regular point. 
\end{lemma}

\begin{proof}
    Let $(g_u,g_v) \in \mathbb{Y}=
    L^2(0,T;(H^{1}(\Omega))')
    \times L^4(\Omega_T)$ be arbitrary.
    Recall that $\mathcal{C}(\tilde f) = \{ \theta(f-\tilde f): \ \theta \geq 0, \ f \in \mathcal{F} \}$ denotes the conical hull of $\tilde f \in \mathcal{F}$.
    The point $\tilde s=(\tilde u, \tilde v, \tilde f) \in \mathcal{S}_{ad}$ is a regular point if for any $(g_u,g_v) \in \mathbb{Y}$, there exists $z=(U,V,F) \in \hat W_4 \times \hat X_4 \times \mathcal{C}(\tilde f)$ such that $G'(\tilde s)[z] = (g_u,g_v)$.

    Since $0 \in \mathcal{C}(\tilde f)$, we can take $F=0$ and it suffices to find $(U,V) \in \hat W_4 \times \hat X_4$ solving
    \begin{equation}\label{eq:linear_U_V_new}
        \begin{cases}
            \langle \partial_t U,\varphi \rangle + ( D_u \nabla U -  \tilde f U \nabla \tilde v - \tilde f \tilde u \nabla V , \nabla \varphi)_{L^2(\Omega_T)} = (r U - 2 \mu U \tilde u,\varphi)_{L^2(\Omega_T)} + \langle g_u,\varphi \rangle,\\
            \partial_t V - D_v \Delta V + \alpha V = \beta U + g_v,  \qquad\qquad \mbox{ in } \Omega_T,
        \end{cases}    
    \end{equation}
    for all $\varphi \in L^4(0,T;W^{1,4}(\Omega))$, with the initial and boundary conditions included in the definition of $\hat W_4$ and $\hat X_4$. We recall that $\hat W_4$ and $\hat X_4$ are subsets of $W_4$ and $X_4$, respectively,  with specified initial and boundary conditions.
    
    \noindent\textbf{Step 1: Operators.} Fix an interval $I_0=(t_0,t_1) \subset (0,T)$ and denote $\Omega_{I_0} := \Omega \times I_0$ and $\hat W_4(I_0)$, $\hat X_4(I_0)$ as subspaces of $\hat W_4$, $\hat X_4$, respectively, where the temporal domain of the functions is restricted to $I_0$.
    
    For $F \in L^4(\Omega_{I_0})$ we define $V=\mathcal{T}_{I_0}(F) \in \hat X_4(I_0)$ as the unique solution of
    $$
        \partial_t V - D_v \Delta V + \alpha V = \beta F + g_v\quad \text{in}\quad \Omega_{I_0},
    $$
    with the prescribed boundary conditions and initial value at $t=t_0$. By the $L^4$-regularity result, $\mathcal{T}_{I_0}$ is linear and continuous. Since $V \in X_4$ and $\Omega \subset \R^2$, we have that \(\nabla V \in L^\infty(I_0;L^q(\Omega,\mathbb R^2))\), for any $1 \leq q < \infty$. 
    
    Moreover, by the argument used to derive \eqref{local-u-regularity},  for any $t_0 < \tau < t_1$ we have 
    $$
        \tilde u \in L^4(\tau,t_1;W^{1,4}(\Omega)) \subset L^4(\tau,t_1;L^\infty(\Omega)).
    $$
    We denote $I=[\tau,t_1]$. 
    The following estimates can be treated as in Lemma  \ref{lem:uniform_bounds} (Step 2), integrating on $(\tau,t)$, for $t \leq t_1$, and letting $\tau \downarrow t_0$. 
    Then, we have that $\tilde u \nabla V \in L^4(I; L^4(\Omega;\mathbb R^2))$ and then $\nabla\cdot (\tilde f \tilde u \nabla V) \in L^4(I;W^{-1,4}(\Omega))$.
    
    Similarly, we define $U=\mathcal{S}_I(V)$ as the unique solution of
    \begin{equation}\label{equationc}
        \langle \partial_t U, \varphi \rangle + ( D_u \nabla U - \tilde f U \nabla \tilde v, \nabla \varphi) + (2\mu \tilde u U - r U,\varphi) = (\tilde f \tilde u \nabla V, \nabla \varphi) + \langle g_u, \varphi \rangle\quad \text{in}\quad \Omega_I,
    \end{equation}
    for all $\varphi \in L^4(0,T;W^{1,4}(\Omega))$, 
    with the prescribed boundary conditions and initial value at $t=t_0$. 
    Since $\tilde s=(\tilde u,\tilde v,\tilde f)\in \mathcal S_{ad}$, it follows that
    \[
        \tilde f \nabla \tilde v \in L^\infty(I;L^q(\Omega;\mathbb R^2)), \qquad 1\le q<\infty,
    \]
    and
    \[
        \mu \tilde u - r \in L^\infty(I;L^4(\Omega)).
    \]
    Therefore, the lower-order terms belong to $L^\infty(I;L^p(\Omega))$ for a suitable $p$.
    By the abstract results on weak maximal $L^p$-regularity 
    (see \cite[Chap.~V, \S2.6]{Amann1995}), Equation \eqref{equationc} 
    admits a unique weak solution
    \[
        U \in L^p(I;W^{1,p}(\Omega)) \cap W^{1,p}(I;L^p(\Omega)),
    \]
    provided that $u_0\in L^p(\Omega)$.  Therefore, $U \in \hat W_4(I)$.
    Again, $\mathcal{S}_I$ is linear and continuous.

    Define the fixed-point map $\Phi_I$ by
    $$
        \Phi_I : L^4(\Omega_I) \to L^4(\Omega_I), \qquad \Phi_I(F) := \mathcal{S}_I(\mathcal{T}_I(F))).
    $$

    \noindent\textbf{Step 2: Contraction on a short time interval.} Let $F_1$, $F_2 \in L^4(\Omega_I)$ and set
    $$
        V_i = \mathcal{T}_I(F_i), \quad U_i = \mathcal{S}_I(V_i), \quad \delta F = F_1 - F_2, \quad \delta V = V_1 - V_2, \quad \delta U = U_1 - U_2.
    $$
    
    From standard energy estimates and using $L^4$-regularity, we show that 
    $$
        \|\Phi_I(F_1) - \Phi_I(F_2) \|_{L^4(\Omega_I)} = \| \delta U \|_{L^4(\Omega_I)} \leq \tilde C |I|^{\theta} e^{\tilde D |I|} \|\delta F\|_{L^4(\Omega_I)},
    $$
    for some constants $\theta > 0$, and $\tilde C, \tilde D > 0$. This allow us to chose $|I|$ small enough such that $\Phi_I$ is a contracting map.
    
    Indeed, $L^4$-regularity in the equation of $V$ gives
    \begin{equation}\label{eq:est_V}
        \|V_1-V_2\|_{L^4(\Omega_I)} + \|\nabla (V_1-V_2)\|_{L^4(\Omega_I)}   \leq C\|F_1 - F_2\|_{L^4(\Omega_I)}.    
    \end{equation}
    
    Now we use $L^4$-regularity for the $U$-equation, with source term of the form $-\nabla \cdot (\tilde f \tilde u \nabla (V_1-V_2))$. 
    We recall that $\tilde f \tilde u \in L^\infty(I;L^4(\Omega))$.
    As in the proof of the uniform \(L^\infty(0,T;L^4(\Omega))\) bound in Lemma \ref{lem:uniform_bounds} (Step 2), we test the equation satisfied by $\delta U$ against  $(\delta U)^3 \in L^4([\tau,T]\cap I;W^{1,4}(\Omega))$ for any $\tau>t_0$.
    We get 
    \begin{equation}
    \label{eq:esti_regular_U}
        \begin{split}
            &\frac{1}{4}\frac{d}{dt} \|\delta U(t)\|_{L^4(\Omega)}^4 + 3 D_u \int_\Omega |\delta U|^2 |\nabla \delta U|^2 dx  + 2\mu \int_\Omega \tilde u |\delta U|^4 dx\\
            &\leq 3\int_\Omega |\tilde f| |\delta U|^3 |\nabla \tilde v| |\nabla \delta U| dx + 3 \int_\Omega |\tilde f| |\delta U|^2 |\tilde u| |\nabla \delta V| |\nabla \delta U| dx + r \int_\Omega |\delta U|^4 dx =: I_1+I_2+I_3.  
        \end{split}
    \end{equation}

    The term $I_1$ in the right-hand side can be estimated following the steps that start in \eqref{eq:esti_u3}. 
    The second integral $I_2$ in the right-hand side of \eqref{eq:esti_regular_U} is estimated similarly. That is, we define $w=\delta U^2$, then $I_2 = \frac{3}{2} \|f\|_{L^\infty(\Omega)} \int_\Omega |\tilde u| |\nabla \delta V| w^{1/2} \nabla w \ dx$.
    Using H\"older inequality, Ladyzhenskaya inequality and Young's inequality, we obtain, for small constants $\epsilon, \varepsilon>0$, 
    $$
        I_2 \leq \|f\|_{L^\infty(\Omega)} \left( \epsilon \|w\|_{L^2(\Omega)}^2 + \varepsilon \|\nabla w\|_{L^2(\Omega)}^2 + C_{\epsilon,\varepsilon} \|\tilde u\|_{L^4(\Omega)}^4 \|\nabla \delta V\|_{L^4(\Omega)}^4 \right).
    $$
    
    We can absorb the term $\varepsilon \|f\|_{L^\infty(\Omega)} \|\nabla w\|_{L^2(\Omega)}^2$.
    Indeed, using the estimate of $I_1$ analogous to \eqref{eq:est_u_eps_L4}, and the equality $\|w\|_{L^2(\Omega)}^2 = \|\delta U \|_{L^4(\Omega)}^4$, then \eqref{eq:esti_regular_U} yields, for suitable $C_i>0$, $i=0,1,2$,
    \begin{equation*}
        \begin{split}
            \frac{d}{dt} \|\delta U(t)\|_{L^4(\Omega)}^4 + (1-\varepsilon)\int_\Omega |\nabla w|^2 dx  + \int_\Omega \tilde u |\delta U|^4 dx \leq & C_1 \|\delta U\|_{L^4(\Omega)}^4 + C_2 \|\tilde u\|_{L^4(\Omega)}^4 \|\nabla \delta V\|_{L^4(\Omega)}^4 \\
            & + C_0 \Bigl(1+ \|\tilde f\|_{L^\infty(\Omega_T)}^4 \|\nabla \delta V\|_{L^4(\Omega)}^4 \Bigr) \|\delta U\|_{L^4(\Omega)}^4.
        \end{split}
    \end{equation*}

    Discarding the positive terms, and applying Gr\"onwall's inequality on $[\tau,t]$ yields
    \begin{equation}
    \label{eq:esti_deltaU_L4}
        \begin{split}
            \|\delta U (t)\|_{L^4(\Omega)}^4
            \leq
            & \left( \|\delta U(\tau)\|_{L^4(\Omega)}^4 + C_2 \int_\tau^t \|\tilde u(s)\|_{L^4(\Omega)}^4 \|\nabla \delta V(s)\|_{L^4(\Omega)}^4 ds \right) \times \\
            &\exp\left(
            C_3\int_\tau^t
            \left( 1+\|\nabla \delta V (s)\|_{L^4(\Omega)}^4 \right)\,ds
            \right).       
        \end{split}
    \end{equation}
    Since $\delta U \in C([0,T];L^4(\Omega))$ and
    $\delta U(0)=u_0-u_0=0$, we have
    \[
        \|\delta U(\tau)\|_{L^4(\Omega)}^4
        \longrightarrow 0
        \qquad\text{as }\tau\downarrow t_0.
    \]
    Moreover, $\tilde u \in C([0,T];L^4(\Omega))$, and
    $\int_0^T \|\delta V(s)\|_{L^4(\Omega)}^4 ds < \infty$, 
    then,
    $$
        \int_\tau^t \sup_{t \in I} \|\tilde u(s)\|_{L^4(\Omega)}^4 \|\nabla \delta V(s)\|_{L^4(\Omega)}^4 ds \to \sup_{t \in I_0} \| \tilde u(s)\|_{L^4(\Omega)}^4 \int_{t_0}^t \|\delta V(s)\|_{L^4(\Omega)}^4 ds  
        \qquad\text{as }\tau\downarrow t_0.
    $$

    Thus, letting $\tau\downarrow t_0$ in \eqref{eq:esti_deltaU_L4}, we obtain
    \[
        \|\delta U (t)\|_{L^4(\Omega)}^4
        \leq
        \left( C_2 \sup_{t \in I_0} \| \tilde u(s)\|_{L^4(\Omega)}^4 \int_{t_0}^t \|\delta V(s)\|_{L^4(\Omega)}^4 ds \right)
        \exp\left(
        C_3\int_{t_0}^t
        \left( 1+\|\nabla \delta V (s)\|_{L^4(\Omega)}^4 \right)\,ds
        \right).
    \]

    On the other hand,
    $$
        \|\delta U\|_{L^4(I_0;L^4(\Omega))}^4 \leq |I_0| \sup_{t \in I_0} \|\delta U(t)\|_{L^4(\Omega)}^4,
    $$
    which, combined with the previous bound, yields
    \begin{equation*}
        \begin{split}
            \|\delta U\|_{L^4(I_0;L^4(\Omega))} &\leq C|I_0|^{1/4} \left[\left( \sup_{t \in I_0} \| \tilde u(s)\|_{L^4(\Omega)}^4 \int_{t_0}^{t_1} \|\delta V(s)\|_{L^4(\Omega)}^4 ds \right)
            \exp\left(
            C_3\int_{t_0}^{t_1}
            \left( 1+\|\nabla \delta V (s)\|_{L^4(\Omega)}^4 \right)\,ds
            \right) \right]^{1/4}\\
            &\leq C|I_0|^{1/4}  \|\tilde u\|_{C(I_0;L^4(\Omega))} \|\nabla\delta V\|_{L^4(I_0;L^4(\Omega))} 
            \exp\left(
            \frac{C_3}{4}\int_{t_0}^{t_1}
            \left( 1+\|\nabla \delta V (s)\|_{L^4(\Omega)}^4 \right)\,ds \right).
        \end{split}
    \end{equation*}
    Thus, using \eqref{eq:est_V} and denoting $D = \int_{0}^{T}
    \|\nabla \delta V (s)\|_{L^4(\Omega)}^4 \,ds$ we get
    \begin{equation*}
        \begin{split}
            \|\delta U\|_{L^4(I_0;L^4(\Omega))} &\leq C|I_0|^{1/4}  \|\tilde u\|_{C(I_0;L^4(\Omega))} \|\delta F\|_{L^4(I_0;L^4(\Omega))} 
            \exp\left(
            \frac{C_3}{4}
            \left( |I_0|+D \right)\right) \leq \tilde C |I_0|^{1/4} e^{
            \frac{C_3}{4}|I_0|}  \|\delta F\|_{L^4(\Omega_{I_0})}.
        \end{split}
    \end{equation*}
    
    The constant $\tilde C$ is independent of $t_0$, $t_1$ and depends only on the coefficients and $\tilde s$.
    
    Choose $\tau_0>0$ such that $\tilde C\tau_0 < 1$. Then, for any interval $I_0$ with $|I_0|^{1/4} e^{
    \frac{C_3}{4}|I_0|} \leq \tau_0$, $\Phi_{I_0}$ is a contraction on $L^4(\Omega_{I_0})$ and has a unique fixed point $F \in L^4(\Omega_{I_0})$ by Banach's fixed point theorem. Setting, $U=F$ and $V=\mathcal{T}_{I_0}(F)$ yields a solution $(U,V) \in \hat W_4(I_0) \times \hat X_4(I_0)$ of \eqref{eq:linear_U_V_new} on $I_0$.

    \noindent\textbf{Step 3: Continuation up to time $T$.} Let $N \in \mathbb{N}$ be such that $T/N \leq \tau_0$ and define $t_k := kT/N$. Starting from $t_0=0$, we construct on each interval $I_k = (t_{k-1},t_k)$ a solution $(U,V)$ by the previous step, using as initial data at $t_{k-1}$ the values $(U(t_{k-1}),V(t_{k-1}))$ obtained from the preceding interval. Uniqueness on each subinterval ensures that these local solutions match at the nodes $t_k$, thus producing a global solution $(U,V) \in \hat W_4 \times \hat X_4$ on $(0,T)$.

    Finally, taking $F=0 \in \mathcal{C}(\tilde f)$, we have found $(U,V,F) \in \hat W_4 \times \hat X_4 \times \mathcal{C}(\tilde f)$ such that $G'(\tilde s)[U,V,F]=(g_u,g_v)$. Hence, $\tilde s$ is a regular point.
\end{proof}

\bibliographystyle{plain}
 \bibliography{biblio}
 \vfill
\end{document}